\documentclass[11pt, a4paper, oneside, reqno]{amsart}
\allowdisplaybreaks
\usepackage[margin=1in]{geometry}
\usepackage{amsfonts}
\usepackage{amsmath}
\usepackage{amssymb,bbm,bm}
\usepackage{amsthm}
\usepackage{hyperref}
\usepackage{mathtools}
\usepackage{tikz}
\usepackage{tikz-cd}
\usepackage{silence} 
\usepackage{enumerate}
\usepackage{stmaryrd} 
\usepackage{mathabx}  
\usepackage{xcolor}
\usepackage[english]{babel}
\makeatletter
\usepackage{svg}
\renewcommand{\tocsection}[3]{%
	\indentlabel{\@ifnotempty{#2}{\bfseries\ignorespaces#1 #2.\,\,}}\bfseries#3}
\renewcommand{\tocsubsection}[3]{%
	\indentlabel{\@ifnotempty{#2}{\ignorespaces#1 #2\quad}}#3}
\renewcommand{\tocsubsubsection}[3]{%
	\quad\quad\quad\indentlabel{\@ifnotempty{#2}{\ignorespaces#1 #2\quad}}#3}

\def\1{\mathbf{1}}

\newtheorem{theorem}{Theorem}[section]
\newtheorem{proposition}[theorem]{Proposition}
\newtheorem*{proposition*}{Proposition}
\newtheorem{corollary}[theorem]{Corollary}
\newtheorem*{theorem*}{Theorem}
\newtheorem{lemma}[theorem]{Lemma}

\theoremstyle{remark}
\newtheorem{remark}[theorem]{Remark}
\newtheorem{example}[theorem]{Example}

\theoremstyle{definition}
\newtheorem{definition}[theorem]{Definition}

\numberwithin{equation}{section}

\definecolor{amber}{rgb}{1.0, 0.49, 0.0}
\definecolor{ao}{rgb}{0.0, 0.5, 0.0}

\newcommand{\Varphi}{{
\varphi}}
\renewcommand{\Psi}{{
\psi}}
\newcommand{\OmegA}{{
\omega}}

\newcommand{\cck}{\kappa^{\,\OmegA\,|\,\Varphi,\,\Psi}}

\newcommand{\nc}{{\rm NC}}
\newcommand{\cync}{{\rm CyNC}}

\title{Combinatorics of cyclic conditional freeness}

\author{Octavio Arizmendi${}^{\dagger}$}
\thanks{O.A. gratefully acknowledges financial support by the grants Conacyt A1-S-9764 and SFB TRR 195.}
\address{${}^{\dagger}$ Centro de Investigación en Matemáticas (CIMAT), A.C., Jalisco S/N, Col. Valenciana CP: 36023 Guanajuato, Gto, México}
\email{octavius@cimat.mx}

\author[G. C\'ebron]{Guillaume C\'ebron${}^{\star}$}
\thanks{G.C. is supported by the Project MESA (ANR-18-CE40-006) and by the Project STARS (ANR-20-CE40-0008) of the French National Research Agency (ANR)}
\author[N. Gilliers]{Nicolas Gilliers${}^{\star, \diamond}$}\thanks{N.G. is supported by the Project STARS (ANR-20-CE40-0008) of the French National Research Agency (ANR)}
\address{${}^{\star}$ Institut de Mathématiques de Toulouse; UMR5219; Université de Toulouse; CNRS; UPS, F-31062 Toulouse, France}
\email{guillaume.cebron@math.univ-toulouse.fr, nicolas.gilliers@math.univ-toulouse.fr}

\address{${}^{\diamond}$  Department of Mathematics, New York University Abu Dhabi, Saadiyat Marina District, Abu Dhabi, United Arab Emirates}
\email{nag9000@nyu.edu}

\date{\today}
\begin{document}
\begin{abstract}
	This work investigates the combinatorial structures underlying cyclic conditional freeness and introduces cumulants that serve to linearize the cyclic conditional additive convolution. In the process, we establish the notion of "cyclic freeness," demonstrating its equivalence to infinitesimal freeness in the presence of tracial states. Furthermore, we show that cyclic conditional freeness can be reduced to cyclic freeness through a multivariate extension of the inverse Markov-Krein transform.
\end{abstract}
\maketitle
\tableofcontents

\section{Introduction}
Non-commutative probability (NCP) is an area of Operator Algebra Theory that approaches the subject from the point of view of probability, see \cite{speicher2020lecture}. In this sense, one considers bounded operators in a von Neumann algebra or unbounded operators affiliated to a von Neumann algebra acting on a pointed Hilbert space as random variables. One associates a distribution to these random variables depending on a linear functional (the state) on the algebra generated by these operators.

One of the essential concepts of probability is \emph{independence} since it allows the building of families of random variables with a natural and manageable joint distribution. Many theories, such as stochastic processes, particularly Markov processes, are based on independence.
A surprising feature of NCP is the possibility of \emph{several (abstract) independences} in opposition to classical probability where independence between classical random variables is somewhat unambiguous: among them one finds \emph{free independence}~\cite{voiculescu1986addition}, \emph{Boolean independence}~\cite{speicher1993boolean} and \emph{monotone independence}~\cite{muraki2001monotonic}.

NCP has provided fresh tools and a deeper understanding across various domains, including physics problems via Random Matrix Theory (RMT). RMT has broadly benefited from Voiculescu's free probability \cite{voiculescu1986addition} since many matrix ensembles in the high-dimensional limit display asymptotic freeness \cite{speicher2014free,voiculescu1991limit}. Beyond the original applications of NCP to classical physics problems such as energy spectra of heavy nuclei~\cite{zbMATH03113773} or 2D quantum gravity~\cite{di19952d}, it has also recently emerged within the context of quantum chaos, with all of its implications not yet fully elucidated. Let us mention the work of Hruza and Bernard~\cite{hruza2023coherent} about noisy many-body quantum systems and the work of Pappalardi, Foini and Kurchan~\cite{pappalardi2022eigenstate} about eigenstate thermalisation.

A critical tool from NCP which appears fundamental in these new developments is the use of \emph{cumulants}. In free probability, free cumulants were introduced by Speicher \cite{nica2006lectures} and have since become a fundamental tool for answering theoretical questions and deriving applications of free probability. For example, the essential role of cumulants from NCP in the two physics works cited is highlighted in the preprints~\cite {bernard2024structured} by Hruza and Bernard and~\cite {fava2023designs} by Fava, Kurchan, and Pappalardi. Let us also mention the recent use of the \emph{free cumulants} by Wang~\cite{Wang_2023} to resolve the spectrum of the radiation density matrix in the model of an evaporating black hole due to Penington-Shenker-Stanford-Yang.

Returning to the notion of non-commutative independence, a given collection of random variables can have \emph{several joint distributions}, depending on the functionals considered for computing expectations. These various non-commutative distributions can interact coherently. One then speaks of independence between \emph{tuples of non-commutative distributions}: two-state independences, three-state independences, and so on. Two typical examples of two-state independences relevant for the present work are the conditional freeness \cite{bozejko1996convolution,bozejko1998new} of Bożejko, Leinert, and Speicher, and the infinitesimal freeness \cite{fevrier2010infinitesimal} of Février and Nica.

More recently, other independences besides free independence have appeared in connection with the fine-grained limiting behavior of spectra of random matrices. Examples include the already mentioned infinitesimal freeness \cite{cebron2022freeness,curran2011asymptotic,shlyakhtenko2018free}.
\emph{Cyclic-monotone independence \cite{collins2018free}} was specifically designed to study random matrix models with purely discrete eigenvalues in the high-dimensional limit. In~\cite{arizmendi2022cyclic}, the first author, Hasebe and Lehner introduced an operatorial model to cyclic-monotone independence, which they leverage to introduce the new concept of cyclic-Boolean independence.
In \cite{cebron2022asymptotic}, the two last authors proposed a random matrix model, the \emph{Vortex Model}, displaying \emph{conditional freeness} in the high dimensional limit. They computed the fluctuations of the empirical spectral values distribution in this convergence with the help of the appropriate independence original to their work, which they named \emph{cyclic conditional freeness}, of prime interest for the present work.

Studying the different independences and providing the corresponding set of cumulants for each one is fundamental from a theoretical standpoint. This is also important for applications in physics. We tackle this in the present work: we define the cumulants adapted for the cyclic conditional freeness.
In defining these cumulants, we first reduce cyclic conditional freeness (an independence notion involving three linear functionals, or "states") to cyclic freeness (an independence notion involving two states).
Surprisingly, this reduction utilizes a multivariate extension of the celebrated (inverse) Markov-Krein transform. Cyclic freeness coincides with infinitesimal freeness when the two states are tracial linear functionals. We leverage this observation to build upon the intuition gained in the work \cite{fevrier2010infinitesimal} of Février and Nica about infinitesimally free cumulants, the cyclic free  cumulants.
Finally, we establish the basic properties of these cumulants, and we provide subordination relations between the appropriate transforms, thus providing an effective way of computing cyclic conditional additive and multiplicative convolutions.

Let us mention two additional motivations for the introduction of cyclic freeness. Firstly, this new independence extends the connection between conditional and infinitesimal freeness found by Février, Mastnak, Nica and Szpojankowski \cite{fevrier2019construction} to non-tracial states. Secondly, cyclic conditional freeness and cyclic freeness, cyclic-monotone and cyclic-Boolean independences share relations reminiscent of the ones holding between conditional freeness and the triptych of the monotone, Boolean and free independences \cite{franz2005multiplicative, bozejko1996convolution} (see Fig. \ref{fig:Figureone}).

\begin{figure}[!h]
	\begin{tikzcd}[ampersand replacement=\&]
		\& {\textrm{cyclic conditional freeness}} \\
		{\textrm{cyclic Boolean}} \& {\textrm{cyclic freeness}} \& {\textrm{cyclic monotone}} \\
		{\textrm{Boolean independence}} \& {\textrm{freeness}} \& {\textrm{monotone independence}} \\
		\& {\textrm{conditional freeness}}
		\arrow[from=1-2, to=2-2]s
		\arrow[from=1-2, to=2-3]
		\arrow[from=1-2, to=2-1]
		\arrow[color = blue, from=2-3, to=3-3]
		\arrow[color=red, bend right = 60, from=3-3, to=2-3]
		\arrow[color = blue, from=2-1, to=3-1]
		\arrow[color = red, bend right = 60, from=3-1, to=2-1]
		\arrow[from=4-2, to=3-1]
		\arrow[from=4-2, to=3-3]
		\arrow[from=4-2, to=3-2]
		\arrow[color = blue, bend right=60,from=1-2, to=4-2]
		\arrow[color = blue, from=2-2, to=3-2]
		\arrow[color=red, bend right = 60, from=3-2, to=2-2]
		\arrow[color=red, bend right = 60, from=4-2, to=1-2]
		\arrow[color=gray, bend left = 50, from=4-2, to=2-2]
	\end{tikzcd}
	\caption{Each {\color{red}red} arrows means that we use the multivariate Markov-Krein transform (we call those arrows cyclic companions in this text). Each black arrow means the target independence is obtained under further assumptions on the linear functionals and the independence at the source. The {\color{gray} grey} arrow represents an extension of a result by Février, Mastnak, Nica and Szpojankowski. The {\color{blue} blue} arrows proceed from considerations as in Proposition 2.16 in \cite{cebron2022freeness} stating the connection between the monotone and cyclic monotone independences.
		\label{fig:Figureone}}
\end{figure}
\subsection{Main Contributions}
We introduce in the present work a new non-commutative independence which we call \emph{cyclic freeness} between random variables in a non-commutative probability space endowed with two, possibly non-unital, linear functionals of which one is tracial. We compare cyclic freeness with other notions of independence, such as the infinitesimal freeness. The definition goes as follows.

Recall that a tracial linear functional $\OmegA\colon\mathcal{A}\to\mathbb{C}$ over an algebra $\mathcal{A}$ satisfies the property $\OmegA(ab) = \OmegA(ba)$ for all $a,b\in\mathcal{A}$.
In what follows, a non-commutative probability space will denote a pair of an algebra $\mathcal{A}$ and a linear functional $\Psi:\mathcal{A}\to\mathbb{C}$.
Besides, the symbol $\Varphi$ and $\Psi$ will always be used to denote a unital linear functional on $\mathcal{A}$ and we will use the symbol $\OmegA$ for a tracial, possibly non-unital, linear functional.

\begin{definition}[Cyclic freeness] \label{def:cyclicfreeness}
	Two unital subalgebras $\mathcal{A}_1 \subset \mathcal{A}$ and $\mathcal{A}_2 \subset\mathcal{A}$ are called \emph{cyclically free} relatively to the pair $(\Psi,\OmegA)$ whenever they are free with respect to $\Psi$ and, for any family $a_1,\ldots,a_n$ of variables in $\mathcal{A}_1\cup \mathcal{A}_2$, $a_i \in \mathcal{A}_{j_i}$, \emph{cyclically alternating}  ($n\geq 2$), and \emph{centered} with respect to $\Psi$:
	$$
		j_{i}\neq j_{i+1},\text{ for all } i=1,\ldots,n-1\, j_{1} \neq j_{n}
	$$
	and
	$$
		\Psi(a_i) = 0 \text{ for all } i=1,\ldots,n,
	$$
	it holds that
	$$
		\OmegA(a_1\dots a_n) =0.
	$$
\end{definition}
The above definition admits the following equivalent characterization: the two subalgebras $\mathcal{A}_1$ and $\mathcal{A}_2$ are cyclically free if for any family $a_1,\ldots,a_n$, $n\geq 1$ of $\mathcal{A}_1\cup \mathcal{A}_2$ which is \emph{alternating} (not necessarily cyclically alternating as in Definition \ref{def:cyclicfreeness}, we only impose $j_i\neq j_{i+1}$, $i=1,\ldots,n-1$) and centered with respect to the linear functional $\Psi$,
\begin{align}
	\Psi(a_1\dots a_n)   & = 0,   \label{def:freenessun}                                                                                                   \\
	\OmegA(a_1\dots a_n) & = \sum_{i=1}^n \Psi(a_{i+1}\dots a_n \OmegA(a_i) a_1\dots a_{i-1} ) \label{def:cyclicfreenesscharacterizationdeuxequationdeux}.
\end{align}
By inserting the moments conditions characterizing freeness of $\mathcal{A}_1$ and $\mathcal{A}_2$ \cite[Lemma 5.8]{nica2006lectures} in the second equation \eqref{def:cyclicfreenesscharacterizationdeuxequationdeux}, cyclic freeness is equivalent to, under the same conditions prevalent to equations \eqref{def:cyclicfreenesscharacterizationdeuxequationdeux} and \eqref{def:freenessun},
\begin{align*}
	 & \Psi(a_1\dots a_n)=0, \\
	 & \OmegA(a_1\dots a_n)=
	\left\{
	\begin{array}{ll}
		\Psi(a_na_1) \dots \Psi(a_{\frac{n+3}{2}}a_{\frac{n-1}{2}}) \OmegA(a_{\frac{n+1}{2}}) & n\text{ is odd and }j_1=j_{n}, \ldots,j_{\frac{n-1}{2}}=j_{\frac{n+3}{2}}, \\
		0                                                                                     & {\rm otherwise.}
	\end{array}
	\right.
\end{align*}
Notice the inversions of the indexes in the products $a_na_1,a_{n-1}a_2,\ldots$ in comparison to the characterization of infinitesimal freeness given in \cite[Definition 1.1]{fevrier2010infinitesimal}.

\begin{example} Let $a \in \mathcal{A}_1$, $b\in \mathcal{A}_2$, $\mathcal{A}_1$ with $\mathcal{A}_2$ two cyclically free subalgebras,
	$$
		\OmegA(aba)=\OmegA(ba^{2}) = -\OmegA (1_{\mathcal{A}}) \Psi(b) \Psi(a^2) + \OmegA(b)\Psi(a^2) + \Psi(b)\OmegA(a^{2}).
	$$
\end{example}
\begin{remark}
	\label{rk:cyfequalif}
	If $\Psi$ is tracial and $\OmegA(1_{\mathcal{A}})=0$, the above moments conditions for cyclic freeness are equivalent to the moments conditions of infinitesimal freeness, again see \cite{fevrier2010infinitesimal}. For the coincidence between infinitesimal freeness and cyclic freeness to hold, it is sufficient to require that the restrictions of $\Psi$ to each subalgebra $\mathcal{A}_1$ and $\mathcal{A}_2$ are tracial: in that case, too, cyclic freeness between $\mathcal{A}_1$ and $\mathcal{A}_2$ is equivalent to infinitesimal freeness. This holds in particular when $\mathcal{A}_1$ and $\mathcal{A}_2$ are, each, generated by one variable.
\end{remark}

We refer the reader to \cite{muraki2001monotonic, muraki2003thefive, speicher1993boolean, speicher2020lecture} for an account on the five natural independences, which we refer to as the Muraki's five : Voiculescu's freeness, boolean independence, (anti)-monotone independence and finally the classical stochastic independence.

We denote by $T(\mathcal{A})$ the tensor algebra over $\mathcal{A}$,
$$
	T(\mathcal{A})=\mathbb{C}1\oplus \bigoplus_{n \geq 1}\mathcal{A}^{\otimes n},
$$
and by $\bar{T}(\mathcal{A})$ the augmentation ideal of $T(\mathcal{A})$, given by
$
	\bar{T}(\mathcal{A})=\bigoplus_{n \geq 1}\mathcal{A}^{\otimes n}.
$
Given a subset $S\subset \mathcal{A}$, we denote by $T(S)$ (resp. $\bar{T}(S)$) the algebra generated by $S$ in $T(\mathcal{A})$ (resp. the intersection of $T(S)$ with the augmentation ideal $\bar{T}(\mathcal{A})$).
Recall that a linear functional $f\colon \mathcal{A}\to\mathcal{\mathbb{C}}$ yields another linear functional on $T(\mathcal{A})$ (we use the same symbol for the two),
$$
	f(a_1 \otimes \dots \otimes a_n)=f(a_1\cdots a_n),\quad f(1)=f(1_{\mathcal{A}}),\quad a_1,\ldots,a_n \in \mathcal{A}.
$$
(we take the product in $\mathcal{A}$ of the $a_i$'s on the right-hand side).


Given a linear function $\Psi\colon\mathcal{A}\to \mathbb{C}$, we define a new linear functional $[\Psi]\colon T(\mathcal{A})\to \mathbb{C}$ over the tensor algebra $\mathcal{A}$ as follows:
\begin{align}
	\label{eqn:definsoulcy}
	 & [\Psi](a_1,\dots, a_n)=\sum_{\substack{1\leq \ell \leq n \\ 1\leq i_1< \dots < i_{\ell} \leq n}} \prod_{k=1}^{\ell}\beta^{\Psi}(a_{i_k},\dots, a_{i_{k+1}-1}),
	 & [\Psi](1)=1.
\end{align}
where the $\beta^{\Psi}$ are the \emph{Boolean cumulants} associated to $\Psi$ and where we use the cyclic convention $i_{\ell+1}=i_1$. In Section \ref{sec:cyclicBoolean}, we interpret the formula \eqref{eqn:definsoulcy} as a \emph{cyclic Boolean moment-cumulant formula}, prescribing the cyclic Boolean cumulants of $[\Psi]$ as being equal to the cyclic averaging of the ordinary Boolean cumulants of $\Psi$.
For example, with $a_1,a_2,a_3 \in \mathcal{A}$:
\begin{align*}
	[\Varphi](a_1,a_2,a_3) & =\beta^{\Varphi}(a_1, a_2 , a_3)+\beta^{\Varphi}(a_3, a_1 , a_2)+\beta^{\Varphi}(a_2, a_3 , a_1)                                                          \\
	                       & \hspace{2cm}+ \beta^{\Varphi}(a_1, a_2)\beta^{\Varphi}(a_3)+\beta^{\Varphi}(a_2, a_3)\beta^{\Varphi}(a_1) + \beta^{\Varphi}(a_3, a_1)\beta^{\Varphi}(a_2) \\
	                       & \hspace{2cm}+\beta^{\Varphi}(a_1)\beta^{\Varphi}(a_2)\beta^{\Varphi}(a_3).
\end{align*}
We now give a more geometrical point of view about the functional $[\Varphi]$. Suppose that $(\mathcal{A},\Varphi)$ is a $W^*$-probability space. Recall that this means, in particular, that $\mathcal{A}$ is a von Neumann algebra and $\Varphi$ is positive and faithful. The Gelfand-Naimark-Segal construction applied to $\mathcal{A}$ is faithful and $\mathcal{A}$ can thus be identified with a subalgebra of $B(L^{2}(\mathcal{A},\Varphi))$, the algebra of bounded linear functional acting on $L^{2}(\mathcal{A},\Varphi)$. In this representation $\Varphi(a)=\langle a\Omega,\Omega\rangle_H$ for some cyclic vector $\Omega \in L^2(\mathcal{A})$. Set $H=L^{2}(\mathcal{A})$.
We denote by $p:H\to H$ the projection onto $\mathbb{C}\Omega$ and set $q=({\rm id}_H-p)$ the projection onto the hyperplane orthogonal to $\mathbb{C}\Omega$.

In the article \cite{kerov1997interlacing} of Kerov, it is proven that if the probability distribution of a random variable $A=A^*$ is $\mu$, then the inverse Markov-Krein transform of $\mu$ has its $n^{th}$ moment equal to:
$${\rm Tr}_{H}\Big[A^n-(q Aq)^n\Big]$$
Notice that the trace is well-defined because for any polynomial $Q\in\mathbb{C}[X]$, $Q(A)-Q(qAq)$ is a trace-class operator (it has finite rank). This inverse Markov-Krein transform is uniquely determined by $\mu$. The natural multivariate version of the \emph{inverse Markov-Krein transform} of $\Varphi$ is thus the distribution given, for any $a_1,\ldots,a_n\in \mathcal{A}$, by
\begin{align}
	\label{eqn:MKtransform}
	a_1\otimes \ldots \otimes a_n \mapsto {\rm Tr}_H\Big[a_1 \cdots  a_n-q a_1q\cdot q a_2 q \ldots q a_n q\Big]
\end{align}
(here again the trace is well-defined).

\begin{proposition}\label{prop:MKt}The multivariate inverse Markov-Krein transform of $\Varphi$, defined by \ref{eqn:MKtransform}
	is equal to the linear form $[\Varphi]$.
\end{proposition}

\begin{remark}
	In the particular case where $\Varphi$ is tracial, this definition of the multivariate inverse Markov-Krein transform coincides with the one proposed independently by Fujie and Hasebe in~\cite{fujie2023free}.
\end{remark}
\begin{proof}Let $a_1,\ldots,a_n \in B(\mathcal{H})$.Inserting ${\rm id}_H=p+q$ in the equation
	$$\Varphi(a_1\cdots a_n)={\rm Tr}_H\Big[p a_1 \cdot {\rm id}_H\cdot a_2\cdot {\rm id}_H\cdots {\rm id}_H\cdot a_np\Big],$$ one infers that the Boolean cumulants are in fact given by
	$$
		\beta^\Varphi(a_1,\ldots, a_n)= {\rm Tr}_H\Big[p a_1 \cdot q\cdot a_2\cdot q\cdots q\cdot a_np\Big].$$
	Now,
	\begin{align*}
		{\rm Tr}_H\Big[a_1 & \cdots  a_n-q a_1q\cdot q a_2 q \ldots q a_n q\Big]=
		{\rm Tr}_H\Big[pa_1 \cdots  a_np + qa_1 \cdots  a_nq -q a_1q\cdot q a_2 q \ldots q a_n q\Big] \\
		                   & = \!\!\!\!\!\!\!\!\sum_{\substack{\ell \geq 1                            \\ 1 \leq i(1)<\cdots <i(\ell)=n}}\!\!\!\!\prod_{k=1}^\ell\beta^{\Varphi}(a_{i(k)},a_{i(k)+1}\ldots,a_{i(k+1)-1}) \\
		                   & \hspace{3cm}+\!\!\!\!\!\!\!\!\sum_{\substack{\ell \geq 1                 \\ 1 \leq i(1)<\cdots <i(\ell)< n}}\!\!\!\!\prod_{k=1}^\ell\beta^{\Varphi}(a_{i(k)},a_{i(k)+1}\ldots,a_{i(k+1)-1})
	\end{align*}
	where the first sum corresponds to the term ${\rm Tr}_{H}[pa_1 \cdots  a_np]$, this is a sum over interval partitions with at least one block. The second sum corresponds to the term ${\rm Tr}_{H}[qa_1 \cdots  a_nq -q a_1q\cdot q a_2 q \ldots q a_n q]$, this is a sum over interval partitions with at least \emph{two blocks}. Grouping these two sums yields the formula \eqref{eqn:definsoulcy} defining $[\Varphi]$.
\end{proof}

Now we present the consequences of Proposition \ref{prop:MKt}, which clarify how cyclic conditional freeness connects with other independences. We start with recalling the relevant definitions.

\begin{definition}[Conditional freeness \cite{bozejko1996convolution}]
	A family of subalgebras $\mathcal{A}_{i}, i\in I$ are said \emph{conditionally free} relatively to the pair $(\Psi,\Varphi)$ if for any alternating sequence $a_1,\ldots,a_n$ with $a_k\in \mathcal{A}_{i_k}$, $i_1\neq i_2\neq \cdots \neq i_n$ and centered with respect to $\Psi$,
	one has
	\[
		\Varphi(a_1\cdots a_n) = 0.
	\]
\end{definition}


\begin{definition}[Cyclic conditional freeness \cite{cebron2022asymptotic}]
	\label{def:cycfreeness}
	A family of subalgebras  $\mathcal{A}_i\subset \mathcal{A}:i\in I$ is \emph{cyclically conditionally free} relatively to the triple $\Psi,\Varphi,\OmegA$ if
	\begin{enumerate}
		\item the family $(\mathcal{A}_i\subset \mathcal{A}:i\in I)$ is conditionally free relatively to $(\Psi,\Varphi)$,
		\item and, for any sequence $ (a_1,\ldots,a_n)$ of $\bigcup_{i\in I}\mathcal{A}_i$ which is \emph{cyclically alternating} and centered with respect to $\Psi$,
		      \begin{equation}
			      \label{eqn:cdnw}
			      \OmegA(a_1\cdots a_n) = \Varphi(a_1) \cdots \Varphi(a_n).
		      \end{equation}
	\end{enumerate}
\end{definition}
\begin{remark}
	When $\Varphi = \Psi$, cyclic conditional freeness reduces to cyclic freeness, see Definition \ref{def:cycfreeness}.
\end{remark}

\begin{theorem}\label{th:fromfreetocyclic}
	Let $\mathcal{A}_i \subset \mathcal{A}, i \in I$ unital subalgebras of $\mathcal{A}$, \emph{conditionally free} relatively to $\Psi,\Varphi$. Then, the subalgebras $T(\mathcal{A}_i),i\in I$ of the tensor algebra $T(\mathcal{A})$ are \emph{cyclically conditionally free} relatively to the triple $(\Psi,\Varphi,[\Varphi])$.
\end{theorem}
\begin{proof}
	See Section \ref{sec:conditionaltocyclicconditional}.
\end{proof}
This theorem has several corollaries, each corresponding to a blue arrow in Fig. \ref{fig:Figureone}.
\begin{corollary}\label{lemma:fromfreetocyclic}
	Let $\mathcal{A}_i \subset\mathcal{A},i \in I$ be free unital subalgebras relatively to $\Psi$. Then, the $T(\mathcal{A}_i),\,i\in I$ are cyclically free subalgebras relatively to $(\Psi,\alpha [\Psi])$ for any $\alpha\in\mathbb{C}$.
\end{corollary}
\begin{proof}
	This is the particular case of the previous Theorem \ref{th:fromfreetocyclic} for which $\Varphi=\Psi$.
\end{proof}
\begin{corollary}\label{lemma:fromBooleantocyclic}
	Let $\mathcal{A}_i\subset\mathcal{A}, i \in I$ be \emph{Boolean independent} subalgebras relatively to $\Psi$. Then, $\bar{T}(\mathcal{A}_i), i \in I$ are \emph{cyclically Boolean independent} subalgebras relatively to $(\Varphi,[\Varphi])$.
\end{corollary}
\begin{proof}
	This is the particular case where $\Psi$ is vanishing on the $\mathcal{A}_i$.
\end{proof}
\begin{corollary}\label{lemma:frommonotonetocyclic}
	Let $\mathcal{A}_i,~i \in \{1,\ldots,n\}$ be subalgebras which are \emph{monotone independent} with respect to $\Psi$. Then, the algebras $\bar{T}(\mathcal{A}_i), i \in \{1,\ldots,n\}$  are \emph{cyclically monotone independent} with respect to $(\Varphi,[\Varphi])$.
\end{corollary}
\begin{proof}
	When $n=2$, this is the case where $\Psi$ vanishes on $\mathcal{A}_1$ and coincide with $\Varphi$ on $\mathcal{A}_2$. The case $n>2$ follows by induction, considering the algebra generated by $\{\mathcal{A}_1, \dots,\mathcal{A}_{n-1}\}$ and the algebra $\mathcal{A}_n$.
\end{proof}

The main contribution of this paper is the computation of the cyclic conditional multivariate cumulants. We give their definition and explain the strategy for proving the Theorem \ref{th:vanishing_of_mixed_cumulants}.

Before proceeding, we need some technical definitions to define the new relevant partitions involved in the calculations of cumulants. The standard notions and notations for non-crossing partitions can be consulted in Section \ref{sec:bgnd}.

Recall that the clockwise cyclic order of $[n]$ is a ternary relation $\mathcal{R}$ on $[n]$ defined as follows: $(x,y,z) \in\mathcal{R}$ if and only if one encounters $y$ when going from $x$ to $z$ in the clockwise direction on the circle where the elements of $[n]$ are arranged in same, clockwise, order.

Given a subset $V\subset [n]$, the \emph{cut} of $\mathcal{R}$ by $V = \{v_1,\ldots,v_k\}$ is the partial order $\prec_V$ on $[n]\backslash V$ defined as follows. The chains of $\prec_V$ are the cyclic intervals $(v_\ell,v_{\ell+1})_{\mathcal{R}}$ (using the convention $v_{k+1}=v_1$) equipped with the total order $\prec_{v_k}$ (a cut of $\mathcal{R}$ by a singleton):
$$ (i \prec_V j) \text{ if and only if } (v_k,i,v_{k+1}),(v_k,j,v_{k+1}),(v_k,i,j) \in \mathcal{R}. $$
This subset $V$ will be chosen as a block of a partition $\pi$ of $[n]$.

Given a non-crossing partition $\pi$ of $[n]$, there is another order, this time on $[n]$, which we can define, depending on the choice of a block $R$ in the Kreweras complement of $\pi$. We consider the set $\{1,1',\ldots,n,n'\}$ ordered clockwise. If $R$ is a subset of $\{1',2',\ldots,n'\}$, we define $\prec_{R}$ on $\{1,1',\ldots,n,n'\}\backslash V$ as before and we restrict it to $[n]$ to obtain a partial order, which we also denote $\prec_{R}$ on $[n]$.

\begin{figure}[!ht]
	\includesvg[scale=0.6]{partition.svg}
	\caption{\label{fig:partition}Let $\pi=\{ \{ \{1\}, \{2,6,15\}, \{3,5\},\{4\},\{7,11,12\}, \{8,10\},\{9\} \{13,14\},\{16\} \}\}$. On the left side, the block $V$ is drawn with dots, and the corresponding chains of the preorder $\prec_{V}$ are represented by arrows. On the right side the block $R$ in the Kreweras complement of $\pi$ is filled with black, $R=\{\{5',7'\},\{15'\}\}$. The chains of the preorder $\prec_{R}$ are also represented by arrows.}
\end{figure}

The set of partitions ${\rm CyNC}(n)$ of $\{1,\ldots,n\}$ involved in the definition of the cumulants for cyclic conditional freeness is the set of \emph{cyclic non-crossing partitions}. A partition $\pi$ of $[n]$ belongs to ${\rm CyNC}(n)$ if $\pi$ is a non-crossing partition of the set $\{1,\ldots,n\}$.
We use the notation ${\rm CyNC}(n)$ (in place of the usual notation ${\rm NC}(n)$) to emphasize the differences between the multiplicative extension of sequences of maps $f_n,\,Df_n\colon\mathcal{A}^{\otimes n}\to \mathbb{C}_{n\geq 1}$ with $Df_n$ \emph{cyclically invariant} for any $n\geq 1$;
\begin{equation}\label{eq:definitionofcyclicinvariance}
	Df_{n}(a_1,\ldots,a_n)=Df(a_n,a_1,\ldots,a_{n-1}),\quad a_1,\ldots,a_n \in \mathcal{A},
\end{equation}
along non-crossing partitions (as in the infinitesimal free moment-cumulant formula) and along ${\rm CyNC}(n)$. In the cyclic case, $\pi \in {\rm CyNC}(n)$, we set
\begin{align} \label{eq:definitionofcyclicextension}
	 & f_\pi(a_1,\ldots,a_n)=\prod_{\{i_1, \dots,  i_k\} \in \pi}  f_{k}(a_{i_1}, a_{i_2}, \ldots, a_{i_{k}}),                                                                                 \\
	 & [f,Df]_{\pi}(a_1,\ldots,a_n)=\sum_{V\in\pi}Df_{|V|}(a_v, v\in V)\!\!\!\!\!\!\!\!\prod_{\{i_1 \prec_V \cdots \prec_V i_k\}\neq V \in \pi}  f_{k}(a_{i_1}, a_{i_2}, \ldots, a_{i_{|V|}}),
\end{align}
See Figure \ref{fig:partition} for a graphical representation.

\begin{definition}[cyclic conditional cumulants]\label{def:cycfreecumulantstwo} The \emph{cyclic conditional free cumulants} of the triple $(\Psi,\Varphi,\OmegA)$ is a sequence
	$$
		\cck=(\kappa_n^{\OmegA|\Varphi,\Psi})_{n\geq 1},
	$$
	where, for each $n\geq 1$,
	$$
		\cck_n \colon \mathcal{A}^{\otimes n} \to \mathbb{C},
	$$
	is recursively defined by, for any $a_1,\ldots,a_n \in \mathcal{A}$:
	\begin{multline}
		\label{eq:defmomentcumulant}
		\OmegA(a_1\cdots a_n)
		+(1-\OmegA(1_\mathcal{A})) [\Psi](a_1,\dots, a_n)-[\Varphi](a_1,\dots, a_n) \\
		=\cck_n(a_1,\dots, a_n)+\sum_{\substack{\pi\neq \mathbf{1}_n \\ \in {\rm CyNC}(n)}}[\kappa^{\Psi},\kappa^{\OmegA | \Varphi,\,\Psi}]_{\pi}(a_1 , \ldots , a_n).\nonumber
	\end{multline}
	where $\kappa^{\Psi} = (\kappa_n^{\Psi}\colon\mathcal{A}^{\otimes n}\to\mathbb{C})_{n\geq 1}$ is the sequence of free cumulants relatively to $\Psi$.

\end{definition}
\begin{example} We expand the above formula at low orders $(1,2$ and $3)$
	\begin{equation*}
		\kappa_1^{\OmegA \, | \, \Psi,\Varphi}(a_1) = \OmegA(a_1)+\beta^{\Psi}(a_1)-\beta^{\Varphi}  (a_1)
	\end{equation*}
	\begin{multline*}
		\kappa_2^{\OmegA \, | \, \Psi,\Varphi}(a_1,a_2)=\OmegA(a_1a_2)+\beta^{\Psi}(a_1,a_2)+\beta^{\Psi}(a_1)\beta^{\Psi}(a_2)-(\beta^{\Varphi}(a_1,a_2)+\beta^{\Varphi}(a_1)\beta^{\Varphi}(a_2))                       \\
		- \kappa_1^{\OmegA \, | \, \Psi,\Varphi}(a_1)\kappa^{\Psi}(a_2)-\kappa_1^{\OmegA \, | \, \Psi,\Varphi}(a_2)\kappa^{\Psi}(a_1)                                                                         \\
		=\OmegA(a_1a_2)+\beta^{\Psi}(a_1,a_2)+\beta^{\Psi}(a_1)\beta^{\Psi}(a_2)-(\beta^{\Varphi}(a_1,a_2)+\beta^{\Varphi}(a_1)\beta^{\Varphi}(a_2))                                                          \\
		- (\OmegA(a_1)+\beta^{\Psi}(a_1)-\beta^{\Varphi}(a_1)){\Psi}(a_2)-(\OmegA(a_2)+\beta^{\Psi}(a_2)-\beta^{\Varphi}(a_2)){\Psi}(a_1)
	\end{multline*}
	\begin{multline*}
		\kappa_3^{\OmegA \, | \, \Psi,\Varphi}(a_1, a_2, a_3) =\OmegA(a_1a_2a_3)+
		\beta^{\Psi}(a_1, a_2, a_3)+\beta^{\Psi}(a_3, a_1, a_2)+\beta^{\Psi}(a_2, a_3, a_1)                                                                                                                                  \\
		+ \beta^{\Psi}(a_1, a_2)\beta^{\Psi}(a_3)+\beta^{\Psi}(a_2, a_3)\beta^{\Psi}(a_1) + \beta^{\Psi}(a_3, a_1)\beta^{\Psi}(a_2)
		+\beta^{\Psi}(a_1)\beta^{\Psi}(a_2)\beta^{\Psi}(a_3))                                                                                                                                                                \\
		-\Big((\beta^{\Varphi}(a_1, a_2, a_3)+\beta^{\Varphi}(a_3, a_1, a_2)+\beta^{\Varphi}(a_2, a_3, a_1) + \beta_2^{\Varphi}(a_1, a_2)\beta^{\Varphi}(a_3)+\beta^{\Varphi}(a_2, a_3)\beta^{\Varphi}(a_1) \\
		+ \beta^{\Varphi}(a_3, a_1)\beta^{\Varphi}(a_2) +\beta^{\Varphi}(a_1)\beta^{\Varphi}(a_2)\beta^{\Varphi}(a_3)\Big)                                                                                    \\
		-(\OmegA(a_1)+\beta^{\Psi}(a_1)-\beta^{\Varphi}(a_1)){\Psi}(a_2,a_3)                                                                                                                                \\
		-\kappa_2^{\OmegA \, | \, \Psi,\Varphi}(a_1, a_2){\Psi}(a_3)-\kappa_2^{\OmegA \, | \, \Psi,\Varphi}(a_1, a_3){\Psi}(a_2)                                                                            \\
		=\OmegA(a_1a_2a_3)+
		\beta^{\Psi}(a_1, a_2, a_3)+\beta^{\Psi}(a_3, a_1, a_2)+\beta^{\Psi}(a_2, a_3, a_1)                                                                                                                                  \\
		+ \beta^{\Psi}(a_1, a_2)\beta^{\Psi}(a_3)+\beta^{\Psi}(a_2, a_3)\beta^{\Psi}(a_1) + \beta^{\Psi}(a_3, a_1)\beta^{\Psi}(a_2)
		+\beta^{\Psi}(a_1)\beta^{\Psi}(a_2)\beta^{\Psi}(a_3))                                                                                                                                                                \\
		-\Big((\beta^{\Varphi}(a_1, a_2, a_3)+\beta^{\Varphi}(a_3, a_1, a_2)+\beta^{\Varphi}(a_2, a_3, a_1) + \beta_2^{\Varphi}(a_1, a_2)\beta^{\Varphi}(a_3)+\beta^{\Varphi}(a_2, a_3)\beta^{\Varphi}(a_1) \\
		+ \beta^{\Varphi}(a_3, a_1)\beta^{\Varphi}(a_2) +\beta^{\Varphi}(a_1)\beta^{\Varphi}(a_2)\beta^{\Varphi}(a_3)\Big)                                                                                    \\
		-(\OmegA(a_1)+\beta^{\Psi}(a_1)-\beta^{\Varphi}(a_1)){\Psi}(a_2,a_3)                                                                                                                                \\
		-\Big(\OmegA(a_1a_2)+\beta^{\Psi}(a_1,a_2)+\beta^{\Psi}(a_1)\beta^{\Psi}(a_2)-(\beta^{\Varphi}(a_1,a_2)+\beta^{\Varphi}(a_1)\beta^{\Varphi}(a_2))                                                   \\
		- ((\OmegA(a_1)+\beta^{\Psi}(a_1)-\beta^{\Varphi}(a_1)){\Psi}(a_2)-(\OmegA(a_2)+\beta^{\Psi}(a_2)-\beta^{\Varphi}(a_2)){\Psi}(a_1))\Big){\Psi}(a_3)                                                   \\
		-\Big((\OmegA(a_1a_3)+\beta^{\Psi}(a_1,a_3)+\beta^{\Psi}(a_1)\beta^{\Psi}(a_3)-(\beta^{\Varphi}(a_1,a_2)+\beta^{\Varphi}(a_1)\beta^{\Varphi}(a_2))                                                    \\
		- (\OmegA(a_1)+\beta^{\Psi}(a_1)-\beta^{\Varphi}(a_1)){\Psi}(a_2)-(\OmegA(a_2)+\beta^{\Psi}(a_2)-\beta^{\Varphi}(a_2)){\Psi}(a_1))\Big){\Psi}(a_2)
	\end{multline*}
\end{example}
In the context of several algebras $\mathcal{A}_i,\,i\in I$, we refer to \emph{mixed cumulants} those cumulants with arguments coming from at least two distinct algebras $\mathcal{A}_i$. We refer the reader to \cite{speicher2020lecture,bozejko1996convolution} for the definition of free and conditionally free cumulants.
\begin{theorem}\label{th:vanishing_of_mixed_cumulants}
	Consider $\mathcal{A}_i,\,i\in I$ subalgebras of $\mathcal{A}$. Then the following are equivalent:
	\begin{enumerate}
		\item The algebras $\mathcal{A}_i,\,i\in I$ are cyclically conditionally free relatively to $(\Psi, \Varphi, \OmegA)$,
		\item the algebras $\mathcal{A}_i,\,i\in I$ are cyclically free with respect to $(\Psi, \Varphi, \OmegA - [\Varphi]-(1-\OmegA(1_{\mathcal{A}}))[\psi])$
		\item the mixed cyclic conditional free cumulants $\kappa^{\,\OmegA\,|\,\Varphi,\,\Psi}$, the mixed free cumulants $\kappa^{\Psi}$ and the mixed conditional cumulants $\kappa^{\Varphi \,|\, \Psi}$ vanish.
	\end{enumerate}
\end{theorem}
\begin{proof}

	See Section \ref{sec:prooftheoremvanishingmixedcumulants}.
\end{proof}

\subsection{Organization of the paper}

Besides the introduction, this work has four other sections.
In Section \ref{sec:cyclicinf}, we define cumulants for cyclic freeness and prove that they characterize it. We also extend the theorem of Février, Mastnak, Nica and Szpojankowski \cite{fevrier2019construction} connecting conditional freeness and infinitesimal freeness to the context of cyclic freeness, see Theorem \ref{th:fromcondfreetocyclic}.  Finally, in
\ref{sec:cyclicallyalternatingproduct}
we provide formulae for computing the cyclic free cumulants with products as entries. This is done later in Proposition \ref{prop:productccfree}, for cyclic conditional  cumulants.

In Section \ref{sec:conditionaltocyclicconditional}, we prove Theorem \ref{th:fromfreetocyclic} establishing a connection between conditional freeness and cyclic conditional freeness and in Section \ref{sec:prooftheoremvanishingmixedcumulants} we prove Theorem \ref{th:vanishing_of_mixed_cumulants} that cyclically free conditional cumulants characterize cyclic conditional freeness. We also draw the connection between our cyclically free conditional cumulants and the cyclic Boolean cumulants.

In Section \ref{sec:convolutions and limit theorems}, we give relations between formal transforms to compute the distribution of the sum and the product of two cyclically conditionally free elements and use them to prove limit theorems.

Finally, in Section \ref{sec:graphs} we give a graph product interpretation to cyclic conditional freeness.

\section{Cyclic and infinitesimal freeness}

\subsection{Background and notation for partitions } \label{sec:bgnd}
We collect basic definitions which will be used through this article. For any $n\geq 1$, $[n]$ is the interval of integers $\{1,2,\ldots,n\}$ and if $m\geq 1$, $\llbracket m,n \rrbracket$ is the set comprising all integers between $m,n$. We will use the variants $\llbracket m,n \llbracket = \{m,m+1,\ldots,n-1\}$ and so on.

A \emph{partition} is a collection of non-empty, disjoint subsets (called blocks) whose union is the entire set. A \emph{non-crossing partition} is a partition with the additional property that if $i,j \sim_\pi$ and $i<k<j<q$ then $k \not\sim_\pi q$. The set of non-crossing partitions of $[n]$ will be denoted ${\rm NC}(n)$. Recall that in the introduction we defined the set ${\rm CyNC}(n)$ of \emph{cyclic non-crossing partitions}.

An \emph{interval partition} is a partition such that every block is an interval of consecutive integers. The set of interval partitions will be denoted ${\rm Int}(n)$.

A \emph{cyclic interval partition} is a partition such that every block is a cyclic interval of consecutive integers (for the clockwise order). The set of cyclic interval partitions of $[n]$ will be denoted  ${\rm CyInt}(n)$. Cyclic interval partitions are non-crossing partitions, with all blocks being intervals but one : the block $1$ and $n$ belongs to (if it exists).

The Kreweras complement $K(\pi)$ of a non-crossing partition $\pi \in \nc(n)$ is the largest non-crossing partition of the set $\{\overline{1},\ldots,\overline{n}\}$ such that $\pi \cup K(\pi)$ is a non-crossing partition of the set $\{1,\overline{1},\ldots,n,\overline{n}\}$ when the elements are arranged in the clockwise order $1,\overline{1},2,\overline{2},\ldots,n,\overline{n}$. See \cite{kreweras1972partitions, nica2006lectures} for more details on non-crossing partitions and Kreweras complements.

Given a word $a = a_1\dots a_n$ on elements in $\mathcal{A}$ and a subset $S \subset [n]$, we denote by $a_S$ the \emph{subword} of $a$ consisting of the letters indexed by $S$. If $S$ is equipped with a total order $\prec$, we denote by $a_{(S,\prec)}$ the subword of $a$ consisting of the letters indexed by $S$, arranged according to the order $\prec$.

Note that a cyclic interval partition $I\neq \mathbf{1}_n$ of $[n]$ has a Kreweras complement which contains only one not-a-singleton block.

\subsection{Cumulants for cyclic freeness} \label{sec:cyclicinf}
In this section we address the problem of defining cumulants for cyclic freeness (see Definition \ref{def:cyclicfreeness}). We begin with defining these cumulants when $\OmegA(1_{\mathcal{A}})=0$. Next, we exhibit a pair $(\Psi,[\Psi])$ such that freeness with respect to $\Psi$ implies cyclic freeness with respect to $(\Psi,[\Psi])$. Finally, we define the cyclic free  cumulants of $(\Psi,\OmegA)$ as the -- restriction of the -- cyclic free cumulants of $(\Psi, D^{\Psi}\OmegA)$ where $D^{\Psi}\OmegA=\OmegA-\OmegA(1_{\mathcal{A}})\cdot [\Psi]$.

Let $f_n,Df_n\colon \mathcal{A}^{\otimes n}\to \mathbb{C}$ be two sequences of linear maps with $Df_n$ cyclically invariant. Recall the extension $[f,Df]$ along ${\rm CyNC}(n)$ defined in \eqref{eq:definitionofcyclicextension}.

\begin{proposition} \label{prop:leibniz}
	The function $[f,Df]$ is the unique functional extending $Df$ to ${\rm CyNC}$ and satisfying the following properties:
	\begin{enumerate}
		\item for any partition $\pi \in \nc(n)$, we have
		      \begin{equation}
			      \label{eqn:cycinv}
			      Df_{\pi}(a_1, \dots, a_n)= Df_{c \cdot \pi}(a_{c(1)},\dots, a_{c(n)}),~a_1,\ldots,a_n \in \mathcal{A},
		      \end{equation}
		      where $c\cdot\pi = \{ c(V),v\in\pi \}$ and $c=(1,\ldots,n)$ is the cyclic shift.
		\item \label{def:Dfpi}for non-crossing partitions $\pi_1\in \nc(\{1,\ldots,m\}$ and $\pi_2\in \nc(\{m+1,\ldots,m+n\}$ and for any ${a_1\otimes\cdots\otimes a_n}\in \mathcal{A}^{\otimes n}$, ${b_1\otimes\cdots\otimes b_m}\in \mathcal{A}^{\otimes m}$:
		      \begin{align*}
			       & D f_{\pi_1\cup \pi_2}(a_1,\dots, a_n, b_1,\dots, b_m)                                                                                          \\
			       & \hspace{2cm}=D f_{\pi_1}(a_1,\cdots, a_n) \cdot f_{\pi_2}(b_1,\dots, b_m)  + f_{\pi_1}({a_1,\cdots, a_n})\cdot D f_{ \pi_2}({b_1,\dots, b_m}).
		      \end{align*}
	\end{enumerate}
\end{proposition}
\begin{proof}
	The formula is a consequence of the formula \eqref{eq:definitionofcyclicextension} and the cyclic invariance of $Df$ (point (1)).
\end{proof}
\begin{example} Let $a_1,\ldots,a_4 \in \mathcal{A}$,
	\begin{multline*}
		Df_{ \{1,4\}, \{2,3\}}(a_1, a_2 , a_3 , a_4)=Df_{\{1,2\},\{3,4\}}(a_2 , a_3 , a_4 , a_1)\\ =f_2(a_2 , a_3)Df_2(a_4 , a_1)+Df_2(a_2 , a_3)f_2(a_4 , a_1).
	\end{multline*}
\end{example}

\begin{definition}[Cyclic free  cumulants]\label{def:cycfreecumulants}
	Suppose that $D\Psi:\mathcal{A}\to \mathbb{C}$ is tracial with $D\Psi(1_\mathcal{A})=0$. The \emph{cyclic free cumulants} relatively to $(\Psi,D\Psi)$ is a sequence :
	$$
		D\kappa_n^{\Psi}\colon\mathcal{A}^{\otimes n}\to \mathbb{C},\quad n\geq 1,
	$$
	recursively defined by
	\begin{align} \label{eqn:cymomentcumulants}
		 & D\Psi(a_1\cdots a_n) = D\kappa^{\Psi}_{n}(a_1,\dots, a_n) + \sum_{\substack{\pi \neq \mathbf{1}_n \\ \in \cync(n)}}[\kappa^{\Psi},D\kappa^{\Psi}]_{\pi}(a_1,\ldots, a_n),
	\end{align}
	where the $\kappa_n^{\Psi}$ are the free cumulants, $a_1,\ldots,a_n \in \mathcal{A}$.
\end{definition}
\begin{example}
	Let $a_1,\ldots,a_4 \in \mathcal{A}$. Here are an example of computations; we have underlined the terms that differ from the infinitesimal free moment-cumulant relations:
	\begin{align*}
		 & D\Psi(a_1) = D\kappa^{\Psi}(a_1),                                                                                                                            \\
		 & D\Psi(a_1a_2)=D\kappa_2^{\Psi}(a_1, a_2)+ D\kappa_1^{\Psi}(a_1)\Psi(a_2)+\Psi(a_1)D\kappa_1^{\Psi}(a_2),                                                     \\
		 & D\Psi(a_1a_2a_3)=D\kappa_2^{\Psi}(a_1, a_2, a_3)+D\kappa_2^{\Psi}(a_1, a_2)\Psi(a_3)+\Psi(a_1)D\kappa_2^{\Psi}(a_1, a_3)+\Psi(a_2)D\kappa_2^{\Psi}(a_2, a_3) \\
		 & \hspace{3cm}+D\kappa^{\Psi}(a_1)\Psi(a_2a_3)+\underline{D\kappa_1^{\Psi}(a_2)\Psi(a_3a_1)}+D\kappa_1^{\Psi}(a_3)\Psi(a_1a_2).
	\end{align*}
\end{example}

\begin{proposition}
	\label{prop:cyclicinvariance}
	The cyclic free  cumulants $(D\kappa^{\Psi}_{n})_{n\geq 1}$ are cyclically invariant:
	$$
		D\kappa_n^{\Psi}(a_1,\cdots,a_n)=D\kappa_n^{\Psi}(a_{c(1)},\cdots, a_{c(n)}),\quad a_1,\ldots,a_n \in \mathcal{A},
	$$
	where $c=(1,\ldots,n)$ is the cyclic permutation of the arguments.
\end{proposition}
\begin{proof}
	It is a simple consequence of the fact that $(D\kappa^{\Psi}_{n},\, n\geq 1)$ are uniquely determined by the equation \eqref{eqn:cymomentcumulants} and that $(D\kappa_{n}^{\Psi}\circ c,\, n\geq 1)$ where $c$ is the cyclic permutation of the arguments also satisfies \eqref{eqn:cymomentcumulants}.
\end{proof}

\begin{theorem}\label{th:cyclicfree}
	Cyclic freeness relatively to $(\Psi,D\Psi)$ is equivalent to the vanishing of all mixed free cumulants $(\kappa_n^{\Psi})_{n\geq 1}$ and mixed cyclic free  cumulants $(D\kappa_n^{\Psi})_{n\geq 1}$, as defined in Definition~\ref{def:cycfreecumulants}.
\end{theorem}
\begin{proof}
	See the following Section~\ref{sec:proofthmvanish}.
\end{proof}



Let $\OmegA$ be a tracial linear functional over $\mathcal{A}$.
We now define the cyclic free  cumulants for a general pair $(\Psi,\OmegA)$ as the restriction of the cyclic free cumulants (see Definition \ref{def:cycfreecumulants}) relatively to the pair $(\Psi, \OmegA - \OmegA(1_{\mathcal{A}}) [\Psi])$ (they are multilinear functional on the algebra $T(A)$) to arguments in $\mathcal{A} \subset T(\mathcal{A})$. Before doing, we need to extend $\Psi$ and $\OmegA$ to $T(\mathcal{A})$. We do this by considering their multiplicative extensions, still denoted by $\Psi$ and $\OmegA$ respectively:
$$
	\Psi(a_1 \otimes \dots \otimes a_n) = \Psi(a_1 \ldots a_n),\quad \OmegA(a_1\otimes\dots\otimes a_n) = \OmegA(a_1 \ldots a_n),\quad a_1,\ldots,a_n \in \mathcal{A}.
$$
and
$$
	\Psi(1) = 1, \quad \OmegA(1) = \OmegA(1_{\mathcal{A}}).
$$

\begin{definition}[Cyclic free  cumulants -- general case]
	\label{def:cyclicfreegencase}The cyclic free  cumulants with respect to the pair $(\Psi,\OmegA)$ are the restriction to $T(\mathcal{A})$ of the cyclic free  cumulants $D\kappa^{\Psi,\,\OmegA}$ of $(\Psi, \OmegA-\OmegA(1_{\mathcal{A}}) [ \Psi ])$
	see Definition \ref{def:cycfreecumulants}. More precisely, the following moment-cumulant formula, for any $a_1,\ldots,a_n \in \mathcal{A}$, $n\geq 1$, holds:
	\begin{align*}
		\OmegA(a_1\cdots a_n) & = \OmegA(1_{\mathcal{A}})[\Psi](a_1,\dots,a_n) + \sum_{\substack{\pi \in \cync}(n)} [\kappa^{\Psi},D\kappa^{\Psi,\OmegA}]_{\pi}(a_1,\ldots,a_n).
	\end{align*}
\end{definition}
\begin{example} We have
	\begin{align*}
		 & \OmegA(a_1) = \OmegA(1_{\mathcal{A}}) \beta^{\Psi}(a_1)+D^{\kappa,\,\OmegA}(a_1),                                                                                                             \\
		 & \OmegA(a_1a_2)=\OmegA(1_{\mathcal{A}}) (\beta^{\Psi}(a_1, a_2)+\beta^{\Psi}(a_1)\beta^{\Psi}(a_2))+D\kappa_2^{\Psi}(a_1, a_2)+ D\kappa_1^{\Psi}(a_1)\Psi(a_2)+\Psi(a_1)D\kappa_1^{\Psi}(a_2), \\
		 & \OmegA(a_1a_2a_3)=\OmegA(1_{\mathcal{A}}) \{ \beta^{\Psi}(a_1, a_2 , a_3)+\beta^{\Psi}(a_1, a_2)\beta^{\Psi}(a_3)+\beta^{\Psi}(a_1, a_3)\beta^{\Psi}(a_2)                                     \\
		 & \hspace{3cm}+\beta^{\Psi}(a_2, a_3)\beta^{\Psi}(a_1) + \beta^{\Psi}(a_1)\beta^{\Psi}(a_2)\beta^{\Psi}(a_3)\}                                                                                  \\
		 & \hspace{3.5cm}+D\kappa_2^{\Psi}(a_1, a_2)\Psi(a_3)+\Psi(a_1)D\kappa_2^{\Psi}(a_2, a_3)+\Psi(a_2)D\kappa_2^{\Psi}(a_2, a_3)                                                                    \\
		 & \hspace{4.0cm}+D^{\kappa,\,\OmegA}(a_1)\Psi(a_2, a_3)+\underline{D^{\kappa,\,\OmegA}(a_2)\Psi(a_3, a_1)}+D^{\kappa,\,\OmegA}(a_3)\Psi(a_1a_2).
	\end{align*}
\end{example}
We now state the main theorem of this section.
\begin{theorem}\label{th:vanishingcyclicinf}
	Cyclic freeness relatively to $(\Psi,\OmegA)$ is equivalent to the vanishing of mixed free cumulants $\kappa^{\Psi}$ and of mixed cyclic free  cumulants $D^{\Psi,\,\OmegA}\kappa$.
\end{theorem}

\begin{proof}Cyclic freeness relatively to $(\Psi,\OmegA)$ between subalgebras $\mathcal{A}_i$ in $\mathcal{A}$ is equivalent to cyclic freeness with respect to $(\Psi,\OmegA)$ in $T(\mathcal{A})$ between the subalgebras $T(\mathcal{A}_i)\subset T(\mathcal{A})$. But the latter is equivalent to cyclic freeness with respect to $(\Psi,\OmegA- (1-\OmegA(1_\mathcal{A}))[\Psi])$ (from Corollary \ref{lemma:fromfreetocyclic} and the fact that the moments conditions for cyclic freeness are linear on the second, tracial state), which is equivalent to the vanishing of mixed free cumulants $\kappa^{\Psi}$ and of the mixed cyclic free  cumulants $D^{\Psi,\OmegA- (1-\OmegA(1_\mathcal{A}))[\Psi]}\kappa$ (as multilinear functionals on $T(\mathcal{A}))$, by Theorem~\ref{th:cyclicfree}.
\end{proof}

\subsection{Proof of Theorem \ref{th:cyclicfree}}\label{sec:proofthmvanish}

This section aims to apply the method proposed by Lehner in \cite{lehner2004cumulants} based on Good's formula for classical cumulants.  Since we do not assume our functionals to be unital, we are in a setting that slightly outflanks the one of \cite{lehner2004cumulants}. For the reader's convenience, we outline the main steps of the procedure.
Pick $n\geq 1$ an integer. As a quick reminder, Good's formula for the $n$-th \emph{classical cumulants} $c_n(X_1,\ldots,X_n)$ of a sequence of (essentially bounded) random variables $X_1,\ldots,X_n$ states, for any integer $n\geq 1$,
$$
	c_{n}(X_1,\ldots,X_n) = \frac{1}{n}\mathbb{E}[\tilde{X}^{[n]}_1\cdots\tilde{X}^{[n]}_n]
$$
where the random variables $\tilde{X}^{[n]}_i,~1 \leq i \leq n$ are defined by
$$
	\tilde{X}^{[n]}_i=\xi X_i + \xi^{2} X^{(2)}_i+ \cdots + \xi^{n}X_i^{(n)},\quad \xi = e^{\frac{2i\pi}{n}}$$
and the vectors $(X_i^{(k)})_{1 \leq i \leq n},\, 1\leq k \leq n$ are independent and identically distributed as $(X_i)_{1 \leq i \leq n}$. The \emph{partitioned cumulants} satisfies a similar formula. Let $\pi$ be a partition of $[n]$ and denote by $c_{\pi}(X_1,\ldots,X_n)$ the corresponding cumulants defined by:
$$
	c_{\pi}(X_1,\ldots,X_n):=\prod_{B=\{i_1<\ldots<i_p\}\in\pi}c_{|B|}(X_{i_1},\ldots,X_{i_p}).
$$
As before, we pick $(X^{(k)}_i)_{1\leq i \leq n},~1\leq k \leq n$ independent copies of the random vector $(X_i)_{1\leq i \leq n}$. For each block $B \in \pi$, we pick a primitive $|B|$-th root of unity $\xi_{|B|}$ and set:
$$
	\tilde{X}^{B}_i := \xi_{|B|} X_i^{(1)} + \xi_{|B|}^2 X_i^{(2)} + \cdots + \xi_{|B|}^{|B|-1}X_i^{(|B|)}
$$
With $B_{(i)}$ the unique block of $B$ that contains $i$, $ 1 \leq i \leq n$, $B_{(i)}$, one has
\begin{equation}
	\label{eqn:partitionedgoodformula}
	c_{\pi}(X_1\ldots,X_n)=\mathbb{E}[\tilde{X}_1^{B_{(1)}}\cdots \tilde{X}_n^{B_{(n)}}]
\end{equation}

In the context of non-commutative probability, and for commutative independences (exchangeable systems in the terminology of \cite{lehner2004cumulants}), one uses Good's formula \ref{eqn:partitionedgoodformula} to introduce non-commutative cumulants for the corresponding notion of independence at stake. As outlined above, the procedure involves taking independent copies of a random variable. In our case, it is possible to build a probability space $\tilde{\mathcal{A}}$ and linear functional $\tilde{\Psi},D\tilde{\Psi}:\tilde{\mathcal{A}}\to \mathbb{C}$($\tilde{\Psi}$ unital and $\tilde{D\Psi}$ \emph{tracial}, with $D\tilde{\Psi}(1_{\tilde{\mathcal{A}}})=0$), such that, for each $i\geq 1$, we have a state-preserving injection
$$\iota_i : (\mathcal{A},\Psi,D\Psi)\to (\mathcal{A},\tilde{\Psi},D\tilde{\Psi}),
$$ and such that the $\iota_{j}(\mathcal{A}), j\geq 1$ are cyclically free. This is because cyclic freeness is a particular case of cyclic conditional freeness for which there is a notion of free product (see \cite{cebron2022asymptotic}).

This permits the construction of cyclically free copies, and any number of them, of the algebra $\mathcal{A}$ in a bigger algebra $\tilde{\mathcal{A}}$. Following \cite{lehner2004cumulants}, we define
$$
	K_n^{\Psi},\quad \kappa_n^{D\Psi}:\mathcal{A}^{\otimes n} \to \mathbb{C},
$$
by for any integer $n\geq 1$ and $a_1,\ldots,a_n \in \mathcal{A}$:
$$
	K_n^{\Psi}(a_1,\dots,a_n)=\frac{1}{n}\tilde{\Psi}(\tilde{a}_1,\dots, \tilde{a}_n),~
	~K_n^{D\Psi}(a_1,\dots,a_n)=\frac{1}{n}D\tilde{\Psi}(\tilde{a}_1,\dots, \tilde{a}_n)
$$
where
$$
	\tilde{a}_{i} = \xi \cdot \iota_1(a_i) + \cdots + \xi^{n}\cdot \iota_n(a_i), \quad \xi = e^{\frac{2i\pi}{n}}.
$$
The work \cite{lehner2004cumulants} provides the following combinatorial characterization of independence: cyclic freeness with respect to $(\Psi,D\Psi)$ is equivalent to the vanishing of the mixed $(K_n^{\Psi})_{n\geq 1}$ and of the mixed $(K_n^{D\Psi})_{n\geq 1}$ cumulants.

The rest of the proof consists in proving that $K_n^{\Psi}=\kappa_n^{\Psi}$ and $K_n^{D\Psi}=D\kappa_n^{\Psi}$. The fact that $K_n^{\Psi}=\kappa_n^{\Psi}$ is already contained in the paper of \cite{lehner2004cumulants} (see Section~4.4).

The partitioned cumulants $K^{D\Psi}_\pi$, $\pi \in \mathcal{P}(n)$ are defined as explained above for partitioned classical cumulants. Namely, given a partition $\pi$ and a block $B\in\pi$ of $\pi$, $B = \{b_1 < \cdots < b_p\}$, we set for each integer $1 \leq i \leq p$,
$$
	a_{b_i}^{\pi} =  \xi_{|B|} \cdot  \iota_{b_1}(a_{b_i}) +\xi^2_{|B|} \cdot \iota_{b_2}(a_{b_i}) + \cdots + \xi^{p-1}_{|B|}\cdot \iota_{b_p}(a_{b_i})
$$
where $\xi_B$ is a $p^{th}$ primitive root of unity and set (see Proposition 2.8 in \cite{lehner2004cumulants}),
\begin{align*}
	 & K_{\pi}^{D\Psi}(a_1,\ldots,a_n) = (\prod_{B\in\pi}|B|)^{-1}D\tilde{\Psi}(a_1^{\pi} \cdots a_n^{\pi}).
\end{align*}
Now, from Proposition 2.7 in \cite{lehner2004cumulants}, the following moment-cumulant formula is in hold, for any $a_1,\ldots,a_n \in \mathcal{A}$:
\begin{equation*}
	D\Psi(a_1,\dots,a_n) = \sum_{\pi \in \mathcal{P}(n)} K^{D\Psi}_{\pi}(a_1,\dots,a_n).
\end{equation*}
\begin{proposition} \label{prop:firstpart}
	The following properties hold:
	\begin{enumerate}
		\item Each $K^{D\Psi}_\pi$ is cyclically invariant, namely, for any $a_1,\ldots,a_n \in \mathcal{A}$:
		      $$
			      K^{D\Psi}_\pi(a_1, \dots, a_n)=K^{D\Psi}_{c\cdot\pi}(a_n,a_1 ,\dots, a_{n-1}),
		      $$
		      where $c = (1,\ldots,n)$ is the cyclic shift.
		\item With $p_1\in \mathcal{P}(\{1,\ldots,m\})$, $p_2\in \mathcal{P}(\{m+1,\ldots,m+n\})$, two partitions and $p_1\cup p_2 \in \mathcal{P}(m+n)$ the concatenation of $p_1$ and $p_2$, for any $a_1,\ldots,a_m\in \mathcal{A}$ and $b_1,\ldots,b_n\in \mathcal{A}$:
		      \begin{align*}
			      K^{D\Psi}_{\pi_1\cup \pi_2}(a_1,\dots,a_m,b_1,\dots,b_n)= K^{D\Psi}_{\pi_1}(a_1, & \dots,a_m)  \kappa^\Psi_{\pi_2}(b_1,\dots,b_n)                                            \\
			                                                                                       & + \kappa^\Psi_{\pi_1}(a_1,\dots,a_m) K^{D\Psi}_{ \pi_2}(b_1,\dots,b_n).\label{eq:pullout}
		      \end{align*}
		      \
	\end{enumerate}
\end{proposition}

\begin{proof}
	\begin{enumerate}

		\item The first assertion follows from the fact that $D\tilde{\Psi}$ is tracial :
		      \begin{multline*}
			      K^{D\Psi}_\pi(a_1, \dots, a_n)= (\prod_{B\in\pi}|B|)^{-1}D\tilde{\Psi}(a_1^{\pi} \cdots a_n^{\pi})\\
			      = (\prod_{B\in c\cdot\pi}|B|)^{-1}D\tilde{\Psi}(a_n^{c\cdot \pi} a_1^{c\cdot \pi} \cdots a_{n-1}^{c\cdot \pi})
			      = K^{D\Psi}_{c\cdot\pi}(a_n,a_1 ,\dots, a_{n-1}).
		      \end{multline*}

		\item The second one follows from the fact that
		      \begin{equation*}
			      D\Psi(a_1\cdots a_mb_{1}\cdots b_n)=D\Psi(a_1\cdots a_m)\cdot\Psi(b_{1}\cdots b_n)+\Psi(a_1\cdots a_m)\cdot D\Psi(b_{1}\cdots b_n)\label{eq:derivationprop}
		      \end{equation*}
		      whenever $(a_1,\ldots,a_m)$ are cyclic free from $(b_1,\ldots,b_n)$ (recall $D\Psi(1_{\mathcal{A}})=0$) obtained by expanding :
		      \begin{equation*}
			      D\Psi([a_1\cdots a_m - \Psi(a_1\cdots a_m)][b_1\cdots b_n - \Psi(b_1\cdots b_n)]) = 0.
		      \end{equation*}
		      Hence,
		      \begin{align*}
			      K^{D\Psi}_{p} & (a_1,\cdots,a_n,b_1,\cdots,b_m)                                                                                                                                                                                                                                    \\
			                    & = (\prod_{B\in p_{1}}|B|)^{-1}(\prod_{B\in p_{2}}|B|)^{-1}D\tilde{\Psi}(a_1^{p_{1}} \cdots a_n^{p_{1}} b_1^{p_{2}} \cdots b_m^{p_{2}})                                                                                                                             \\
			                    & = (\prod_{B\in p_{1}}|B|)^{-1}(\prod_{B\in p_{2}}|B|)^{-1}\big(D\tilde{\Psi}(a_1^{ p_{1}} \cdots a_n^{ p_{1}}) \tilde{\psi}(b_1^{ p_{2}} \cdots b_m^{ p_{2}}) + \tilde{\Psi}(a_1^{ p_{1}} \cdots a_n^{ p_{1}})D\tilde{\psi}(b_1^{ p_{2}} \cdots b_m^{ p_{2}})\big) \\
			                    & = K^{D\Psi}_{ p_1}(a_1,\dots,a_m) \cdot \kappa^\Psi_{ p_2}(b_1,\dots,b_n) + \kappa^\Psi_{ p_1}(a_1,\dots,a_m)\cdot K^{D\Psi}_{  p_2}(b_1,\dots,b_n)
		      \end{align*}
	\end{enumerate}
\end{proof}
Applying Proposition \ref{prop:firstpart}, we conclude that for any non-crossing partition $\pi$,
$$
	D\kappa^{\Psi}_{\pi}(a_1,\ldots,a_n) = K^{D\Psi}_{\pi}(a_1,\ldots,a_n),\quad a_1,\ldots,a_n \in \mathcal{A}.
$$
What is left to prove is that $K^{D\Psi}_{\pi}(a_1,\ldots,a_n)=0$ whenever $\pi$ has a crossing. This is the content of Proposition below.

\begin{proposition}	 Let $\pi$ be a partition with a crossing:
	(recall that it means $
		\exists a < b < c < d,~ a \sim_\pi c,~b\sim_\pi d
	$)
	then $K^{D\Psi}_{\pi}(a_1,\dots,a_n)=0$.
\end{proposition}
\begin{proof}

	For the last point, let $ p$ be a partition with a crossing. Let us prove that the associated partitioned cumulants vanishes; $K^{D\Psi}_{ p}(a_1,\cdots,a_n)=0$.

	Recall that an connected partition $p \in \mathcal{P}(n)$ is a partition for which the coarsest non-crossing partition greater than $p$ is the one-block partition $\{ \{1,\ldots,n\} \}$.

	Given $p \in \mathcal{P}(n)$ we denote by $p^{\rm nc}$ the non-crossing closure of $p$, ie the coarsest non-crossing partition greater than $p$. The connected components of $p$ are the restrictions of $p$ to the blocks of ${p}^{nc}$.
	By using the previous proposition, we infer that $K^{D\Psi}_{ p}(a_1,\cdots,a_n)$ can be written in terms of the partitioned cumulants associated with the connected components of $p$. In fact, let $p_1$ be a connected component of $p$ corresponding to an interval block in $p^{nc}$. Then, by applying cyclic invariance, we can assume that this block is $\{1,\ldots,m\}$ for some $m<n$. Writing $p=p_1 \cup p_2$ where $p_2$ is the restriction of $p$ to $\{m+1,\ldots,n\}$, we have from the previous proposition:
	\begin{align*}
		K^{D\Psi}_{ p}(a_1,\dots,a_n)= K^{D\Psi}_{ p_1}(a_1, & \dots,a_m)  \kappa^\Psi_{ p_2}(a_{m+1},\dots,a_n)                        \\
		                                                     & + \kappa^\Psi_{ p_1}(a_1,\dots,a_m) K^{D\Psi}_{ p_2}(a_{m+1},\dots,a_n).
	\end{align*}
	We can then show, by induction on the number of connected components of $p$, that $K^{D\Psi}_{ p}$ is equal to a linear combination of products of free cumulants $\kappa^{\Psi}$ and the cumulants $K^{D\Psi}_{ p'}$ where $ p'$ are the connected components of $p$. Hence, it suffices to prove that $K^{D\Psi}_{ p}=0$ whenever $p$ is connected and contains a crossing. In particular, $p$ has no interval blocks and has at least two blocks. We can assume that $1$ and $n$ are not in the same block of $p$ (otherwise, by cyclic invariance, we can rotate the indices to make it so). Hence, for any $a_1,\ldots,a_n \in \mathcal{A}$, the cumulants $K^{D\Psi}_{ p}(a_1,\cdots,a_n)$ is the expectation of a cyclically alternating product of centered elements with respect to $\tilde{\Psi}$ and with respect to $D\tilde{\Psi}$. By cyclic freeness, we conclude that $K^{D\Psi}_{ p}(a_1,\cdots,a_n)=0$.
\end{proof}
\subsection{From conditional freeness to cyclic and infinitesimal freeness}
\label{sec:linkwith}
This section presents a connection between conditional freeness and infinitesimal or cyclic freeness.
If $S$ is a set, we denote by $\mathbb{C}\langle S \rangle$ the free unital complex algebra generated by $S$.
Theorem \ref{th:fromcondfreetocyclic} below corresponds to the {\color{gray} grey} arrow in Fig. \ref{fig:Figureone}. Following a remark by Belinschi and Shlyakhtenko, the authors Février, Mastnak, Nica and Szpojankowski introduced in \cite{fevrier2019construction} the tracial functional
$$
	W(\Varphi):\mathbb{C}\langle x_a:a\in\mathcal{A}\rangle\to \mathbb{C}
$$
defined by, for $a_1,\ldots,a_n \in \mathcal{A}$:
\begin{align}
	 & W(\Varphi)(x_{a_1} \cdots x_{a_n})=\sum_{1\leq i\leq n}\beta^{\Varphi}(a_i \otimes a_{i+1} \otimes \cdots \otimes a_n \otimes a_1 \otimes \cdots \otimes a_{i-1} \otimes a_i),\label{eq:defFMNS} \\
	 & W(\Varphi)(1)=0. \nonumber
\end{align}
The algebra $\mathbb{C}\langle x_a, a \in \mathcal{A} \rangle$ is the polynomial algebras generated by the non-commutative indeterminates $x_a, a \in \mathcal{A}$. This is not the same as the tensor algebra of $\mathcal{A}$ as the indeterminates $x_{\lambda a}$ and $x_{a}$ are not set equal ($\lambda \in \mathbb{C})$.
If given another linear functional $\psi\colon\mathcal{A}\to\mathbb{C}$, $\psi$ extends to a linear map $\Psi \colon \mathbb{C}\langle x_a:a\in\mathcal{A}\rangle \to \mathbb{C}$ :
\begin{equation*}
	\Psi(x_{a_1}\cdots x_{a_n})=\Psi(a_1\cdots a_n),\quad \Psi(1)=1,\quad a_1,\ldots,a_n\in\mathcal{A}.
\end{equation*}
The result of Février, Mastnak, Nica and Szpojankowski is stated as follows.

\begin{theorem}[Theorem 1.3 and 1.4 in \cite{fevrier2019construction}]
	\label{thm:szpojankowski}
	Suppose that $\Psi$ is \emph{tracial}. Let $\mathcal{A}_1 ,\mathcal{A}_2 \subset \mathcal{A}$ be two \emph{conditionally free} unital subalgebras relatively to $(\Psi,\Varphi)$. Then, the two algebras $\mathbb{C}\langle x_a:a\in\mathcal{A}_1\rangle$ and $\mathbb{C}\langle x_a:a\in\mathcal{A}_2\rangle$ are \emph{infinitesimally free} subalgebras (or equivalently cyclically free) of $(\mathbb{C}\langle x_a:a\in\mathcal{A}\rangle, \Psi,W(\Varphi))$,
\end{theorem}
We prove the following generalisation of Theorem \ref{thm:szpojankowski}.
\begin{theorem}\label{th:fromcondfreetocyclic}
	Let $\mathcal{A}_1 ,\mathcal{A}_2 \subset \mathcal{A}$ be two conditionally free unital subalgebras with respect to $(\Psi,\Varphi)$. Then, the subalgebras $\mathbb{C}\langle x_a:a\in\mathcal{A}_1\rangle$ and $\mathbb{C}\langle x_a:a\in\mathcal{A}_2\rangle$ are \emph{two cyclically free unital subalgebras of $\mathbb{C}\langle x_a:a\in\mathcal{A}\rangle$ with respect to the pair $(\Psi,W(\Varphi))$}, where $W(\Varphi)$ is given by~\eqref{eq:defFMNS}.
\end{theorem}

\begin{proof}
	Since the linear functional $W(\Varphi)$ is tracial, it suffices to prove for every cyclically alternating sequence $(P_1,\ldots,P_N)$ ($N\geq 2$) of centered elements with respect to $\Psi$ that:
	$$
		W(\Varphi)(P_1\cdots P_N) =0.$$
	For convenience of writing, we will assume that $P_j$ is a monomial, but everything can be extended to the general case by the multilinearity of all functionals involved.  We rewrite the full product as
	$$P_1\cdots P_N=x_{a_1}\cdots x_{a_n},$$
	with $P_k=\overrightarrow{\prod}_{i\in I_k}x_{a_i}$ for a certain interval partition $[n]=I_1\cup \cdots \cup I_N$. We have $\mathcal{A}_{j(1)},\ldots,\mathcal{A}_{j(N)} \in \{\mathcal{A}_1,\mathcal{A}_2\}$ in such a way that $i\in I_k \Rightarrow a_i\in \mathcal{A}_{j(k)}$ and the cyclic alternating condition means that $j(k)\neq j(k+1)$ for $1\leq k \leq N$ (with the convention $N+1=1$). Then
	\begin{equation*}
		W(\Varphi)(P_1\cdots P_N)=\sum_{1\leq i\leq n}\beta^{\Varphi}(a_i, a_{i+1}, \cdots , a_n, a_1, \cdots, a_{i-1}, a_i),
	\end{equation*}
	and it remains to prove that, for any $i\in [n]$,
	$$\beta^{\Varphi}(a_i, a_{i+1}, \cdots , a_n , a_1 , \cdots , a_{i-1} , a_i)=0,$$
	This will result from the following Lemma~\ref{lemma:vanishingofbeta}.
\end{proof}
\begin{lemma}\label{lemma:vanishingofbeta} Let~$\mathcal{A}_1$ and $\mathcal{A}_2 \subset \mathcal{A}$ be two conditionally free unital subalgebras relatively to $(\Psi,\Varphi)$.

	Let $a_1,\dots a_n \in \mathcal{A}$ and assume that for a certain interval partition $I=\{I_k,\,1\leq k \leq |I|\}$ of $[n]$ \emph{with at least two blocks},
	\begin{itemize}
		\item for all $i\in I_k$, $a_i\in \mathcal{A}_{j(k)}$ for some  $j(1),\ldots,j(|I|) \in \{1,2\}$,
		\item $\Psi\left(\overrightarrow{\prod}_{i\in I_k}a_i\right)=0$,
		\item $j(k)\neq j(k+1)$ for $1\leq k \leq |I|$ (with the cyclic convention $|I|+1=1$).
	\end{itemize}
	Then, for any $1\leq i\leq j\leq n$,
	\begin{equation*}
		\beta^{\Varphi}(a_j, a_{j+1}, \ldots a_n, a_1,\cdots, a_{i-1}, a_i)=0.\label{eq:useful}
	\end{equation*}
\end{lemma}
\begin{proof}
	Let $1\leq i \leq j \leq n$. Set $r:=n+1-j+i$. Let $a_1,\dots,a_n \in \mathcal{A}$ and set $w = a_j, a_{j+1}, \cdots, a_n, a_1,\cdots , a_{i-1},a_i$.
	We have \cite{arizmendi2015relations}
	\begin{align*}
		\beta^{\Varphi}(w)
		=\sum_{\substack{\pi \in \nc(r) \\ 1\sim_\pi r}}\kappa_\pi^{\Varphi|\Psi}(w)
		=\sum_{\substack{V\subset [r]   \\ \{1,r\}\subset V }}\kappa_V^{\Varphi|\Psi}(w)\cdot \sum_{\substack{\pi \in \nc(r)\\ V\in\pi }}\prod_{ W\in \pi, W\neq V}\kappa_W^{\Psi}(w).
	\end{align*}
	Owing to conditional freeness of the algebras $\mathcal{A}_1$ and $\mathcal{A}_2$, $\kappa_V^{\Varphi|\Psi}(a_j , a_{j+1}, \cdots , a_n , a_1 , \cdots , a_{i-1} , a_i)=0$ whenever the variables do not belongs to the same algebra. So we can restrict the sum to the subset $V$, which connects variables $a_k$'s in the same algebra (let's say $\mathcal{A}_1$, as the other case is similar).
	For a fixed $ V \subset [r]$ containing $1$ and $r$ and connecting $a_k$'s in the same algebra $\mathcal{A}_1$, we split $[r]\setminus V$ in maximal intervals $W_1,\ldots, W_\ell$ in such a way that the sum of the free cumulants yields $\Psi$ on each interval:
	\begin{align*}
		 & \sum_{\substack{\pi \in \nc(r) \\ V\in\pi }}\prod_{W\in \pi, W\neq V}\kappa_W^{\Psi}(w)
		=\prod_{m=1}^\ell\Psi(w_{W_m}).
	\end{align*}
	The product on the right-hand side of the above equality vanishes
	since at least one interval $W_m$ contains a variable in $\mathcal{A}_2$, leading to a term
	$$\Psi(w_{W_m})=\Psi(\underbrace{\vphantom{\overrightarrow{\prod}_{i\in I_{t+1}}a_i}\cdots}_{\in \mathcal{A}_1}\underbrace{\vphantom{\overrightarrow{\prod}_{i\in I_{t+1}}a_i}\overrightarrow{\prod}_{i\in I_t}a_i}_{\in \mathcal{A}_2}\cdot \underbrace{\vphantom{a_{r+1}}\overrightarrow{\prod}_{i\in I_{t+1}}a_i}_{\in \mathcal{A}_1}\cdots \underbrace{\overrightarrow{\prod}_{i\in I_{t+s}}a_i}_{\in \mathcal{A}_2}\underbrace{\vphantom{\overrightarrow{\prod}_{i\in I_{t+1}}a_i}\cdots }_{\in \mathcal{A}_1})=0$$
	by freeness of $\mathcal{A}_1$ from $\mathcal{A}_2$ (since $\Psi$ applied to an alternating product of variables which are centered variables with respect to $\Psi$, except possibly the first and the last one).
\end{proof}
\subsection{Cyclically alternating product of cyclically free variables} \label{sec:cyclicallyalternatingproduct}
In the previous section, we have introduced cumulants for the cyclic free independence. We now prove formulae for these cumulants when the arguments are products of random variables. First, we recall definitions pertaining to \emph{type $B$ free probablity theory} \cite{biane2003non}.

Let $n\geq 1$ be an integer. The set of non-crossing partitions of type $B$ of $[n]$, denoted ${\rm NC}^{(B)}(n)$, is the set of all non-crossing partitions of $\{-1 < -2 < \ldots < -n < 1 < \ldots < n \}$ invariant with respect to the inversion map $x \mapsto -x$: this means that $-B \in \pi$ for any $B\in \pi$.
Following \cite{fevrier2010infinitesimal}, we denote by ${\rm NCZ}(n) \subset \nc^{(B)}(n)$ the set of non-crossing partitions of type $B$ with a \emph{zero block}: a block stable by inversion. This zero block is necessarily unique \cite[Section 1.2]{biane2003non}. The complementary set $\nc^{(B)}(n)\setminus {\rm NCZ}^{(B)}(n)$ is denoted ${\rm NCZ}^{\prime}(n)$:
\begin{equation}
	\nc^{(B)}(n) = {\rm NCZ}(n) \sqcup {\rm NCZ}'(n).
\end{equation}
We let ${\rm Abs}: \nc^{(B)}(n)\to \nc(n)$ be the map that take the image of a non-crossing partition of type $B$ by the absolute map ${\rm Abs}:\{-1,\ldots,n\} \to \{1,\ldots,n\}$. The fibre above a non-crossing partition $\pi$ by the absolute value map is described as follows. The intersection of ${\rm Abs}^{-1}(\pi)$ with ${\rm NCZ}(n)$ is in bijection with $\pi$: each block of $\pi$ yields an element of ${\rm Abs}^{-1}(\pi)\cap {\rm NCZ}(n)$. The intersection of ${\rm Abs}^{-1}(\pi)$ with ${\rm NCZ}^{\prime}(n)$ is in bijection with blocks of the Kreweras complement of $\pi$. We will freely identify a partition $\pi \in {\rm NCZ}(n)$ (resp. $\pi \in {\rm NCZ}'(n)$) with a pair $(\tilde{\pi},O)$ where $\tilde{\pi} = {\rm Abs}(\pi)$ and $O$ is the zero block of $\pi$ (resp. $R$ the block of the Kreweras complement of $\tilde{\pi}$ associated to $\pi$) and write $\pi = (\tilde{\pi},O)$ (resp. $\pi = (\tilde{\pi},R)$).

As before, we have two ways of extending multilinear maps over $\mathcal{A}$ along type $B$ non-crossing partitions.  Given two sequences $f_n,g_n:\mathcal{A}^{\otimes n}$, $n\geq 1$, we set first :
\begin{align}
	\label{eqn:cymultun}
	 & |f,g|^{(B)}_{\pi}(a_1, \cdots , a_n) = g_{|O|}(a_{O})\prod_{V \neq O\,\in \tilde{\pi}} f_{|V|}(a_{V}) &  & \pi=(\tilde{\pi},O) \in {\rm NCZ}(n), \\
	 & |f,g|^{(B)}_{\pi}(a_1, \cdots, a_n) = \prod_{V \in \tilde{\pi}} f_{|V|}(a_{V})                        &  & \pi=(\tilde{\pi},R) \in {\rm NCZ}'(n)
\end{align}
and also
\begin{align}
	\label{eqn:cymultdeuxun}
	 & [f,g]^{(B)}_{\pi}(a_1, \cdots , a_n) = g_{|O|}(a_{O})\prod_{V \neq O\,\in \tilde{\pi}} f_{|V|}(a_{(V,<_O)}) &  & \pi=(\tilde{\pi},O) \in {\rm NCZ}(n),                           \\
	 & [f,g]^{(B)}_{\pi}(a_1, \cdots , a_n) = \prod_{V \in \tilde{\pi}} f_{|V|}(a_{(V,<_R)})                       &  & \pi=(\tilde{\pi},R) \in {\rm NCZ}'(n)	\label{eqn:cymultdeuxdeux}
\end{align}

We notice that while we have the equality $g_n = [f,g]^{(B)}_{(1_n,\{1,\ldots,n\})}$, the values of $f$ on $a_1\otimes\cdots\otimes a_n$ and on the cyclically permuted words of $a_1\otimes\cdots \otimes a_n$ are embodied in $[f,g]^{(B)}$, for any integer $1 \leq i \leq n$:
$$
	[f,g]^{(B)}_{(1_n,\{i\})}(a_1,\cdots,a_n)=f(a_{i},\dots, a_n, a_1,\dots,a_{i-1}).
$$
Following the philosophy adopted in the previous section, we may use a different symbol ${\rm CyNC}^{(B)} = {\rm CyNCZ} \sqcup {\rm CyNCZ}'$ for the set of non-crossing partitions of type $B$, for which we require multiplicative functions (functions that factorises over the blocks of a partition) to fulfil relations \eqref{eqn:cymultdeuxdeux} and $\eqref{eqn:cymultdeuxun}$.

Since a partition $\pi \in {\rm CyNCZ}$ is of the form $\pi = ({\rm Abs}(\pi), O)$ where $O$ is a block of ${\rm Abs}(Kr(\pi))$, we use $\prec_\pi$  for the order $\prec_{O}$ on $[n]\setminus O$ defined in the introduction. Likewise, given a partition $\pi \in {\rm CyNCZ}$, $\pi = ({\rm Abs}(\pi), R)$ where $R$ is a block of ${\rm Kr}({\rm Abs}(\pi))$, we use the notation $\prec_\pi$ for the order $\prec_{R}$ on $[n]$ defined in the introduction.
\begin{remark}
	Going back to the definition of $[\Psi]$, equation \eqref{eqn:definsoulcy}, with the definition of the order $\prec_{\pi}$, $\pi \in {\rm CyNCZ'}(n)$ we just give we see that by inserting the relations expressing the Boolean cumulants in terms of the free cumulants \cite{arizmendi2015relations}, we obtain
	$$
		[\Psi](a_1 \otimes \cdots \otimes a_n)=\sum_{\pi\in{\rm CyNCZ'}} [\kappa^{\Psi},0]^{(B)}_{\pi}(a_1,\ldots,a_n),\quad a_1,\ldots,a_n \in \mathcal{A}.
	$$
\end{remark}
\begin{proposition}
	Let $\pi \in {\rm CyNC}^{(B)}(n)$.
	\begin{itemize}
		\item If $\pi,\pi^{\prime}\in{\rm CyNCZ}(n)$, with $\pi^{\prime} \leq \pi$ (for the containement order), the order $\prec_{\pi^{\prime}}$ is equal to $\prec_{\pi}$ where both are defined.
		\item If $\pi\in{\rm CyNCZ'}(n)$ and $\pi^{\prime} \leq \pi$, the order $<_{\pi^{\prime}}$ is equal to $\prec_{\pi}$.
	\end{itemize}
\end{proposition}
\begin{proof}
	If $\pi \in {\rm CyNCZ}$ then $\pi = ({\rm Abs}(\pi),O)$. If $\pi^{\prime}\leq \pi$ and $\pi^{\prime} \in {\rm CyNCZ}$, then $\pi^{\prime} = (\rm Abs(\pi^{\prime}), O^{\prime})$ with ${\rm Abs}(\pi^{\prime}) \leq {\rm Abs}(\pi)$ and $O^{\prime} \subset O$. Hence, the order $<_{\pi^{\prime}}$ is equal to $\prec_{\pi}$ where both are defined.

	We remark that $\pi^{\prime} \leq \pi$ and $\pi \in {\rm CyNCZ}^{\prime}$ implies $\pi^{\prime} \in {\rm CyNCZ}^{\prime}$ for otherwise a block of $\pi$ would contain a pair $-i,i$. The second point follows from the fact the block $O^{\prime}$ of the Kreweras complement of ${\rm Abs}(\pi^{\prime})$ such that $\pi^{\prime} = ({\rm Abs}(\pi^{\prime}), O^{\prime})$ contains the block $O$ of the Kreweras complement of $\pi$ such that $\pi=({\rm Abs}(\pi),O)$. To prove this we use the fact that ${\rm Abs}(Kr(\pi)) = Kr({\rm Abs}(\pi))$ and that $O = {\rm Abs}(O_{\pi})$, $O^{\prime} = {\rm Abs}(O_{\pi^{\prime}})$ where $O_{\pi}$ and $O_{\pi'}$ are the balanced blocks of the Kreweras complement of $\pi$ and $\pi'$ respectively. Since $\pi^{\prime} \leq \pi$ implies $Kr(\pi) \geq Kr(\pi^{\prime})$, one has $O_{\pi} \subset O_{\pi}^{\prime}$ and thus ${\rm Abs}(O_{\pi}) \subset {\rm Abs}(O_{\pi^{\prime}})$.
\end{proof}
For the remaining part of the section, $D\Psi\colon \mathcal{A}\to\mathbb{C}$ and $D\Psi(1_{\mathcal{A}})=0$ and we recall that the $D\kappa^{\psi}_n,\,n\geq 1$ are the cyclic free cumulants relatively to $\Psi$ (see Definition \ref{def:cycfreeness}).
\begin{proposition} \label{prop:ordertypeB}
	For any $\pi \in {\rm CyNCZ}(n)$ (with a balanced block) and $a_1,\ldots,a_n \in \mathcal{A}$:
	\begin{align*}
		[\Psi, D\Psi]^{(B)}_{\pi}(a_1,\ldots,a_n)=\sum_{\substack{\pi^{\prime} \in {\rm CyNCZ}(n) \\ \pi^{\prime} \leq \pi}}[\kappa^{\Psi},D\kappa^{\Psi}]_{\pi'}^{(B)}(a_1, \dots, a_n).
	\end{align*}
	For any $\pi \in {\rm CyNCZ}^{\prime}(n)$ (without a balanced block) and $a_1,\ldots,a_n \in \mathcal{A}$:
	\begin{align*}
		[\Psi, D\Psi]^{(B)}_{\pi}(a_1,\cdots,a_n)= \sum_{\substack{\pi^{\prime} \in {\rm CyNCZ}^{\prime}(n), \\ \pi^{\prime} \leq \pi}} \kappa^{\Psi}_{\pi^{\prime}}(a_1,\cdots,a_n) = \sum_{\substack{\pi^{\prime} \in {\rm CyNC}^{(B)}(n),\\ \pi^{\prime} \leq \pi}} [\kappa^{\Psi},D\kappa^{\Psi}]^{(B)}_{\pi^{\prime}}(a_1,\cdots,a_n).
	\end{align*}
\end{proposition}

\begin{proof}
	The balanced block of a partition $\pi \in {\rm CyNCZ}$ is the block $O\in \pi$ for which there exists $i\in [n]$ such that $i,-i \in O$. Thus, if $\pi^{\prime} \leq \pi$, then the balanced block $O^{\prime}$ of $\pi^{\prime}$ is contained in the balanced block $O$ of $\pi$. This implies that the order $\prec_{O^{\prime}}$ coincides with $\prec_O$ on the complementary set of $O$ in $[n]$. Hence, for any block $V\neq O \in \pi$:
	$$
		{\Psi}(a_{(V,\prec_O)})=\sum_{\pi_V \in {\rm NC}(|V|)} \kappa^{\Psi}_{\pi}(a_{(V,\prec_{O^{\prime}})}).
	$$
	The first formula follows.
	For the last two inequalities, we use Proposition \ref{prop:ordertypeB}.
\end{proof}
The zeta function of CyNCZ is notated $\zeta_{\rm CyNCZ}:{\rm CyNCZ}^{(2)}\to\mathbb{C}$ and $\mu_{\rm NCZ}:{\rm CyNCZ}^{(2)}\to\mathbb{C}$ is the Moebius function of {\rm NCZ}. By definition, for any $\pi \in {\rm CyNCZ}$:
\begin{align*}
	\sum_{\substack{\pi^{\prime} \in {\rm CyNCZ}, \\ \pi^{\prime} \leq \pi}} \mu_{\rm CyNCZ}(\pi^{\prime},\pi) = \begin{cases} 0 & \text{ if } \quad\pi^{\prime} \neq \pi, \\
              1 & \text{ if }\quad \pi^{\prime} = \pi.
	\end{cases}
\end{align*}
Let $n,p\geq 1$ integers. Let $I = \{I_1,\ldots,I_n\}$ an interval partition of $[p]$. The partition $I$ yields an interval partition $\bar{I}=\{I_1,I_{-1},\ldots,I_n,I_{-n}\}$ in ${\rm NCZ^{\prime}}(p)$ with $I_{-j} = -I_j$ . Given a partition $\pi\in{\rm CyNCZ}(n)$, we denote by $\hat{\pi}$ the partition of ${\rm CyNCZ}(p)$:
$$
	i,j \in \{-1,\ldots,-p,1\ldots,p\} \quad i\sim_{\hat{\pi}} j \Leftrightarrow i \in I_k,~j\in I_l\,{\rm and}\, k\sim_{\pi}l.
$$
Notice that $\pi \vee \bar{I} \in {\rm CyNCZ}$ whenever $\pi \in {\rm CyNCZ}$ and the balanced block is $\bigcup_{k \in O} I_k \cup \bigcup_{k\in O}-I_k$ with $O$ the balanced block of $\pi$.
\begin{proposition}
	\label{prop:product}
	Let $p\geq 1$ be an integer and let $a_1,\ldots,a_p \in \mathcal{A}$. For any non-crossing partition $\pi \in {\rm NCZ}(n)$, $n\geq 2$:
	\begin{equation}
		[\kappa^{\Psi}, D\kappa^{\Psi}]^{(B)}_\pi(a_{I_1},\dots,a_{I_n})=\sum_{\substack{\sigma \in {\rm CyNCZ}(p),\\ \sigma\vee \bar{I}=\hat{\pi}}} [\kappa^{\Psi},D\kappa^{\Psi}]^{(B)}_{\sigma}(a_1,\dots,a_p)
	\end{equation}
\end{proposition}

\begin{proof} Let $p\geq 1$ be an integer and $a_1,\ldots,a_p \in \mathcal{A}$.
	\begin{align*}
		 & \sum_{\substack{\sigma \in {\rm CyNCZ}(p),                 \\ \sigma\vee \bar{I}=\hat{\pi}}} [\kappa^{\Psi},D\kappa^{\Psi}]^{(B)}_{\sigma}(a_1,\dots,a_p) \\
		 & \hspace{1cm}=\sum_{\substack{\sigma \in {\rm CyNCZ}(p),    \\ \sigma\vee \bar{I}=\hat{\pi}}} \sum_{\substack{\pi^{\prime}\in{\rm CyNCZ}(p),\\\pi^{\prime} \leq \sigma}} \mu_{\rm CyNCZ}(\pi^{\prime},\sigma)  [\Psi,D{\Psi}]^{(B)}_{\pi^{\prime}}(a_1, \dots, a_p) \\
		 & \hspace{1cm}=\sum_{\substack{\pi^{\prime} \in {\rm CyNCZ}, \\\pi^{\prime} \leq \hat{\pi}}} \Big[\sum_{\substack{\sigma \in {\rm CyNCZ}(p),\\ \sigma \geq \pi^{\prime},\\\sigma\vee \bar{I}=\hat{\pi}}} \mu_{\rm Cy NCZ}(\pi^{\prime},\sigma)\Big] [\Psi,D{\Psi}]^{(B)}_{\pi^{\prime}}(a_1, \dots,  a_p).
	\end{align*}
	We set, for any pair of partitions $\pi^{\prime} \leq \hat{\pi}$:
	$$
		\alpha_{\bar{I}}(\pi^{\prime},\hat{\pi}) =  \sum_{\substack{\sigma \in {\rm CyNCZ}(p),\\ \sigma \geq \pi^{\prime},\\\sigma\vee \bar{I}=\hat{\pi}}}\mu_{\rm Cy NCZ}(\pi^{\prime},\sigma)
	$$
	Notice that for $\alpha_I(\pi^{\prime},\hat{\pi})$ not to be zero, we must have $\hat{\pi} \geq \bar{I}\vee \pi^{\prime}$. Assume that $\chi \geq \bar{I}\vee \pi^{\prime}$
	\begin{align*}
		\alpha_{\bar{I}} \star \zeta_{\rm CyNCZ}(\pi^{\prime},\chi)=\sum_{\pi^{\prime} \leq \hat{\pi} \leq \chi} \alpha_I(\pi^{\prime},\hat{\pi}) = \sum_{\bar{I}\vee \pi^{\prime} \leq \hat{\pi} \leq \chi} \sum_{\substack{\sigma \in {\rm CyNCZ}(p), \\ \sigma \geq \pi^{\prime},\\\sigma\vee \bar{I}=\hat{\pi}}}\mu_{\rm Cy NCZ}(\pi^{\prime},\sigma)
	\end{align*}
	The interval $[\bar{I}\vee \pi^{\prime}, \chi]$ is the image of the map $\{\chi \geq \sigma \geq \pi^{\prime}\} \subset {\rm NCZ}\ni \alpha \mapsto \bar{I}\vee \alpha \in {\rm NCZ}$ (which is well-defined) and the second sum is a sum over the fiber of this map at $\hat{\pi}$. Thus,
	$$
		\alpha_{\bar{I}} \star \zeta_{\rm CyNCZ}(\pi^{\prime},\chi) = \delta(\pi^{\prime},\chi).
	$$
	Hence, whenever $\bar{I}\vee \pi^{\prime} > \pi^{\prime}$, one has $\alpha_{\bar{I}} \star \zeta_{\rm NCZ}(\pi^{\prime},\chi) = 0$. Thus $\alpha_{\bar{I}}(\pi^{\prime},\chi)=0$ if $\pi^{\prime} \not\geq \bar{I}$.
	This implies:
	\begin{align*}
		\sum_{\substack{\sigma \in {\rm CyNCZ}(p),           \\ \sigma\vee \bar{I}=\hat{\pi}}} [\kappa^{\Psi},D\kappa^{\Psi}]^{(B)}_{\sigma}(a_1, \dots, a_p) &=
		\sum_{\substack{\pi^{\prime} \in {\rm CyNCZ},        \\ \bar{I} \leq \pi^{\prime} \leq \hat{\pi}}} \alpha(\pi^{\prime},\hat{\pi}) [\Psi,D{\Psi}]^{(B)}_{\pi^{\prime}}(a_1, \dots, a_p) \\
		 & =\sum_{\substack{\pi^{\prime} \in {\rm CyNCZ},    \\\bar{I} \leq \pi^{\prime} \leq \hat{\pi}}} \mu_{\rm CyNCZ}(\pi^{\prime},\hat{\pi}) [\Psi,D{\Psi}]^{(B)}_{\pi^{\prime}}(a_1, \dots, a_p)
		\\
		 & =\sum_{\substack{\pi^{\prime} \in {\rm CyNCZ}(n), \\ \pi^{\prime} \leq \pi}} \mu_{\rm CyNCZ}(\hat{\pi}^{\prime}, \hat{\pi})[\Psi,D{\Psi}]^{(B)}_{\pi^{\prime}}(a_{I_1}, \dots, a_{I_n})
	\end{align*}
	Because $n \not\sim_{\bar{I}} 1$, $D{\Psi}_{\hat{\pi}^{\prime}}(a_1, \dots, a_p)=D\Psi_{\pi^{\prime}}(a_{I_1}, \dots, a_{I_n})$. We are left with proving:
	$$
		\mu_{\rm CyNCZ}(\hat{\pi}^{\prime}, \hat{\pi}) = \mu_{\rm CyNCZ}({\pi}^{\prime}, {\pi}).
	$$
	It comes from the fact that the operator $\hat{~}$ is a monotone bijection between $[\hat{\pi}^{\prime},\hat{\pi}]^{(B)} \subset {\rm CyNCZ}$ and $[{\pi}^{\prime},{\pi}]^{(B)} \subset {\rm CyNCZ}$. The proof is complete.
\end{proof}

The formula is stated for arguments grouped according to an interval partition $I$. The formula in Proposition \ref{prop:product} can be extended to cyclic interval partitions. Suppose $J$ is a cyclic interval partition, let $J_1$ be the block containing $1$, $J_2$ the next block in the lexicographic order and so on. Let $q$ be the first integer such that $c^q(J)$ is an interval partition (where $c$ is the cyclic shift of $[p]$). Then, for any $\pi \in {\rm CyNCZ}(n)$:
$$
	[\kappa^{\Psi},D\kappa^{\Psi}]^{(B)}_{\pi}(a_{(J_1,\prec_R)},\dots,a_{(J_n,\prec_R)}) = [\kappa^{\Psi},D\kappa^{\Psi}]^{(B)}_{\pi}(c^{q}(a)_{c^q(J_1)},\dots,c^{q}(a)_{c^q(J_n)}).
$$
with $R$ the unique block in the Kreweras complement of $J$ which is not a singleton.

We apply Proposition \ref{prop:product} to the right-hand side of the above equality and use the fact that $\sigma \vee c^q(\bar{J}) = \widehat{\pi}$ if and only if $c^{-q}(\sigma) \vee \bar{J} = \widehat{\pi}$ and the cyclic invariance of $[\kappa^{\psi},D\kappa^{\psi}]^{(B)}_{\pi}$ to obtain:
$$
	[\kappa^{\Psi},D\kappa^{\Psi}]^{(B)}_{\pi}(a_{(J_1,\prec_R)},\dots,a_{(J_n,\prec_R)}) = \sum_{\substack{\sigma \in {\rm CyNCZ}(p),\\ \sigma\vee \bar{J}=\hat{\pi}}} [\kappa^{\Psi},D\kappa^{\Psi}]^{(B)}_{\sigma}(a_1,\dots,a_p).
$$
The following proposition is a straightforward consequence of the previous one and the vanishing of mixed cumulants for cyclically free random variables.
\begin{corollary}
	\label{prop:productcy}
	Let $\{a_1,\ldots,a_n \}$ be cyclically free from $\{b_1,\ldots,b_n\}$ with respect to $(\Psi,D \Psi)$. The following formula holds:
	\begin{equation*}
		D\kappa_n^{\Psi}(a_1b_1,\cdots, a_nb_n)= \sum_{\pi \in {\rm CyNC^{(B)}}(n)} [\kappa^{\Psi},D\kappa^{\Psi}]^{(B)}_{\pi}(a_1,\dots, a_n)[\kappa^{\Psi},D\kappa^{\Psi}]_{Kr(\pi)}^{(B)}(b_1,\dots, b_n).
	\end{equation*}
\end{corollary}

\begin{remark}
	If $\{a_1,\ldots,a_n\}$ is infinitesimally free from $\{b_1,\ldots,b_n\}$ relatively to $(\Psi,\partial \Psi)$, \cite[Theorem 6.4]{fevrier2010infinitesimal} yields
	\begin{equation*}
		\partial \kappa_n^{\Psi}(a_1b_1,\dots, a_nb_n)= \sum_{\pi \in {\rm NC^{(B)}}(n)} |\kappa^{\Psi},\partial\kappa^{\Psi}|_{\pi}^{(B)}(a_1,\cdots, a_n)|\kappa^{\Psi},\partial\kappa^{\Psi}|_{Kr(\pi)}^{(B)}(b_1,\cdots, b_n).
	\end{equation*}
\end{remark}

\begin{remark}
	In the more general context of cyclic freeness relatively to a couple $(\psi,\OmegA)$ with $\omega(1_{\mathcal{A}})\neq 0$, one can define the cyclic free cumulants $D\kappa_n^{\psi,\OmegA}$ as in Definition \ref{def:cycfreeness} and the formula from Proposition \ref{prop:ordertypeB} becomes, for any $\pi \in {\rm CyNCZ}(n)$:
	\begin{multline}
		[\Psi, \OmegA]^{(B)}_{\pi}(a_1,\ldots,a_n) + \OmegA(1_{\mathcal{A}})[\Psi, [\Psi]]^{(B)}_{\pi}(a_1,\ldots,a_n) \\
		= \sum_{\substack{\pi^{\prime} \in {\rm CyNCZ}(n) \\ \pi^{\prime} \leq \pi}}[\kappa^{\Psi},D\kappa^{\Psi,\OmegA}]_{\pi'}^{(B)}(a_1, \dots, a_n).
	\end{multline}
	while the other formulae in Proposition \ref{prop:product} and $\ref{prop:productcy}$ remains unchanged upon replacing the $D\kappa^{\Psi}$ with $D\kappa^{\Psi,\OmegA}$.
\end{remark}

\section{Cyclic conditional freeness} \label{sec:cyclicconditionalfreeness}
Let us begin by recalling how cyclic conditional freeness relates to the other cyclic independences. In an augmented unital algebra $\mathcal{A}=\mathbb{C}1_{\mathcal{A}}\oplus \mathcal{A}^0)$
with $\mathcal{A}^0$ a two-sided ideal, equipped with a unital linear functional $\Varphi$ and a tracial linear functional $\OmegA$, two subalgebras $\mathcal{A}_1^0,\mathcal{A}_2^0 \subset {\mathcal{I}}$ are \emph{cyclic Boolean independent} \cite{arizmendi2022cyclic} if
\begin{itemize}
	\item the algebra $\mathcal{A}_1^0$ and $\mathcal{A}_2^0$ are Boolean independent in $(\mathcal{A},\Varphi)$,
	\item for any cyclically alternating sequence of elements $a_1,\ldots,a_n$ in $\mathcal{A}_1^0\cup\mathcal{A}_2^0$,
	      $(n\geq 2)$,
	      $$\OmegA(a_1\cdots a_n)=\Varphi(a_1\cdots a_n).$$
\end{itemize}
Two subalgebras $\mathcal{A}_1\subset \mathcal{I},\mathcal{A}_2 $ are \emph{cyclic monotone independent \cite{collins2018free, arizmendi2022cyclic}} if
\begin{itemize}
	\item the algebras $\mathcal{A}_1$ and $\mathcal{A}_2$ are monotone independent in $(\mathcal{A},\Varphi)$,
	\item for any alternating product $(b_0,a_1,b_1\ldots,a_n,b_n)$ on elements in $\mathcal{A}_1\cup\mathcal{A}_2$, $a's\in\mathcal{A}_1$ and $b's \in \mathcal{A}_2$,
	      $$
		      \OmegA(b_0a_1\cdots a_nb_n)=\OmegA(a_1\cdots a_n)\Varphi(b_1)\cdots \Varphi(b_{n-1})\Varphi(b_nb_0)
	      $$
\end{itemize}
\begin{proposition}[Proposition 3.2.1, 3.2.2 and 3.2.4 of \cite{cebron2022asymptotic}]
	Suppose that $\mathcal{A}=\mathbb{C}1_{\mathcal{A}}\oplus \mathcal{A}^0$ is an augmented unital algebra with $\mathcal{A}^0$ a two-sided ideal. Then,
	\begin{itemize}
		\item If $\Psi(a_1)=\Psi(a_2)=0$ for any $a_1 \in \mathcal{A}_1^0,\,a_2 \in \mathcal{A}_2^0$, then $\mathcal{A}_1^0$ and $\mathcal{A}_2^0$ are cyclic Boolean independent with respect to $(\Varphi, \OmegA)$.
		\item  If $\Psi(a_1)=0,\Psi(a_2)=\Varphi(a_2)$ for any $a_1\in \mathcal{A}_1^0,\, a_2\in\mathcal{A}_2$  then $\mathcal{A}_1^0$ and $\mathcal{A}_2$ are cyclic monotone independent with respect to $(\Varphi, \OmegA)$.
		\item  If $\Psi(a_1)=\Varphi(a_1)$ and $\Psi(a_2)=\Varphi(a_2)$ for any $a_1\in \mathcal{A}_1,\,a_2\in\mathcal{A}_2$ then $\mathcal{A}_1$ and $\mathcal{A}_2$ are cyclically free with respect to $(\Varphi, \OmegA)$.
	\end{itemize}
\end{proposition}
\subsection{Proof of Theorem \ref{th:fromfreetocyclic}} \label{sec:conditionaltocyclicconditional}
We are now in position to prove Theorem \ref{th:fromfreetocyclic}.
A generic element in $T(\mathcal{A})$ will be denoted $\overset{\rightarrow}{a}$.
Set $\tilde{\Varphi}:=[\Varphi]-\Varphi$. It is clear that $[\Varphi]$ is tracial; to prove cyclic conditional freeness it is sufficient to prove that, for any cyclically alternating sequence $\vec{a}_1,\ldots,\vec{a}_N$ of elements in $T(\mathcal{A}_1)\cup T(\mathcal{A}_2)$ centered with respect to $\Psi$,
$$
	[\Varphi](\vec{a}_1\otimes \cdots \otimes \vec{a}_N) =\Varphi(\vec{a}_1\otimes \cdots \otimes \vec{a}_N),$$
or equivalently, that
$$
	\tilde{\Varphi}(\vec{a}_1\otimes \cdots \otimes \vec{a}_N) =0.$$
The linear functional $\tilde{\Varphi}$ descends to a linear functional on the quotient of $T(A)$ by the ideal generated by the elements
$$
	\vec{a}\otimes 1_\mathcal{A}\otimes \vec{b }-\vec{a}\otimes \vec{b},\quad \vec{a},\vec{b}\in T(\mathcal{A}).
$$
We use again the symbol $\tilde{\Varphi}$ for this linear functional.
By linearity, we can assume that each vector $\vec{a}_k$ is a pure tensor (not containing $1_{\mathcal{A}}$ as a factor) in some $\mathcal{A}_{j(k)}^{\otimes n_k}$, centred with respect to $\Psi$.
We may write:
$$
	\vec{a}_1\otimes \cdots \otimes \vec{a}_N=x_1\otimes \cdots \otimes x_n,
$$
with $\vec{a}_k=\otimes_{i\in I_k}x_i$ for a certain interval partition $[n]=I_1\cup \cdots \cup I_N$, $n\geq 2$,  chosen in such a way so that $i\in I_k \Rightarrow x_i\in \mathcal{A}_{j(k)}$ with $j(1),\ldots,j(k)\in \{1,2\}$ and $j(k)\neq j(k+1)$ for $1\leq k \leq N$ .
With the convention $i_{N+1}=i_1$ and indices written modulo $n$, we observe that
$$
	\Varphi(x_1\cdots x_n)=\sum_{\substack{1 \leq \ell \leq n \\ 1= i(1)<i(2)\ldots <i(\ell)\leq  n}}\ \ \prod_{k=1}^\ell\beta^{\Varphi}(x_{i(k)}, x_{i(k)+1}, \dots , x_{i(k+1)-1}).
$$
Note that $i(1)=1$ in the above sum. From this equality, we infer
\begin{align*}
	\tilde{\Varphi}(\vec{a}_1\otimes \cdots \otimes \vec{a}_N)=\tilde{\Varphi}(x_1\otimes \ldots \otimes x_n)
	 & =[\Varphi](x_1\otimes \ldots \otimes x_n)-\Varphi(x_1\cdots x_n)                                                 \\
	 & =\sum_{\substack{1 \leq \ell \leq n-1                                                                            \\ 1< i(1)<i(2)\ldots <i(\ell)\leq n}}\ \ \prod_{k=1}^\ell\beta^{\Varphi}(x_{i(k)}, x_{i(k)+1}, \ldots,  x_{i(k+1)-1}) \\
	 & =\sum_{\substack{1<i(1)< i(2)<n}}\Varphi(x_{i(1)}\cdots x_{i(2)-1}) \beta^\Varphi(x_{i(2)}, \ldots , x_{i(1)-1}) \\
	 & =\sum_{1\leq i<j \leq n}\beta^\Varphi(\underline{x_{j}, x_{j+1}}, \ldots,x_{i})\Varphi(x_{i+1}\cdots x_{j-1}),
\end{align*}
The result follows from Lemma~\ref{lemma:vanishingofbeta}.

\subsection{Proof of Theorem \ref{th:vanishing_of_mixed_cumulants}} \label{sec:prooftheoremvanishingmixedcumulants}
Under the notations of Theorem \ref{th:vanishing_of_mixed_cumulants}, if $\mathcal{A}_1$ and $\mathcal{A}_2$ are cyclically conditionally free then $T(\mathcal{A}_1)$ and $T(\mathcal{A}_2)$ are cyclically conditionally free in $T(\mathcal{A})$ relatively to the triple $(\Psi,\Varphi,[\Varphi])$. The algebras $T(\mathcal{A}_1)$ and $T(\mathcal{A}_2)$ are also cyclically conditionally free relatively to (the extension of) $\Psi,\Varphi,\OmegA$.
Now, since the moment conditions for cyclic conditional freeness involving $\OmegA$ are linear in $\OmegA$, we see that $T(\mathcal{A}_1)$ and $T(\mathcal{A}_2)$ are cyclically free relatively to $(\Psi,\Phi,\OmegA-[\phi])$ and further, from Corollary \ref{lemma:fromfreetocyclic}, $T(A_1)$ and $T(A_2)$ are cyclically free with respect to
\begin{equation}
	\label{eqn:D}
	D^{\Psi,\Varphi}\OmegA = \OmegA - [\Varphi] - (\OmegA(1_\mathcal{A})-1)[\Psi].
\end{equation}
The $n^{th}$ cyclic conditional cumulant $\cck$ is the restriction to $\mathcal{A}^{\otimes n} \subset T(\mathcal{A})^{\otimes n}$ of the $n^{th}$ cyclic free  cumulant $D\kappa_n^{\Psi, D^{\Psi,\Varphi}\OmegA}$ defined in Definition~\ref{def:cycfreecumulants} relatively to the pair $(\Psi, D^{\Psi,\Varphi}\OmegA)$. In particular, the equivalence between (2) and (3) is a consequence of Theorem~\ref{th:vanishingcyclicinf}. It remains to prove the equivalence between (1) and (2).

$(1)\Rightarrow (2)$. We have already proved this implication.

$(2)\Rightarrow (1)$. Let us denote by $\mathcal{B}$ the algebra generated by $\mathcal{A}_1$ and $\mathcal{A}_2$ in $\mathcal{A}\subset T(A)$, and by $\mathcal{C}$ the algebra generated by $\mathcal{A}_1$ and $\mathcal{A}_2$ in $T(\mathcal{A})$. Using \eqref{eqn:D}, $\OmegA$ is completely determined on $\mathcal{B}$ by $(\Psi,\Varphi,D^{\Psi,\Varphi}\OmegA|_{\mathcal{C}})$. On the other hand, whenever $\mathcal{A}_1$ and $\mathcal{A}_2$ are cyclically free in $(T(\mathcal{A}),\Psi, D^{\Psi,\Varphi}\OmegA)$, the linear map $D^{\Psi,\Varphi}\OmegA$ is completely determined on $\mathcal{C}$ by the pairs $(\Psi|_{T(\mathcal{A}_1)},D^{\Psi,\Varphi}\OmegA|_{T(\mathcal{A}_1)})$ and $(\Psi|_{T(\mathcal{A}_2)},D^{\Psi,\Varphi}\OmegA|_{T(\mathcal{A}_2)})$, or equivalently by $(\Psi|_{\mathcal{A}_1},\Varphi|_{\mathcal{A}_1},\OmegA|_{\mathcal{A}_1})$ and $(\Psi|_{\mathcal{A}_2},\Varphi|_{\mathcal{A}_2},\OmegA|_{\mathcal{A}_2})$. Consequently, the linear functional $\OmegA$ is completely determined on $\mathcal{B}$ by $(\Psi|_{\mathcal{A}_1},\Varphi|_{\mathcal{A}_1},\OmegA|_{\mathcal{A}_1})$ and $(\Psi|_{\mathcal{A}_2},\Varphi|_{\mathcal{A}_2},\OmegA|_{\mathcal{A}_2})$. As wanted, whenever $\mathcal{A}_1$ and $\mathcal{A}_2$ are cyclically free in $(T(\mathcal{A}),\Psi, D^{\Psi,\Varphi}\OmegA)$, the joint distribution of $\mathcal{A}_1$ and $\mathcal{A}_2$ with respect to $(\Psi,\Varphi,\OmegA)$ is the same as two cyclically conditionally free algebras.
\subsection{Cumulants for cyclic Boolean independence}
\label{sec:cyclicBoolean}
Cyclic Boolean independence between the ideals $\mathcal{A}_1^0$ and $\mathcal{A}_2^0$ is equivalent to cyclic conditional freeness between the ideals $\mathcal{A}_1^0\oplus \mathbb{C}$ and $\mathcal{A}_2^0\oplus \mathbb{C}$ when $\mathcal{A}_1^0,\mathcal{A}_2^0\subset \ker(\Psi)$. For any $a_1,\ldots,a_n \in \mathcal{A}^0 \subset {\rm ker}(\Psi)$, the moment-cumulant relation defining the cyclic conditional cumulants yield
\begin{equation}
	\label{eqn:ckboolean}
	\OmegA(a_1 \cdots  a_n)=\sum_{1\leq i_1<\ldots <i_\ell\leq n}\ \ \prod_{k=1}^\ell\beta^{\Varphi }(a_{i_k}, \dots, a_{i_{k+1}-1})
	+\cck(a_1,\dots,a_n).
\end{equation}
If defining, for any $n\geq 1$ and $a_1,\ldots,a_n \in \mathcal{A}^0$,
$$
	c^{\Varphi,\OmegA}_n(a_1,\ldots, a_n):= \cck(a_1,\ldots,a_n)+\sum_{i=1}^n\beta^\Varphi(a_i, a_{i+1},\ldots, a_n, a_1,\ldots, a_{i-1}),
$$
the equation \eqref{eqn:ckboolean} yields
\begin{equation} \label{eqn:momentcumulantsboolean}
	\OmegA(a_1 \cdots  a_n)=c^{\Varphi,\OmegA}_n(a_1, \ldots, a_n)+\sum_{\substack{\ell \geq 2 \\ 1\leq i_1<\ldots <i_\ell\leq n}}\ \ \prod_{k=1}^\ell\beta^{\Varphi }(a_{i_k},\dots, a_{i_{k+1}-1}),
\end{equation}
which means that $c_n, n\geq 1$ are the cyclic Boolean cumulants \cite{arizmendi2022cyclic}.
We now state a formula for the cyclic Boolean cumulants of the pair $(\Psi,[\Psi])$ in terms of the Boolean cumulants of $\Psi$.
To state the following proposition, we will introduce new notations to distinguish $\Psi$ from its extension over $T(\mathcal{A})$:
$$
	\vec{\Psi}: T(\mathcal{A})\to\mathbb{C},\quad \vec{\Psi}(a_1\otimes\cdots\otimes a_n)=\Psi(a_1\cdots a_n).
$$
Let $N\geq 1$.
We use $I[i_1,\ldots,i_p]$, where $1 \leq i_1 < \cdots < i_p \leq N$, for the partition into intervals $\{ \llbracket i_j,i_{j+1} \llbracket,~ j\in \{0,\ldots,p\}\}$ with the convention $i_0 =1$, $i_{p+1}=N$.
Given $\vec{x}_i \in T(A)$ pure tensors, we denote $I[\vec{x}_1,\ldots,\vec{x}_p]$ the partition $I[i_1,\ldots,i_p]$ where $i_j=|\vec{x}_1| + \cdots + |\vec{x}_j|+1$.
The Boolean cumulants of $\vec{\Psi}$ are easily seen to satisfy (see for example \cite{gilliers2026bigraph}), for any $\vec{x}_1,\ldots,\vec{x}_p \in T(\mathcal{A}^0)$:
$$
	\beta_{J}^{\vec{\Psi}}(\vec{x}_1, \dots, \vec{x}_p)=\sum_{\substack{I \in {\rm Int}([N]):\\ I\vee I[\vec{x}_1,\ldots,\vec{x}_p]=\hat{J}}}\prod_{\{i_1,\ldots,i_k\}\in I}\beta^{\Psi}(a_{i_1},\ldots,a_{i_k}).
$$
for any interval partition $J$ of $\{1,\ldots,p\}$,
where $x_1\otimes\cdots\otimes x_p = a_1\otimes\cdots\otimes a_N$, $a_i \in \mathcal{A}$ and $x_i \in T(\mathcal{A})$ and $\hat{J}$ is the partition on $\{1,\ldots,N\}$ induced by $J$ ($i \sim_{\hat{J}} j$, $i,j\in\{1,\ldots,N\}$ if and only if $i,j$ belongs to two blocks of $I[\vec{x}_1,\ldots,\vec{x}_p]$ equivalent under $J$).

If instead $J$ is a partitions into a \emph{cyclic} intervals, we can rotate it by applying to $J$ a power $q$ of the cyclic shift $c=(1,\ldots,p)$ to obtain an interval partition $\tilde{J}$. Since $\beta_{\tilde{J}}(\vec{x}_{c^q(1)}, \cdots , \vec{x}_{c^q(p)})=\beta_{J}(\vec{x}_1,\cdots ,\vec{x}_p)$, we obtain that, for any cyclic interval partition $J$ of $\{1,\ldots,p\}$:
\begin{equation} \label{eqn:productcycumulantsboolean}
	\beta_{J}^{\vec{\Psi}}(\vec{x}_1,\dots, \vec{x}_p)=\sum_{\substack{I \in {\rm CyInt}([N]):\\ I\vee I[\vec{x}_1,\ldots,\vec{x}_p]=\hat{J}}}\prod_{\{i_1,\ldots,i_k\}\in I}\beta^{\Psi}(a_{i_1},\ldots,a_{i_k}).
\end{equation}
Given a cyclic interval partition $I$ with at least two blocks, we write
$$
	\beta^{\Psi}_I(a_1,\cdots, a_n)= \prod_{\{i_1 \prec_R\ldots \prec_R i_k\} \in I}\beta^{\Psi}(a_{i_1},\ldots,a_{i_k})
$$
where $R$ is the only non-trivial block in the Kreweras complement of $I$. With this notation, equation \eqref{eqn:momentcumulantsboolean} reads
$$
	\OmegA(a_1 \cdots  a_n)=c^{\Varphi,\OmegA}_n(a_1, \ldots, a_n)+\sum_{I \neq \mathbf{1}_n \in {\rm CyInt}(n)}\beta^{\Psi}_I(a_1,\cdots, a_n),\quad a_1,\ldots,a_n \in \mathcal{A}^0.
$$

\begin{proposition}
	The cyclic Boolean cumulants $c_n^{(\Psi,[\Psi])}$,~$n\geq 1$ of the pair $(\Psi, [\Psi])$ satisfy
	\begin{equation}
		c_n^{\Psi, [\Psi]}(\vec{x}_1, \dots, \vec{x}_p)= \sum_{i=1}^n \beta^{\Psi}_n(a_i,\cdots, a_N , a_1 , \cdots , a_{i-1})+\!\!\!\!\!\!\!\!\!\!\sum_{\substack{I \in{\rm CyInt}(N)\\I \vee I[\vec{x}_1,\ldots,\vec{x}_p] = \mathbf{1}_N \\|I|\geq 2}}\!\!\!\!\!\!\!\!\!\!\!\!\!\! \beta_I^{\Psi}(a_1 ,\dots, a_N)
	\end{equation}
	for any $\vec{x_1},\ldots,\vec{x_N} \in T(\mathcal{A}^0)$.
	In particular, for any $n\geq 1$ and $a_1,\ldots,a_n \in \mathcal{A}^0$ ($x_1 = a_1 \otimes \cdots \otimes a_n$ in the previous equality):
	$$
		c_n^{(\Psi,[\Psi])}(a_1,\dots, a_n)=\sum_{i=1}^n \beta^{\Psi}_n(a_i,\dots, a_n , a_1 , \dots , a_{i-1}).
	$$
\end{proposition}
\begin{proof}
	Under the notation of the proposition, we have
	$$
		[\Psi](\vec{x}_1 \cdots \vec{x}_p)=\sum_{i=1}^n \beta^{\Psi}_n(a_i,\dots, a_n , a_1 , \dots , a_{i-1}) + \sum_{I \in {\rm CyInt}([N])} \beta_I^{\Psi}(a_1,\ldots,a_N)
	$$
	The proof follows by dividing the sum over cyclic interval partitions in the right-hand side of the above equation for $[\Psi](\vec{x}_1,\ldots,\vec{x}_p)$ into two parts: the contributions from the partition $\mathbf{1}_N$ and the contributions from non-trivial partitions and using the relations \eqref{eqn:productcycumulantsboolean}.
\end{proof}

\section{Convolutions and limit theorems}
\label{sec:convolutions and limit theorems}
\subsection{Cyclic conditional additive convolution}
\label{sec:cyclicconditionaladditiveconvolution}
In this section, we give functional equations between (formal) Cauchy transforms for computing the distribution of the sum of cyclically conditionally free random variables.
We denote by $\mathbb{C}((z))$ the space of formal Laurent series in $z$.
Given $a \in \mathcal{A}$, we define the (formal) Cauchy transform of the distribution of $a$ as $G_a^{\Psi} \in \mathbb{C}((z))$ (resp. $G_a^{\Varphi}\in \mathbb{C}((z))$) relatively to $\Psi$ (resp. relatively to the state $\Varphi$) by,
$$
	G_a^{\Psi}(z)=\frac{1}{z}+\sum_{n=1}^\infty \frac{\Psi(a^n)}{z^{n+1}} \in\mathbb{C}((z)),\quad (\textrm{resp. } G^{\Varphi}_a(z)=\frac{1}{z}+\sum_{n=1}^\infty \frac{\Varphi(a^n)}{z^{n+1}}) \in\mathbb{C}((z))).
$$
The Cauchy transform of $a$ with respect to $\OmegA$ is
$$
	G^{\OmegA}_a(z)=\frac{\OmegA(1_{\mathcal{A}})}{z} + \sum_{n\geq 1} \frac{\OmegA(a^{n})}{z^{n+1}} \in \mathbb{C}((z)).
$$
Lemma \ref{lem:mk} confirms that $[\Psi]$ can be considered as a multivariate extension of the inverse Markov--Krein transform. This is a consequence of Proposition \ref{prop:MKt} and the formula defining the Markov--Krein transform, which can be found in \cite{kerov1997interlacing}. Nevertheless, we provide a proof, by deriving the relation between  Cauchy transforms from the combinatorial definition we gave of $[\Varphi]$.
\begin{lemma}
	\label{lem:mk}
	Let $a \in \mathcal{A}$. The Cauchy transform of $a$ with respect to $[\Psi]$ is related to the Cauchy transform of $a$ with respect to $\Psi$ by,
	\begin{equation}
		\label{eqn:cauchyw}
		G^{[\Psi]}_a(z)=-\frac{(G_a^{\Psi}(z))'}{G_a^{\Psi}(z)},
	\end{equation}
	where $G_a^{[\Psi]}(z)=\frac{1}{z}+\sum_{n=1}^\infty \frac{[\Psi](a^{\otimes n})}{z^{n+1}}$.
\end{lemma}
\begin{remark} We define the moment generating function of $a$ with respect to $[\Psi]$ by,
	$$
		M^{[\Psi]}(z) = \sum_{n \geq 1} [\Psi](a^{\otimes n })z^n
	$$
	Since $M^{[\Psi]}(z)=\frac{1}{z}G^{[\Psi]}(\frac{1}{z})-1$, the above relation between the Cauchy transforms of $a$ relatively to $[\Psi]$ and relatively to $\Psi$ yields:
	$$
		M^{[\Psi]}(z)=z\frac{(M^{\Psi})'(z)}{1+M^{\Psi}(z)}.
	$$
	In fact,
	\begin{align*}
		G^{[\Psi]}_a(z)=-{\rm \ln}(G^{\Psi}_a)'(z)=-\frac{d}{dz}{\rm \ln}(\frac{1}{z}(M^{\Psi}(\frac{1}{z})+1))=\frac{1}{z}+\frac{(M^{\Psi})^{\prime}(\frac{1}{z})}{z^2(1+M^{\Psi}(\frac{1}{z}))}
	\end{align*}
	and the desired formula follows.
\end{remark}
\begin{proof}
	We observe first that, for $n\geq 1$ and $a \in \mathcal{A}$,
	\begin{align*}  [\Psi](a^n) & =\sum_{\ell=1}^n\sum_{1\leq i(1)<\ldots <i(\ell)\leq  n}\ \ \prod_{k=1}^{\ell-1}\beta^{\Psi}_{i(k+1)-i(k)}(a)\cdot \beta^{\Psi}_{i(1)+n-i(\ell)}(a) \\
                            & =\sum_{1\leq i \leq j \leq n}\Psi(a^{j-i})\cdot \beta^{\Psi}_{i+n-j}(a)
                \\
                            & =\sum_{k=0}^{n-1}\sum_{\substack{1\leq i \leq j \leq n                                                                                              \\j-i=k}}\Psi(a^{k})\cdot \beta^{\Psi}_{n-k}(a)
                =\sum_{k=0}^n(n-k)\cdot\Psi(a^{k})\cdot \beta^{\Psi}_{n-k}(a),
	\end{align*}
	The above equality yields
	\begin{align*}
		G_a^{[\Psi]}(z)-\frac{1}{z} & =\sum_{n=1}^\infty \frac{[\Psi](a^n)}{z^{n+1}}
		\\
		                            & =\sum_{n=1}^\infty \frac{1}{z^{n+1}}\sum_{k=0}^n(n-k)\cdot\Psi(a^{k})\cdot \beta^{\Psi}_{n-k}(a)        \\
		                            & =\sum_{k=0}^\infty \sum_{n=k+1}^\infty\frac{1}{z^{n+1}}(n-k)\cdot\Psi(a^{k})\cdot \beta^{\Psi}_{n-k}(a) \\
		                            & =\sum_{k=0}^\infty \sum_{m=1}^\infty\frac{1}{z^{k+m+1}}m\cdot \Psi(a^{k})\cdot \beta^{\Psi}_{m}(a)      \\
		                            & =zG^\Psi_a(z)\cdot \left( \sum_{m=1}^\infty\frac{m}{z^{m+1}}\beta^{\Psi}_{m}(a)\right)                  \\
		                            & =zG^\Psi_a(z)\cdot \left( -\sum_{m=1}^\infty\frac{1}{z^{m}}\beta^{\Psi}_{m}(a)\right)'                  \\
		                            & =zG^\Psi_a(z)\cdot \left(\frac{1}{zG^{\Psi}_a(z)}-1\right)'                                             \\
		                            & =zG^\Psi_a(z)\cdot \frac{-G^{\Psi}_a(z)-z(G_a^{\Psi}(z))'}{(zG^{\Psi}_a(z))^2}                          \\
		                            & =-\frac{1}{z}-\frac{(G_a^{\Psi})'(z)}{G_a^{\Psi}(z)}.
	\end{align*}
\end{proof}
Recall the definition of $D^{\Psi,\Varphi}\OmegA$ : $D^{\Psi,\Varphi}\OmegA = \OmegA - [\Varphi] + (1-\OmegA(1_\mathcal{A}))[\Psi]$.
\begin{lemma} \label{lem:relationcauchytr}
	Let $a\in\mathcal{A}$. Then the Cauchy transform of $a$ with respect to $D^{\Psi,\Varphi}\OmegA$ is related to the other Cauchy transforms of $a$ (relatively to $\Psi$, $\Varphi$ and $\OmegA$) as formal series by:
	\begin{align}
		\label{eqn:relationcauchytr}
		G^{D^{\Psi,\Varphi}\OmegA}_a(z)=G^{\OmegA}_a(z)-(1-\OmegA(1_\mathcal{A}))\frac{(G_a^{\Psi})'(z)}{G_a^{\Psi}(z)}+\frac{(G_a^{\Varphi})'(z)}{G_a^{\Varphi}(z)}
	\end{align}
\end{lemma}
\begin{proof}This follows from the additivity of the Cauchy transform with respect to the linear functional used to compute it,  the formula for $D^{\Psi,\Varphi}\OmegA = \OmegA - [\Varphi] - (1-\OmegA(1_\mathcal{A}))[\Psi]$ and Lemma \ref{lem:mk}.
\end{proof}
The Cauchy transforms with respect to $\Psi$ and $\Varphi$ are invertible for the compositions of Laurent series. Let $a,b \in \mathcal{A}$ be free with respect to $\psi$ and set
$$
	W_a (z) := (G_{a}^{\Psi})^{\langle -1 \rangle }\circ G_{a+b}^{\Psi}(z),\quad W_b (z)  := (G_{b}^{\Psi})^{\langle -1 \rangle}\circ G_{a+b}^{\Psi}(z),
$$
where $(-)^{\langle-1\rangle}$ denotes the inverse of the corresponding Cauchy transform for composition. We call the series $W$'s the subordination series. Suppose that $a$ and $b$ are cyclically conditionally free. We have seen that $a$ and $b$ are cyclically free relatively to $\Psi$ and $D^{\Psi,\Varphi}\OmegA$ (see Section \ref{sec:prooftheoremvanishingmixedcumulants}). But since $\Psi$ restricted to the algebra generated by $a$ or $b$ is tracial, $a$ and $b$ are in fact infinitesimally free. From \cite{belinschi2012free}, as Laurent series,
$$
	G_{a+b}^{D\OmegA}(z)=G_{a}^{D\OmegA}(W_a(z)) W'_a(z)+G_{b}^{D\OmegA}(W_b(z))W'_b(z),
$$
where we have set $D\OmegA = D^{\Psi,\Varphi}\OmegA$ for brevity.
By using the Lemma \ref{lem:relationcauchytr}, we get
\begin{align*}G_{a}^{D\OmegA}(W_a(z))W_a'(z)
	 & =G^{\OmegA}_a(W_a(z))W'_a(z)-(1-\OmegA(1_\mathcal{A}))\frac{(G_a^{\Psi})'(W_a(z))}{G_a^{\Psi}(W_a(z))}W_a'(z)+\frac{(G_a^{\Varphi})'(W_a(z))}{G_a^{\Varphi}(W_a(z))}W_a'(z) \\
	 & =G^{\OmegA}_a(W_a(z))W_a'(z)-(1-\OmegA(1_\mathcal{A}))\frac{(G_{a+b}^{\Psi})'(z)}{G_{a+b}^{\Psi}(z)}+\frac{(G_a^{\Varphi})'(W_a(z))}{G_a^{\Varphi}(W_a(z))}W_a'(z),
\end{align*}
from which we deduce the following proposition.
\begin{proposition}
	\label{prop:subordinationcyccond}
	Let $a$ and $b$ be cyclically conditionally free relatively to $(\Psi,\Varphi,\OmegA)$. The following relation between Laurent series holds:
	\begin{align*}
		G_{a+b}^{\OmegA}(z)= & G^{\OmegA}_a(W_a(z))W_a'(z)+G^{\OmegA}_b(W_b(z))W_b'(z)-(1-\OmegA(1_\mathcal{A}))\frac{(G_{a+b}^{\Psi})'(z)}{G_{a+b}^{\Psi}(z)}                                                    \\
		                     & +\frac{(G_a^{\Varphi})'(W_a(z))}{G_a^{\Varphi}(W_a(z))}W_a'(z)+\frac{(G_b^{\Varphi})'(W_b(z))}{G_b^{\Varphi}(W_b(z))}W_b'(z)-\frac{(G_{a+b}^{\Varphi})'(z)}{G_{a+b}^{\Varphi}(z)}.
	\end{align*}
\end{proposition}
\begin{remark}
	We specialize the above relations to two cases. We first consider cyclic Boolean independence.
	We assume that $\mathcal{A}=\mathbb{C}\oplus \mathcal{I}$ with $I={\rm ker}{\Psi}$ a two-sided ideal and assume further $\OmegA(1_{\mathcal{A}})=0$. For any $a,b \in \mathcal{I}$, $G^{\Psi}_a=G^{\Psi}_b=G^{\Psi}_{a+b}(z)=\frac{1}{z}$. Thus, the subordination series satisfy $W_{a}(z)=W_{b}(z)=z$ and one deduces the following relation from Proposition \ref{prop:subordinationcyccond}:
	\begin{align*}
		G^{\OmegA}_{a+b}(z)=G^{\OmegA}_a(z)+G^{\OmegA}_{b}(z)+\frac{(G_a^{\Varphi})^{\prime}(z)}{G_a^{\Varphi}(z)}+\frac{(G^{\Varphi}_{b})^{\prime}(z)}{G^{\Varphi}_b(z)}-\frac{(G^{\Varphi}_{a+b})^{\prime}(z)}{G^{\Varphi}_{a+b}(z)} + \frac{1}{z}.
	\end{align*}
	This coincides with Theorem 4.2 in \cite{arizmendi2022cyclic} (with $\tilde{\mathfrak{g}}_a(z)=G^{\OmegA}_a(z)$).

	We move on to the cyclic monotone case; the distribution of $a$ under $\psi$ is trivial and the distribution of $b$ under $\varphi$ is equal to its distribution under $\psi$:
	\begin{align*}
		 & G^{\Psi}_a(z)= \frac{1}{z},\quad G^{\Psi}_b(z)=G^{\Varphi}_b(z),\quad G^{\Psi}_{a+b}(z)=G^{\Varphi}_b(z),\quad G^{\Varphi}_{a+b}(z)= G^{\Varphi}_a(F^{\Varphi}_b(z)) \\
		 & W_a(z)=F^{\Varphi}_b(z),\quad W_b(z)=z,
	\end{align*}
	with $\OmegA(1_{\mathcal{A}})=0$ and $F_b^{\Varphi}:=\frac{1}{G_b^{\Varphi}(z)}$:
	$$
		G^{\OmegA}_a(W_a(z))W^{\prime}_{a}(z)+G^{\OmegA}_b(W_b(z))W^{\prime}_b(z) = G_a^{\OmegA}(F_b^{\Varphi}(z))(F_b^{\Varphi}(z))^{\prime}+G_b^{\OmegA}(z),
	$$
	and
	\begin{align*}
		 & -(1-\OmegA(1_\mathcal{A}))\frac{(G^\Psi_{a+b})'(z)}{G_{a+b}^{\Psi}(z)}+\frac{(G_a^\Varphi)'(W_a(z))}{G_a^\Varphi(W_a(z))}W_a'(z)+\frac{(G_b^{\Varphi})'(W_b(z))}{G_b^\Varphi(W_b(z))}W_b'(z)-\frac{(G_{a+b}^{\Varphi})'(z)}{G^{\Varphi}_{a+b}(z)}                                                        \\
		 & \hspace{2cm}=-\frac{(G^{\Psi}_{b})^{\prime}(z)}{G^{\Psi}_b(z)}+\frac{G_a^{\Varphi}(z)}{G_a^{\Varphi}(z)}(F_b^{\Varphi}(z))(F_b^{\Varphi}(z))^{\prime}+\frac{(G_b^{\Varphi})^{\prime}(z)}{G_b^{\Varphi}(z)}-\frac{(G^{\Varphi}_{a} \circ F^{\Varphi}_b)^{\prime}(z)}{G_a^{\Varphi}\circ F_b^{\Varphi}(z)} \\
		 & \hspace{2cm}=0.
	\end{align*}
	We obtain finally:
	\begin{equation*}
		G^{\OmegA}_{a+b}(z)=G_a^{\OmegA}(F^{\Varphi}_b(z))(F_b^{\Varphi})^{\prime}(z)+G^{\OmegA}_b(z).
	\end{equation*}
	The above formula coincides with Theorem 7.6 of \cite{arizmendi2022cyclic}.
\end{remark}

\subsection{Cyclic conditional multiplicative convolution}
\label{sec:productcyclicallyfreerandomvariables}
The following proposition is a direct corollary of Proposition \ref{prop:productcy} and the definition of the cyclic conditional $\cck$:
\begin{proposition}
	\label{prop:productccfree}
	Let $(a_1,\ldots,a_n)$ and $(b_1,\ldots, b_n)$ be cyclically conditionally free random variables, then:
	$$
		\cck(a_1b_1, \dots ,a_nb_n)=\sum_{\pi\in \cync^{(B)}(n)}[\kappa^{\Psi},\kappa^{\OmegA\,|\,\Psi,\Varphi}]^{(B)}_\pi(a_1, \dots, a_n)\cdot [\kappa^{\Psi},\kappa^{\OmegA\,|\,\Psi,\Varphi}]^{(B)}_{Kr(\pi)}(b_1,\dots, b_n)
	$$
\end{proposition}

We now want to write subordination relations to compute the distributions (with respect to $\Psi, \Varphi, \OmegA$) of the product $xy$ of two cyclic conditionally random variables $x$ and $y$.
We will prove first subordination formulae for the $\eta$ transforms, respectively the infinitesimal $\eta$-transform (both transforms are introduced hereafter) of the random variables $xy$ when $x$ and $y$ are conditionally free and $x$ and $y$ are infinitesimally free, respectively. For $x\in\mathcal{A}$, We define the following formal series
\begin{align*}
	M^{\Varphi}_x(z):=\Varphi(\frac{zx}{1-zx})=\sum_{n\geq 1} \Varphi(x^{n})z^n,\quad \eta^{\Varphi}_x(z)=\frac{M^{\Varphi}_x(z)}{1+M^{\Varphi}_x(z)}.
\end{align*}
If $\Varphi(x) = 1 $, $M^{\Varphi}_x$ is invertible for the composition of series. In this case, the $S$-transform of $x$ is defined by \cite{voiculescu1987multiplication}:
\begin{align*}
	\quad S^{\Varphi}_x(z) := \frac{1+z}{z} (M^{\Varphi}_x)^{\langle -1 \rangle}(z),
\end{align*}
where we use the superscript $\langle-1\rangle$ for the compositional inverse. We recall from \cite{voiculescu1987multiplication} the following, when $x$ and $y$ are free random variables:
$$
	S^{\Varphi}_{xy}(z)=S^{\Varphi}_{x}(z)S^{\Varphi}_y(z).
$$
The author in \cite{tseng2022operator} introduced the following \emph{infinitesimal} transforms, to study the infinitesimal multiplicative convolution:
\begin{align*}
	\partial^{\Varphi} M_x(z):= \Varphi'(\frac{1}{1-zx}),\quad\quad \partial^{\Varphi} \eta_x(z):=\frac{\partial^{\Varphi} M_x(z)}{(M^{\Varphi}_x(z)+1)^2},~
\end{align*}
and finally, when $\Varphi(x)=1$, we set :
\begin{align}
	\label{def:infinitesimalS}
	\partial^{\Varphi} S_x (z):=-(\partial^{\Varphi} M_x \circ (M_x^{\Varphi}{})^{\langle -1 \rangle})(z)\frac{d}{{d}z}(M_x^{\Varphi}{})^{\langle-1\rangle}(z).
\end{align}
The infinitesimal $S$-transform $\partial^{\Varphi} S$ has the property that when $x$ and $y$ are infinitesimally free \cite{tseng2022operator}:
\begin{align*}
	\partial S^{\Varphi}_{xy}(z)=\partial S^{\Varphi}_x(z) S^{\Varphi}_y(z) + S^{\Varphi}_x(z)\partial S^{\Varphi}_y(z).
\end{align*}
We denote by $\tilde{W}^{\Varphi}_x,\tilde{W}^{\Varphi}_y \in \mathbb{C}((z)) z $ the \emph{subordination series for free multiplicative convolution} :

$$
	\tilde{W}^{\Varphi}_x(z)=({(\eta^{\Varphi}_x)}^{\langle -1\rangle}\circ\eta^{\Varphi}_{xy})(z),\quad \tilde{W}^{\Varphi}_y(z)=(({\eta^{\Varphi}_{y}})^{\langle -1 \rangle}\circ\eta^{\Varphi}_{xy})(z).
$$
\begin{remark}
	The coefficients of these two series can be expressed in terms of the Boolean cumulants
	\cite{lehner2021Boolean}:
	$$
		\tilde{W}^{\Varphi}_x(z)=\sum_{n\geq 0} \beta^{\Varphi}_{2n+1}(x,y,\ldots,y,x)z^{n+1}, ~\tilde{W}^{\Varphi}_y(z)=\sum_{n\geq 0} \beta^{\Varphi}_{2n+1}(y,x,\ldots,x,y)z^{n+1}.
	$$
\end{remark}
The series $\tilde{W}_x^{\Varphi}$ and $\tilde{W}_y^{\Varphi}$ are solutions to fixed-point equations in the space of formal series:
\begin{equation*}
	\tilde{W}^{\varphi}_x(z)=z\rho_x(z\rho_y(\tilde{W}^{\varphi}_x(z))),\quad \tilde{W}^{\varphi}_y(z)=z\rho_y(z\rho_x(\tilde{W}^{\varphi}_y(z))),
\end{equation*}
with the $\rho$-series $\rho_x$ of a variable $x\in \mathcal{A}$ defined as:
$$
	\rho^{\Varphi}_x(z)=\frac{1}{z}\eta^{\Varphi}_x(z).
$$
They satisfy the following properties, which allow computations of the distributions of $xy$ when $x$ and $y$ are free:
\begin{align*}
	\eta^{\Varphi}_x \circ \tilde{W}^{\Varphi}_x(z) = \eta^{\Varphi}_{xy}(z) = \eta^{\Varphi}_{y} \circ \tilde{W}^{\Varphi}_y(z),\quad
	\tilde{W}^{\Varphi}_x(z)\tilde{W}^{\Varphi}_y(z)=z\eta^{\Varphi}_{xy}(z).
\end{align*}
For Boolean multiplicative convolution, the appropriate transform to consider is the $\rho$-transform \cite{popa2008new}:
\begin{equation}
	\label{eqn:rhoBooleanmult}
	\rho^{\Varphi}_{xy}(z)=\rho^{\Varphi}_x(z)\rho^{\Varphi}_y(z),
\end{equation}
whenever $x-1$ is Boolean independent of $y-1$ in a probability space $(\mathcal{A},\Varphi)$.
For monotone convolution, the following relation holds, whenever $x-1$ is independent from $y$, see \cite{popa2008combinatorial}:
\begin{align*}
	\eta^{\Varphi}_{xy}(z)=\eta^{\Varphi}_x (\eta^{\Varphi}_{y}(z)).
\end{align*}
In the next proposition, set, if $x \in \mathcal{A}$:
$$
	\partial^{\Varphi}\rho_{x} = \frac{1}{z} \partial^{\varphi}\eta_{x}
$$

\begin{proposition}
	\label{prop:subinf}
	Let $x,y$ be infinitesimally free random variables in an infinitesimal probability space $(\mathcal{A},\Varphi,\partial \varphi)$, with $\varphi(x)= \varphi(y)=1$. Then,
	\begin{align*}
		\partial^{\Varphi}\rho_{xy}(z)=(\tilde{W}^{\Varphi}_x)'(z)(\partial^{\Varphi}\rho_x \circ \tilde{W}^{\Varphi}_x)(z) +(\tilde{W}^{\Varphi}_y)'(z)(\partial^{\Varphi} \rho_y \circ \tilde{W}^{\Varphi}_y)(z).
	\end{align*}
\end{proposition}
\begin{proof} One might first consider the case where both $\Varphi(x)$ and $\Varphi(y)$ are invertible. Then, from the definition of $\partial^{\Varphi} S_x$ (see \eqref{def:infinitesimalS}), one gets:
	\begin{equation*}
		-(M^{\Varphi}_x)'(z) (\partial^{\Varphi}S_x \circ M^{\Varphi}_x)(z)= \partial^{\Varphi} M_x(z).
	\end{equation*}
	The above formula yields:
	\begin{equation*}
		\partial^{\Varphi} \eta_x(z)=\frac{\partial^{\Varphi} M_x(z)}{(M^{\Varphi}_x(z)+1)^2}=-\frac{(M^{\Varphi}_x)'(z)}{(M^{\Varphi}_x(z)+1)^2}(\partial^{\Varphi} S_x \circ M^{\Varphi}_x)(z) = K_x(z) (\partial^{\Varphi} S_x\circ M^{\Varphi}_x)(z).
	\end{equation*}
	where we have set:
	\begin{equation*}
		K_x(z) = -\frac{(M^{\Varphi})'_x(z)}{(M^{\Varphi}_x(z)+1)^2}.
	\end{equation*}
	From the following formula for the infinitesimal $S$-transform $\partial^{\Varphi} S_{xy}$ of $xy$:
	\begin{equation*}
		\partial^{\Varphi} S_{xy}(z)=\partial^{\Varphi} S_x(z) S^{\Varphi}_y(z) + S^{\Varphi}_x(z)\partial^{\Varphi} S_y(z),
	\end{equation*}
	we infer
	\begin{equation*}
		\frac{\partial^{\Varphi} \eta^{\Varphi}_{xy}\circ(M_{xy}^{\Varphi})^{\langle -1 \rangle}(z)}{K_{xy} \circ (M^{\Varphi})^{\langle -1 \rangle}_{xy}(z)}=\frac{\partial^{\Varphi} \eta^{\Varphi}_{x}\circ (M_x^{\Varphi})^{\langle -1 \rangle}(z)}{K_{x} \circ (M_{x}^{\Varphi})^{\langle -1 \rangle}(z)} S^{\Varphi}_y(z)+S^{\Varphi}_x(z)\frac{\partial^{\Varphi} \eta^{\Varphi}_{y}\circ(M^{\Varphi}_y)^{\langle -1 \rangle}(z)}{K_{y} \circ (M_{y}^{\Varphi})^{\langle -1 \rangle}(z)},
	\end{equation*}
	and
	\begin{equation*}
		\label{eqn:last}
		\frac{\partial^{\Varphi}{\eta}_{xy}(z)}{K_{xy}(z)}=\frac{\partial^{\Varphi} \eta_x\circ (M^{\Varphi}_x)^{\langle -1 \rangle}\circ M^{\Varphi}_{xy}(z)}{K_x\circ(M^{\Varphi}_x)^{\langle -1 \rangle}\circ M^{\Varphi}_{xy}(z))}S^{\Varphi}_y(M^{\Varphi}_{xy}(z))+S^{\Varphi}_x(M^{\Varphi}_{xy}(z))\frac{\partial^{\Varphi} \eta_y\circ(M^{\Varphi}_y)^{\langle -1 \rangle}\circ M^{\Varphi}_{xy}(z)}{K_y\circ(M^{\Varphi}_y)^{\langle -1 \rangle}\circ M^{\Varphi}_{xy}(z))}.
	\end{equation*}

	yields
	\begin{equation*}
		\frac{\partial^{\Varphi} \eta^{\Varphi}_{xy}(z)}{K_{xy}(z)} = S^{\Varphi}_y(M^{\Varphi}_{xy}(z))\frac{\partial^{\Varphi} \eta^{\Varphi}_{x} \circ \tilde{W}^{\Varphi}_x(z)}{K_x\circ \tilde{W}^{\Varphi}_x(z)}+S^{\Varphi}_x(M^{\Varphi}_{xy}(z))\frac{\partial^{\Varphi} \eta_{y} \circ \tilde{W}^{\Varphi}_{y}(z)}{K_y\circ \tilde{W}^{\Varphi}_{y}(z)}.
	\end{equation*}
	We compute the right-hand side of this last equation as follows:
	\begin{equation*}
		\frac{K_{xy}(z)}{K_x \circ W_x(z)}=\frac{M_{xy}'(z)}{(M_{xy}(z)+1)^2}[\frac{M_x'(W_x(z))}{(M_x\circ W_x+1)^2}]^{-1}=(\tilde{W}^{\Varphi}_x)'(z).
	\end{equation*}
	We infer
	\begin{align*}
		{\partial^{\Varphi} \eta^{\Varphi}_{xy}(z)} & = S^{\Varphi}_y(M^{\Varphi}_{xy}(z))\frac{K_{xy}(z)}{K_x\circ \tilde{W}^{\Varphi}_x(z)}(\partial^{\Varphi} \eta_{x} \circ \tilde{W}^{\Varphi}_x)(z)+S^{\Varphi}_x(M^{\Varphi}_{xy}(z))\frac{K_{xy}(z)}{K_y\circ \tilde{W}^{\Varphi}_{y}(z)}(\partial^{\Varphi} \eta_{y} \circ \tilde{W}^{\Varphi}_{y})(z) \\
		                                            & =S^{\Varphi}_y(M^{\Varphi}_{xy}(z))(\tilde{W}^{\Varphi}_x)'(z)( \partial^{\Varphi} \eta_x \circ \tilde{W}^{\Varphi}_x)(z)+S^{\Varphi}_x(M^{\Varphi}_{xy}(z))(\tilde{W}^{\Varphi}_{y})'(z)(\partial^{\Varphi} \eta_y \circ \tilde{W}^{\Varphi}_{y})(z).
	\end{align*}
	We multiply both side of the equality by $\eta^{\Varphi}_{xy}(z)$, we obtain
	\begin{equation*}
		{\partial^{\Varphi} \eta^{\Varphi}_{xy}(z)}=[\eta^{\Varphi}_{xy}(z)S_y(M_{xy}(z))](\tilde{W}^{\Varphi}_x)'(z)(\partial \eta_x \circ W_x)(z)+[\eta^{\Varphi}_{xy}(z)S_x(M^{\Varphi}_{xy}(z))](\tilde{W}^{\Varphi}_{y})'(z)(\partial \eta_y \circ \tilde{W}^{\Varphi}_{y})(z).
	\end{equation*}
	Now,
	\begin{equation*}
		\eta^{\Varphi}_{xy}(z)S^{\Varphi}_y(M^{\Varphi}_{xy}(z))=\eta^{\Varphi}_{xy}(z)\frac{1+M^{\Varphi}_{xy}(z)}{M^{\Varphi}_{xy}(z)}((M^{\Varphi}_{y})^{\langle -1 \rangle}\circ M^{\Varphi}_{xy}(z))=\tilde{W}^{\Varphi}_{y}(z).
	\end{equation*}
	Thus, we find
	\begin{align*}
		\eta^{\Varphi}_{xy}(z)\partial\eta^{\Varphi}_{xy}(z)=\tilde{W}^{\Varphi}_{y}(z) (\tilde{W}^{\Varphi}_x)'(z)(\partial \eta_x \circ \tilde{W}^{\varphi}_x)(z) + \tilde{W}^{\varphi}_x(z) (\tilde{W}^{\Varphi}_{y})'(z)(\partial \eta_y \circ \tilde{W}^{\Varphi}_{y})(z).
	\end{align*}
\end{proof}

We prove a subordination formula for the $\eta$-transform of the product of two conditionally free random variables.
\begin{proposition}
	\label{prop:conditionalsub}
	Let $x,y \in \mathcal{A}$ be conditionally free random variables relatively to $(\Psi, \Varphi)$, with $\psi(x)=\psi(y)=1$. Then,
	\begin{equation}
		\label{eqn:toshowcondmult}
		\rho_{xy}^{\Varphi}(z) = \rho_x^{\Varphi}(\tilde{W}^{\Psi}_x(z))\rho_y^{\Varphi}(\tilde{W}^{\Psi}_y(z)).
	\end{equation}
\end{proposition}
\begin{proof}  We refer to \cite{popa2011multiplicative} for the definition of the various transforms we will use. In particular we will be using the conditional $R$-transform $C^{\Varphi|\Psi}_x\in\mathbb{C}[[z]]$, defined by the following functional equation,
	\begin{align*}
		C_x^{\Varphi|\Psi}(z(1+M_x^{\Psi}(z)))(1+M_x^{\Varphi}(z))=M^{\Varphi}_x(z)(1+M^{\Psi}_x(z)).
	\end{align*}
	The conditional cumulants $\kappa^{\Varphi | \Psi}_n(x)$ of $x$ are defined as the Taylor coefficients at $0$ of the (shifted) conditional $R$-tranform:
	$$
		C_x^{\Varphi | \Psi}(z)=\sum_{n\geq 1} \kappa^{\Varphi|\Psi}_n(x)z^n.
	$$
	The conditional $S^{\Varphi|\Psi}_x$ transform  of $x$ is defined by:
	\begin{align*}
		(S_x^{\Varphi|\Psi}(z))^{-1} := (\tilde{C}_x^{\Varphi|\Psi}\circ (C^{\Psi}_x)^{\langle -1 \rangle})(z).
	\end{align*}
	Note that $(C^{\Psi}_x)^{\langle -1 \rangle}(z))^{\langle -1 \rangle}$ is well-defined since $\Psi(x)\neq 0$.
	Recall that simple arithmetic manipulations yield
	$$M^{\Psi}_x(\frac{z}{1+C^{\Psi}_x(z)})=C^{\Psi}_x(z) \Leftrightarrow M^{\Psi}_x(w)=C_x^{\Psi}(w(1+M_x^{\Psi}(w))),$$
	and further
	$$
		w\tilde{C}_x^{\Varphi|\Psi}(w)(1+M_x^{\Varphi}(\frac{w}{1+C_x^{\Psi}(w)}))=M^{\Varphi}_x(\frac{w}{1+C_x^{\Psi}(w)})(1+C_x^{\Psi}(w))).
	$$
	Substituting $w = (C_x^{\Psi})^{\langle -1 \rangle}(z)=(1+z)(M_x^{\Psi})^{\langle -1 \rangle}(z)$ in the previous formula yields
	$$
		(M^{\Psi}_x)^{\langle -1 \rangle}(z)(\tilde{C}_x^{\Varphi|\Psi} \circ(C^{\Psi}_x)^{\langle -1 \rangle})(z)= \frac{M^{\Varphi}_x \circ (M_x^{\Psi})^{\langle -1 \rangle}(z)}{1+M^{\Varphi}_x\circ(M_x^{\Psi})^{\langle -1 \rangle}(z)}.
	$$
	We thus obtain for the conditional $S$-transform of $x$,
	$$
		S_x^{\Varphi|\Psi}(z) = (1+M^{\Varphi}_x \circ(M_x^{\Psi})^{\langle -1 \rangle}(z))\frac{(M_x^{\Psi})^{\langle -1 \rangle}(z)}{M^{\Varphi}_x\circ(M_x^{\Psi})^{\langle -1 \rangle}(z)} = [\frac{z}{\eta^{\Varphi}_{x}(z)}]\circ (M_x^{\Psi})^{\langle -1 \rangle}(z).
	$$
	Since the conditional $S$-transform is multiplicative over conditionally free multiplicative convolution \cite{popa2011multiplicative}:
	$$
		S_{xy}^{\Varphi |\Psi}(z)=S^{\Varphi |\Psi}_x(z) S_y^{\Varphi |\Psi}(z)
	$$
	implies:
	$$
		\frac{z}{\eta_{xy}^{\Varphi}(z)}=\bigg([\frac{z}{\eta^{\Varphi}_x(z)}]\circ(M_x^{\Psi})^{\langle -1 \rangle}\circ (M_{xy}^{\Psi})(z)\bigg)\bigg([\frac{z}{\eta^{\Varphi}_y(z)}] \circ (M_{y}^{\Psi})^{\langle -1 \rangle}\circ (M_{xy}^{\Psi})(z)\bigg).
	$$
	Hence:
	\begin{equation*}
		\eta_{xy}^{\Varphi}(z)=\frac{z}{\tilde{W}^{\Psi}_x(z)\tilde{W}^{\Psi}_y(z)} \eta^{\Varphi}_x(\tilde{W}^{\Psi}_x(z)) \eta^{\Varphi}_y(\tilde{W}^{\Psi}_y(z)).
	\end{equation*}
	This concludes the proof since $\tilde{W}^{\Psi}_x(z) \tilde{W}^{\Psi}_y (z) = z \eta_{xy}^{\Psi}(z)$.
\end{proof}
\begin{proposition}
	\label{prop:submult}
	Let $x,y \in \mathcal{A}$ be cyclically conditionally free relatively to $\Psi,\Varphi,\OmegA$, with $\Psi(x)\neq 0$ and $\Psi(y)\neq 0$. Then
	\begin{multline*}
		{M}^{\OmegA}_{xy}(z)  =z{\ln}(\tilde{W}^{\Psi}_x)'(z))\,{M}^{\OmegA}_y(\tilde{W}^{\Psi}_x(z))+z{\ln}(\tilde{W}^{\Psi}_y)'(z) \,{M}^{\OmegA}_x(\tilde{W}^{\Psi}_y(z))                                                     \\
		-z\frac{d}{dz}{\rm \ln}(\frac{1-\eta_{xy}^{\Varphi}(z)}{(1-\eta_{x}^{\Varphi}(\tilde{W}^{\Psi}_x(z))(1-\eta_{y}^{\Varphi}(\tilde{W}^{\Psi}_y(z))})
		-(1-\OmegA(1_\mathcal{A})) z{\rm \ln}(1-\eta^{\Psi}_{xy})'(z)
	\end{multline*}
\end{proposition}
\begin{proof}
	We compute first:
	\begin{align*}
		\partial^{[\Varphi]} \eta_x(z) & :=\frac{\partial^{[\Varphi]} M_x(z)}{(M^{\Psi}_x(z)+1)^2}                                   \\
		                               & =(1-\eta_x^{\Psi}(z))^2(\partial^{[\Varphi]} M_x(z))                                        \\
		                               & =(1-\eta_x^{\Psi}(z))^2(z\frac{\frac{d}{dz}(M_x^{\Varphi})(z)}{1+M_x^{\Varphi}(z)})         \\
		                               & =(1-\eta_{x}^{\Psi}(z))^2z(\frac{\frac{d}{dz}\eta^{\Varphi}_x(z)}{(1-\eta^{\Varphi}_x(z))}) \\
		                               & =-(1-\eta_x^{\Psi}(z))^2z\frac{d}{dz}{\rm \ln}(1-\eta_x^{\Varphi}(z)).
	\end{align*}
	This implies, since $\ell \mapsto \partial^{\ell} \eta_x$ is linear and by setting $D^{\Psi,\Varphi}\OmegA : = \OmegA -[\Varphi] + (1-\OmegA(1_\mathcal{A})) [\Psi]$ (as in Proposition \ref{th:vanishing_of_mixed_cumulants}):
	\begin{align*}
		\partial^{D\Psi}\eta_x(z) & =\partial^{\OmegA} \eta_x(z)+(1-\eta_x^{\Psi}(z))^2 z\frac{d}{dz}{\rm \ln}(1-\eta_x^{\Varphi}(z))-{(1-\OmegA(1_\mathcal{A})) z}(1-\eta^{\Psi}_{x}(z))^2\frac{d}{dz}{\rm \ln}(1-\eta_x^{\Psi}(z)).
	\end{align*}
	Recall from Proposition \ref{th:vanishing_of_mixed_cumulants} that if $x,y$ are cyclically conditionally free relatively to $(\Psi,\Varphi,\OmegA)$ then they are cyclically free in $(T(\mathcal{A}),\Psi, D^{\psi,\varphi}\Psi)$). Remark \ref{rk:cyfequalif} and Proposition \ref{prop:subinf} imply:
	\begin{align}
		\label{eqn:fondamental}
		\partial^{D^{\Psi,\varphi}\OmegA}\eta_{xy}(z)=z{\ln}(\tilde{W}^{\Psi}_x)'(z)(\partial^{D^{\Psi,\varphi}\OmegA}\eta_x \circ \tilde{W}^{\Psi}_x)(z)  + z\ln(\tilde{W}^{\Psi}_y)'(z)(\partial^{D^{\Psi,\varphi}\OmegA} \eta_y \circ \tilde{W}^{\Psi}_y)(z),
	\end{align}
	and
	\begin{align}
		\label{eqn:relationun}
		 & z\frac{(\tilde{W}^{\Psi}_x)'(z)}{\tilde{W}^{\Psi}_x(z)}(\partial \eta^{D^{\Psi,\varphi}\OmegA}_x \circ \tilde{W}^{\Psi}_x)(z) =z\frac{(\tilde{W}^{\Psi}_x)'(z)}{\tilde{W}^{\Psi}_x(z)}\partial \eta_{x}^{\OmegA}(\tilde{W}^{\Psi}_x(z))+(1-\eta^{\Psi}_{xy}(z))^2z{\ln}(1-\eta_{x}^{\Varphi}(W_x^{\Psi}))'\nonumber \\
		 & -{(1-\OmegA(1_\mathcal{A})) z}(1-\eta^{\Psi}_{xy}(z))^2{\ln}(1-\eta_{xy}^{\Psi})'(z).
	\end{align}
	Of course, we have a similar relation when $x$ is replaced by $y$. By injecting \eqref{eqn:relationun} in \eqref{eqn:fondamental}, we deduce the first formula of the proposition.
\end{proof}
As for additive convolution, we specialize the above proposition to two cases; cyclic Boolean independence between $x-1$ and $y-1$ and cyclic monotone independence between $x-1$ and $y$.

\begin{proposition}
	\label{prop:multcycBoolean}
	Assume that $\mathcal{A}$ is augmented, let $x,y \in\mathcal{A}$ and assume that $x-1$ and $y-1$ (in particular $\Psi(x)=\Psi(y)=1$) are cyclic boolean independent. One has the following relations between the moment generating functions of $x$,$y$ and $xy$ with respect to $\OmegA$:

	\begin{equation*}
		M^{\OmegA}_{xy}(z)=M_x^{\OmegA}(z)+M_y^{\OmegA}(z)-z \frac{d}{dz} {\rm \ln}(\frac{1-z\rho^{\Varphi}_x(z)\rho^{\Varphi}_{y}(z)}{(1-z\rho^{\Varphi}_x(z))(1-z\rho^{\Varphi}_y(z))}) - \frac{z}{1-z}.
	\end{equation*}

\end{proposition}
\begin{proof}
	From equation \eqref{eqn:rhoBooleanmult}, specializing the relation in Proposition \eqref{prop:submult} to the case where $x-1$ and $y-1$ are Boolean independent imply setting $\tilde{W}^{\Psi}_x(z)=\tilde{W}^{\Psi}_y(z)=z$ in the relation of Proposition \eqref{prop:submult} and $\eta^{\Psi}_{x}(z)=\eta^{\Psi}_{y}(z)=\eta^{\Psi}_{xy}(z) = z$.
\end{proof} If setting, for a random variable $x \in \mathcal{A}$:
\begin{align*}
	\partial^{\OmegA | \Varphi}\Sigma_x(z):=M_x^{\OmegA}(z)+z\frac{d}{dz}\ln(1-z\rho^{\Varphi}_{x}(z)) + \frac{z}{1-z},
\end{align*}
from Proposition \ref{prop:multcycBoolean}, we infer the following:
$$
	\partial^{\OmegA | \Varphi}\Sigma_{xy}(z) = \partial^{\OmegA | \Varphi}\Sigma_{x}(z)+\partial^{\OmegA | \Varphi}\Sigma_{y}(z),
$$
whenever $x-1$ and $y-1$ are cyclic Boolean independent. It is instructive to compute $\frac{1}{z}\partial^{\OmegA | \Varphi}\Sigma_x(\frac{1}{z})$; we find
\begin{align*}
	\frac{1}{z}\partial^{\OmegA | \Varphi}\Sigma_x(\frac{1}{z}) & = G^{\OmegA}_x(z)+\frac{d}{dz}{\rm \ln}(zG^{\Varphi}_x(z)) + \frac{1}{z(z-1)} \\
	                                                            & = \mathfrak{h}^{\OmegA}_x +  \frac{1}{z(z-1)},
\end{align*}
where we have used the notations from \cite{arizmendi2022cyclic} (we added a superscript). The transform $\mathfrak{h}^{\OmegA}$ linearizes \emph{additive} cyclic Boolean convolution and related to the generating series $\mathfrak{c}$ of the cyclic Boolean cumulants via
$$
	\mathfrak{c}_x=\frac{1}{z}\mathfrak{h}_x (\frac{1}{z})+z\frac{d}{dz}\eta^{\Psi}_x(z)
$$
see also Section \ref{sec:cyclicBoolean}.
As an immediate consequence of the above formula, we get the following corollary.
\begin{corollary} Let $x_1,\ldots,x_N$ be a sequence of identically distributed elements such that $x_1-1,\ldots,x_N-1$ are cyclic Boolean independent relatively to $(\Varphi,\Psi)$. Then:
	\begin{equation}
		M^{\OmegA}_{x_1\cdots x_N}(z)=NM^{\OmegA}_{x_1}(z)+z\frac{d}{dz}{\rm \ln}\Big(\frac{1-z\rho^{\Varphi}_{x_1}(z)^{N}}{(1-z\rho^{\Varphi}_{x_1}(z))^N}\Big) - N\frac{z}{z-1}.
	\end{equation}
\end{corollary}

\begin{remark} As a final remark about cyclic Boolean multiplicative convolution, we comment on infinite multiplicative infinite divisibility in the case of a unitary operator $u$ acting on some finite dimensional Hilbert space $H$. We say that $u$ is cyclic Boolean multiplicative infinitely divisible if for each integer $n \geq 1$, there exists a finite sequence $(u_1,\ldots,u_n)$ of unitaries of $B(H)$, all identically distributed under each of the functional $\Varphi$ and $\OmegA$ such that $u=u_1\cdots u_n$ and the $u_1-1,\ldots,u_n-1$ are cyclic Boolean independent. In particular, $u$ is Boolean multiplicative infinitely divisible. There is not many of such of operators. In fact, if we assume that $u$ has two distinct eigenvalues $e^{{\sf i}b_1}$ and $e^{{\sf i}b_2}$, with multiplicities $m_1$ and $m_2$ then the $\eta$-transform $\eta^{\Varphi}$ of $u$ has $0$ in its range; $\eta_u^{\Varphi}(\frac{m_1}{{\rm dim}(H)}e^{{\sf i}b_1} + \frac{m_2}{{\rm dim}(H)}e^{{\sf i}b_2})=0$. From \cite{franz2005multiplicative}, this prevents $u$ from being infinitely divisible.
\end{remark}
We move on to computing the moment series with respect to $\OmegA$ of two cyclically monotone random variables $x$ and $y$.
\begin{proposition} Let $x$ and $y$ be two random variables in $\mathcal{A}$ such that $x-1$ is cyclic monotone independent from $y$ relatively to $\Varphi,\OmegA$, then:
	\begin{align}
		\tilde{M}_{xy}^{\OmegA}(z)=\tilde{M}_y^{\OmegA}(z) + (\tilde{M}_x^{\OmegA}(\eta_{y}^{\Varphi}(z)))(\eta^{\Varphi}_{y})^{\prime}(z),
	\end{align}
	where $\tilde{M}^{\Varphi}_x(z)=z^{-1}M_x^{\Varphi}(z)$.
\end{proposition}
\begin{proof}
	It amounts to take $\tilde{W}^{\Psi}_x(z)=\eta^{\Psi}_y(z),\tilde{W}^{\Psi}_y(z)=z$, $\eta^{\Psi}_{xy}(z)=\eta_y^{\Varphi}(z),\eta^{\Psi}_x(z)=z,\eta_y^{\Psi} (z)=\eta_y^{\Varphi}(z)$ in the formula of Proposition \ref{prop:submult} and $(1-\OmegA(1_\mathcal{A})) = 1$.
\end{proof}
\subsection{Limit theorems}
\label{sec:limittheorems}

In this section, we utilize the relation \eqref{eqn:relationcauchytr} to prove limit theorems for sums of cyclically conditionally free random variables. For the Central Limit Theorem, we rely on the Central Limits theorems for conditionally free random variables and for infinitesimally free random variables, which we recall now.
\begin{proposition}[Theorem 4.3 in \cite{bozejko1996convolution}] Let $(X_i)_{i\leq 0}$ be a sequence of conditionally free random variables identically distributed with respect to both states $\Psi$ and $\Varphi$ and centered. We notate:
	\begin{equation*}
		\alpha^2 := \Psi(X_1^{2}),\quad \beta^2 = \Varphi(X_1^2).
	\end{equation*}
	As $N$ tends to infinity, $S_N:= \frac{1}{\sqrt{N}}\sum_{i=1}^N X_i$ converges toward an element $s_{\alpha,\beta}$ whose $\Varphi$ distribution has Cauchy transform:

	$$G^{\Varphi}_{\alpha^2,\beta^2}(z)=\frac{z( \frac{\beta^2}{2}-\alpha^2)+\frac{\beta^2}{2}\sqrt{z^2-4\alpha^2})}{(z^2(\beta^2-\alpha^2)-\beta^4)}$$
	and is distributed under $\Psi$ as a semi-circular element with variance $\alpha^2$.
\end{proposition}
We use the notations of the previous proposition.
\begin{proposition}[see Theorem 5.1 in \cite{popa2010freeness}] Let $(X_i)_{i\geq 0}$ be a sequence of infinitesimally free, centered identically distributed elements relatively to $(\Psi,\partial \Psi)$ and set
	$$
		\alpha^2 := \Psi(X_i^2),~\alpha^{\prime} := \partial \Psi(X_1^2).
	$$
	As $N$ tends to infinity, $S_N$ tends in infinitesimal distribution to an infinitesimal semi-circular element $s^{\prime}_{\alpha^2,\alpha^{\prime}}$ whose infinitesimal Cauchy transform is given by
	$$\partial^{\Psi}G_{s_{\alpha^2,\alpha^{\prime}}}(z):= \partial \Psi((z-s_{\alpha^2,\alpha^{\prime}})^{-1})=\frac{\alpha^{\prime}}{\alpha^2}(\frac{1}{\sqrt{z^2-4\alpha^2}}-\frac{z-\sqrt{z^2-4\alpha^2}}{2\alpha^2}).$$

\end{proposition}
\begin{theorem}
	Let $(X_i)_{i\geq 0}$ be a sequence of cyclically conditionally free random variables relatively to $(\Psi, \Varphi,\OmegA)$, centered and identically distributed.

	Then, as $N$ tends to infinity, $S_N$ converges in distribution to a random variable $s_{\alpha^2,\beta^2,\gamma}$ whose distribution has for Cauchy transform with respect to $\OmegA$:
	\begin{align*}G^{\OmegA}_{s_{\alpha,\beta^2,\gamma}}(z)= \frac{\alpha^{\prime}}{\alpha^2}(\frac{1}{\sqrt{z^2-4\alpha^2}}-\frac{z-\sqrt{z^2-4\alpha^2}}{2\alpha^2})-(1-\OmegA(1_\mathcal{A}))\frac{1}{\sqrt{z^2-4\alpha^2}} \\
		-\left(
		\frac{\frac{\beta^2}{2}-\alpha^2 + \frac{\beta^2}{2}\frac{z}{\sqrt{z^2-4\alpha^2}}}{ z(\frac{\beta^2}{2}-\alpha^2)+\frac{\beta^2}{2}\sqrt{z^2-4\alpha^2}} - \frac{2z(\beta^2-\alpha^2)}{z^2(\beta^2-\alpha^2)-\beta^4} \right),
	\end{align*}
	where $\alpha^2=\Psi(X^2)$, $\beta^2=\Varphi(X^2)$, $\gamma^2=\OmegA(X_1^2)$ and $\alpha^{\prime}=[(1-\OmegA(1_\mathcal{A}))\alpha^2-\beta^2+\gamma^2)].$
\end{theorem}
Before proving the last proposition, let us derive the cyclic Boolean Central Limit Theorem, as in \cite[Theorem 4.7]{arizmendi2022cyclic}. Formally, we want to set $\OmegA(1_{\mathcal{A}})=1$ and $\alpha=0,~\beta=1,\gamma^2=1$ which implies $\alpha^{\prime}=\gamma^2-1=0$.
Straightforward computations show that, as $\alpha$ tends to $0$:
\begin{align*}
	 & \frac{\frac{\beta^2}{2}-\alpha^2 + \frac{\beta^2}{2}\frac{z}{\sqrt{z^2-4\alpha^2}}}{ z(\frac{\beta^2}{2}-\alpha^2)+\frac{\beta^2}{2}\sqrt{z^2-4\alpha^2}} - \frac{2z(\beta^2-\alpha^2)}{z^2(\beta^2-\alpha^2)-\beta^4}=\frac{1}{z}-\frac{2z}{z^2-1} + o_{\alpha,0}(1)
\end{align*}
and
\begin{align*}
	\frac{\alpha^{\prime}}{\alpha^2}\left(\frac{1}{\sqrt{z^2-4\alpha^2}}-\frac{z-\sqrt{z^2-4\alpha^2}}{2\alpha^2}\right)= \alpha^{\prime}(\frac{1}{z^3}) + o_{\alpha,0}(1).
\end{align*}
We obtain, with the notations of \cite{arizmendi2022cyclic}:
\begin{align*}
	\tilde{G}^{\OmegA}_{s_{0,1,\gamma^2}}=-\frac{2}{z} + \frac{2z}{z^2-1}+ \frac{\gamma^2-1}{z^3}=-\frac{2}{z} + \frac{2z}{z^2-1},
\end{align*}
which is the cyclic Boolean Central Limit Theorem.
\begin{proof}
	We let $S = s_{\alpha^2,\beta^2,\gamma}$. Let $S_n=\frac{1}{\sqrt{n}}(a_1+\cdots+a_n)$,  we will denote by $S$ its possible limit.
	Firstly, with respect to $\Psi$, $S_n$ satisfies the free CLT with variance $\alpha^2=\Psi(a^2)$. So
	$$
		G^{\Psi}_{S_n}(z)\to \frac{z-\sqrt{z^2-4\alpha^2}}{2\alpha^2}=:G^{\Psi}_S(z),
	$$
	and from an easy calculation we get:
	$$
		-\frac{G^{\Psi}_S(z)'}{G^{\Psi}_S(z)}=\frac{1}{\sqrt{z^2-\alpha^2}}.
	$$
	Secondly, with respect to the pair $(\Psi,\Varphi)$, $S_n$ satisfies the conditionally free CLT. So
	$$G^{\Varphi}_{S_n}(z)\to \frac{z(\frac{\beta^2}{2}-\alpha^2)+1/2 \beta^2\sqrt{z^2-4\alpha^2})}{(z^2(\beta^2-\alpha^2)-\beta^4)}=:G^{\Varphi}_{S}(z),$$
	that satisfies
	$$
		\frac{G^{\Varphi}_S(z)'}{G^{\Varphi}_S(z)}=\frac{\frac{\beta^2}{2}-\alpha^2 + \frac{\beta^2}{2}\frac{z}{\sqrt{z^2-4\alpha^2}}}{ z(\frac{\beta^2}{2}-\alpha^2)+\frac{\beta^2}{2}\sqrt{z^2-4\alpha^2}} - \frac{2z(\beta^2-\alpha^2)}{z^2(\beta^2-\alpha^2)-\beta^4}
	$$
	Finally, $S_n$ satisfies an Infinitesimal CLT for $D^{\Psi,\Varphi}\OmegA$ with variance
	$$
		D^{\Psi,\Varphi}\OmegA(a^2)=\OmegA(a^2)-\Varphi(a^2)+(1-\OmegA(1_A))\Psi(a^2) = \gamma^2-\beta^2+(1-\OmegA(1_{\mathcal{A}}))\alpha^2,
	$$
	which implies:
	$$G^{D^{\Psi,\Varphi}\OmegA}_{S_n}(z)\to \frac{\alpha^{\prime}}{\alpha^2}\frac{1}{\sqrt{z^2-4\alpha^2}}-\frac{z-\sqrt{z^2-4\alpha^2}}{2\alpha^2}.$$
\end{proof}

We say that $a\in \mathcal{A}$ is a Bernoulli with parameters $((\lambda_{\Psi},\lambda_{\Varphi},\lambda_{\OmegA}),(\alpha_{\Psi},\alpha_{\Varphi},\alpha_{\OmegA}))$, $\lambda_{\Psi},\lambda_{\Varphi}\in [0,1]$ and $\lambda_{\OmegA} \in \mathbb{R}$, $\alpha_{\Psi},\alpha_{\Varphi},\alpha_{\OmegA} \in \mathbb{R}$ relatively to $(\Psi, \Varphi, \OmegA)$
if it is a conditionaly free Bernoulli element relatively to $(\Psi,\Varphi)$ and an infinitesimal Bernoulli element relatively to $(\Psi,D^{\Psi,\Varphi}\OmegA)$. In particular, for $n\geq 1$ :
$$
	\Psi(a^n)=\lambda_{\Psi}\alpha_{\Psi}^n,~\Varphi(a^n)=\lambda_{\Varphi}\alpha_{\Varphi}^n,
$$
and
$$
	D^{\Psi,\Varphi}\OmegA(a^{\otimes n})=\lambda_{\OmegA}\alpha_{\Psi}^{n}+n\lambda_{\Psi}\alpha_{\Psi}^{n-1}\alpha_{\OmegA}.
$$
From the formulae:
\begin{equation*}
	G_a^{\Psi}(z)=\frac{1-\lambda_{\Psi}}{z}+\frac{\lambda_{\Psi}}{z-\alpha_{\Psi},},~ -\frac{(G_a^{\Psi})^{\prime}(z)}{G_{a}^{\Psi}(z)}=\frac{1}{z}+\frac{1}{z-\alpha_{\Psi}},
\end{equation*}
we easily get that $a$ is Bernoulli if and only if, for any integer $n\geq 1$:
\begin{equation*}
	\OmegA(a^n)=\alpha_{\Psi}^n((1-\OmegA(1_\mathcal{A}))+\lambda_{\OmegA})-\alpha^n_{\Varphi}\lambda_{\Varphi}+n\lambda_{\Psi}\alpha_{\Psi}^{n-1}\alpha_{\OmegA}.
\end{equation*}

\begin{proposition}[Theorem 4.4 in \cite{bozejko1996convolution}]
	Assume that $\lambda_{\Psi}(N)N\to\lambda_{\Psi}$ and $\lambda_{\Varphi}(N)N\to\lambda_{\Varphi}$ and take for each $N\geq 1$ a sequence $(X^{(N)}_1,\ldots,X_N^{(N)})$ of i.i.d conditionally free Bernoulli elements with rates $(\lambda_{\Psi}(N),\lambda_{\Varphi}(N))$ and $\alpha_{\Psi}=\alpha_{\Varphi}=1$. Then,
	$$
		B_{N}=\sum_{i=1}^N X_i^{(N)}
	$$
	converges in non-commutative distribution to a conditionally free Poisson random variable relatively to $(\Psi,\Varphi)$ with parameters $((\lambda_{\Psi},\lambda_{\Varphi}), (1,1))$:
	$$
		G_{B_{N}}^{\Psi}(z) \to \mathcal{P}_{\lambda_{\Psi}}(z):=\frac{1}{2z}\big({z}+(1-\lambda^{\Psi})-\sqrt{(z-(1+\lambda_{\Psi}))^2-4\lambda_{\Psi}})\big)
	$$
	and
	$$
		G^{\Varphi}_{B_{N}}(z)\to \mathcal{P}^c_{\lambda_{\Psi},\lambda_{\Varphi}}:=\frac{z(2\lambda_{\Psi}-\lambda_{\Varphi})+\lambda_{\Varphi}(1-\lambda_{\Psi})-\lambda_{\Varphi}\sqrt{(z-(1+\lambda_{\Psi}))^2-4\lambda_{\Psi}}}{2z(z(\lambda_{\Psi}-\lambda_{\Varphi})+\lambda_{\Varphi}(1-\lambda_{\Psi}+\lambda_{\Varphi}))}.
	$$
\end{proposition}
Recall that the free Poisson distribution with parameters $\alpha,\lambda$ has Cauchy transform:
\begin{equation*}
	\mathcal{P}_{\alpha,\lambda}(z)=\frac{z-\alpha(1+\lambda)-\sqrt{(z-\alpha(1+\lambda))^2-4\lambda\alpha^2}}{2\alpha z}.
\end{equation*}
In comparison with our previous notations, $\mathcal{P}_{1,\lambda}(z)=\mathcal{P}_{\lambda}(z)$.
\begin{proposition}[Corollary 36 in \cite{belinschi2012free}] For each $N\geq 1$, let $(X_1^{(N)},\ldots,X_{N}^{(N)})$ be i.i.d infinitesimal Bernoulli elements relatively to $(\Psi,\partial\psi)$ with rates $\lambda_{\theta}(N)N\to \lambda_{\theta}$ and $\lambda^{\prime}_{\theta}(N)N\to \lambda^{\prime}_{\theta}$.
	Then $B^{(N)}$ converges in distribution toward an infinitesimal Poisson distribution with parameter $(\lambda_{\theta},\lambda_{\theta^{\prime}}$):
	$$
		G_{B_N}^{\theta}(z)\to \mathcal{P}_{\lambda_{\theta}}(z),~G_{B_N}^{\theta^{\prime}}(z)\to \frac{d}{du}_{|u=0}\mathcal{P}_{1+u,\lambda_{\theta}+u\lambda^{\prime}_{\theta}}(z).
	$$
\end{proposition}
\begin{proposition}
	\label{prop:poissonlimittheorem}
	For each $N\geq 1$, let $X^{(N)}_1,\ldots,X_N^{(N)}$ be a finite sequence of size $N$ of i.i.d Bernoulli elements relatively to $\Psi,\Varphi,\OmegA$. Under the assumptions:
	$$
		\lambda_{\Psi}(N)N \to \lambda_{\Psi},\quad\lambda_{\Varphi}(N)N \to \lambda_{\Varphi},\quad\lambda_{\OmegA}(N)N \to \lambda_{\OmegA},
	$$
	the Cauchy transform under $\OmegA$ of the sum $B_N$ of the $X_i^{(N)}, ~ 1 \leq i \leq N$ converges toward  $G^{\OmegA}_{{\lambda_{\Psi},\lambda_{\Varphi},\lambda_{\OmegA}}}$ satisfying:
	$$
		\frac{d}{du}_{|u=0}\mathcal{P}_{1+u,\lambda_{\Psi}+\lambda_{\OmegA}u}(z) =G^{\OmegA}_{{\lambda_{\Psi},\lambda_{\Varphi},\lambda_{\OmegA}}}(z)-(1-\OmegA(1_{\mathcal{A}}))\frac{(\mathcal{P}_{\lambda_\Psi})^{\prime}(z)}{\mathcal{P}_{\lambda_\Psi}(z)}+\frac{(\mathcal{P}^c_{\lambda_{\Psi},\lambda_{\Varphi}})^{\prime}(z)}{\mathcal{P}^c_{\lambda_{\Psi},\lambda_{\Varphi}}(z)}.
	$$
\end{proposition}
In the context of cyclic Boolean independence, that is, when $ \lambda_{\Psi}=0,~\OmegA(1_{\mathcal{A}})=0,$ then:
\begin{align*}
	 & \mathcal{P}^c_{0,\lambda_{\Varphi}}(z)= \frac{1}{z-(1-\lambda_{\Varphi})}, \quad \frac{d}{dz}{\rm \ln}\mathcal{P}^c_{0,\lambda_{\Varphi}}(z) = -\frac{1}{z-(1-\lambda_{\Varphi})}, \\
	 & \mathcal{P}_{0}(z) = \frac{1}{z},                                                                                                                                                  \\
	 & \frac{d}{du}_{|u=0}\mathcal{P}_{1+u,\lambda_{\OmegA}u}(z)=-\lambda_{\OmegA}\frac{1}{z}+\lambda_{\OmegA}\frac{1}{z-1},
\end{align*}
and we infer
$$
	G^{\OmegA}_{{0,\lambda_{\Varphi},\lambda_{\OmegA}}}(z)=\lambda_{\OmegA}\frac{1}{z(z-1)}+\frac{1}{z(z-(1-\lambda_{\Varphi}))}.
$$
We conclude these computations with the following statement, a corollary of Proposition \ref{prop:poissonlimittheorem}.
\begin{proposition}
	With the notations introduced so far, for each $N\geq 1$ we pick cyclic Boolean independent and identically distributed Bernoulli elements $(X_1^{(N)},\ldots,X_{N}^{(N)})$ in $(\mathcal{B}(\mathcal{H}^{(N)}),\Varphi,\OmegA={\rm Tr}_{\mathcal{H}^{(N)}})$ with jump rate parameters $(\lambda_{\Varphi}(N),\lambda_{\OmegA}(N))$:
	$$
		\Varphi((X_{i}^{(N)})^{n}) = \lambda_{\Varphi}(N),\quad\Psi((X_{i}^{(N)})^{n}) = 1+\lambda_{\OmegA}(N)-\lambda_{\Varphi}(N).
	$$
	Then, under the assumptions that the sequences $\lambda_{\Varphi}(N)N$ and $\lambda_{\OmegA}(N)N$ converge respectively toward $\lambda_{\Varphi}$ and $\lambda_{\OmegA}$ with $\lambda_{\Varphi} \in [0,1[$ and $\lambda_{\OmegA}\in\mathbb{N}^{\star}$, as $N$ tends to infinity,
	\begin{enumerate}
		\item the $\lambda_{\OmegA}$ largest eigenvalues of $B^{(N)}$ converge to $1$,
		\item the following eingenvalue in the decreasing order has, asymptotically, multiplicity one and converges to $1-\lambda_{\Varphi}$,
		\item the remaining part of the spectrum of $B^{(N)}$ converges to $0$.
	\end{enumerate}
\end{proposition}
\section{Products of graphs and cyclic conditional freeness}
\label{sec:graphs}
In this last section we want to show how some known graph products may be generalized to include models for the cyclic freeness and for cyclic conditional freeness. In order to do this, we will work with rooted graphs. 	For details about the material exposed in the first part of this section, we refer the reader to \cite{accardi2007decompositions}.

\subsection{Rooted graphs, products of graphs}

By a \textit{rooted graph} we understand a pair $(G,e)$, where $G=(V,E)$ is an undirected graph with set of vertices
$V=V(G)$, and set of edges $E=E(G)\subseteq \{(v,v'):v,\ v'\in V,\ v\neq v'\}$ and $e\in V$ is
a distinguished vertex called the \textit{root}. For rooted graphs we will use the notation $V^0=V\backslash \{e\}$.
Two vertices $v,v'\in V$ are called \textit{adjacent} if $(v,v')\in E$,  i.e.
vertices $v,v'$ are connected with an edge. Then we write $v \sim v'$.
The \textit{degree} of $v\in V$ is defined by $d(v)=|\{v'\in V:v'\sim v\}|$. We will assume that $G$ is locally finite, that is, $deg(v)<\infty$, for all $v\in V$. 

The \textit{adjacency matrix} $A=A(G)$ of $G$ is the matrix defined by
\begin{equation*}
	A_{v,v'}=\left\{
	\begin{array}{ll}
		1 & {\rm if} \;\; v\sim v' \\
		0 & {\rm otherwise.}
	\end{array}
	\right.
\end{equation*}
We identify $A$ with the densely defined symmetric operator on $l^{2}(V)$ defined by
\begin{equation}
	A\delta(v)=\sum_{v\sim v'}\delta(v'),
\end{equation}
for $v\in V$, where $\delta(v)$ denotes the indicator function on $\{v\}$.

If $(G,e)$ is a rooted graph with adjacency matrix $A$ acting on $\oplus_{v\in\mathbb{V}(G)} \mathbb{C}v$, the distribution of $G$ in the root state is the lineal functional $\mu_G\colon \mathbb{C}[X]\to \mathbb{C}$, such that  $\mu(X^m)=\langle e, A^n(e) \rangle$.

Let $G_1 = (V_1, E_1)$ and $G_2 = (V_2, E_2)$ be two
rooted graphs
with distinguished vertices $e_1\in V_1$ and $e_2\in V_2$.

The \emph{star product} of $G_1$ with $G_2$  is the graph $G_1\star G_2 =((V_1 \times \{e_2\})\cup (\{e_1\}\times V_2), E)$ such that for $(v_1, w_1), (v_2, w_2) \in V_1\times V_2$
the edge $e = (v_1, w_1) \sim (v_2, w_2)\in E$ if and only if one of the following holds:
\emph{i)} $v_1 = v_2 = e_1$ and $w_1\sim w_2$,
\emph{ii)} $v_1 \sim v_2$ and $w_1 = w_2 = e_2$.

The \emph{comb product} of $G_1$ with $G_2$ is the graph $G_1\rhd G_2 =(V_1 \times V_2, E)$ such that for $(v_1, w_1), (v_2, w_2) \in V_1\times V_2$
the edge $e = (v_1, w_1) \sim (v_2, w_2)\in E$ if and only if one of the following holds:
\emph{i)} $v_1 = v_2 $ and $w_1\sim w_2$,
\emph{ii)} $v_1 \sim v_2$ and $w_1 = w_2 = e_2$.

\emph{A rooted set} $(V,e)$ is simply a set with a distinguished element $e$.
Given two rooted sets $(V_1,e_1)$ and $(V_2,e_2)$, the \emph{free product} of
$(V_1,e_1)$ and $(V_2,e_2)$ is the rooted set
\[
	(V_1,e_1) * (V_2,e_2) := (V_1 * V_2, e),
\]
where
\[
	V_1 * V_2
	=
	\{e\}
	\;\cup\;
	\{v_1 v_2 \cdots v_m :
	v_k \in V_{i_k}^0,\;
	i_k \in \{1,2\},\;
	i_k \neq i_{k+1},\;
	m \in \mathbb{N}\},
\]
with $V_i^0 := V_i \setminus \{e_i\}$.
The root $e$ is the empty word.

\medskip

The \emph{free product of rooted graphs} $(G_1,e_1)$ and
$(G_2,e_2)$ is the rooted graph
\[
	(G_1,e_1) * (G_2,e_2) := (G_1 * G_2, e),
\]
whose vertex set is $V_1 * V_2$ and whose edge set is
\[
	E_1 * E_2
	:=
	\{(vu, v'u) :
	(v,v') \in E_1 \cup E_2,\;
	u, vu, v'u \in V_1 * V_2\}.
\]

The (non-associative) \emph{orthogonal product} of the rooted graphs $(G_1, e_1)$ and $(G_2,e_2)$ denoted $G_1 \vdash G_2$ resembles the comb product $G_1 \rhd G_2$ of the two graphs, to the exception that $G_2$ is not attached to the root of $G_1$;
$V(G_1 \vdash G_2)=(V(G_1)\setminus\{e\})\times V(G_2)\cup \{(e_1,e_2)\}$ and the set of edges of $G_1 \vdash G_2$ is the set of pairs $(v_1, w_1), (v_2, w_2) \in V(G_1 \vdash G_2)$
such that one of the following holds:
\emph{i)} $v_1 = v_2\neq e_1$ and $w_1\sim w_2$,
\emph{ii)} $v_1 \sim v_2$ and $w_1 = w_2 = e_2$.

The free product $G_1 * G_2$ decomposes as
${G_1 * G_2} = \Gamma_1 \star \Gamma_{2}$, where the two rooted graphs $\Gamma_1$ and $\Gamma_2$ are the \emph{subordination branches} associated to $G_1$ and $G_2$, respectively. Each of the subordination branch is the limit, in the appropriate sense, of a sequence of graphs $(\Gamma_1^{(n)})_{n\geq 1}$ and $(\Gamma_2^n)_{n\geq 1}$ defined by
\begin{eqnarray*}
	\Gamma^{(n)}_1 :=G_1\vdash (G_2\vdash\Gamma^{(n-1)}_1 ), \quad \Gamma_1^{(1)} := G_1,\\
	\Gamma^{(n)}_2 :=G_2\vdash (G_1\vdash\Gamma^{(n-1)}_2 ), \quad \Gamma_2^{(1)} := G_2.\end{eqnarray*}

Let $(G_1,e_1)$ and $(G_2,e_2)$ be two rooted graphs. We consider the free product $G_1*G_2$ of the two graphs $G_1$ and $G_2$. We denote by $\mu_1$ and $\mu_2$ the spectral distributions of $G_1$ and $G_2$ in their root states, respectively. The spectral distribution in the root state of $G_1*G_2$ is given by the free additive convolution $\mu_1 \boxplus \mu_2$ of the distributions of $\mu_1$ and $\mu_2$.
Moreover, the adjacency matrix $A$ of $G_1*G_2$ can be written as a sum:
$$A=A^{(1)}+A^{(2)},$$
with $A^{(1)}$ and $A^{(2)}$ having distribution $\mu_1$ and $\mu_2$, respectively, in the root state of $G_1\star G_2$. Denoting by $$\Psi:\mathbb{C}\langle X_1,X_2\rangle\to \mathbb{C}$$ the joint distribution of $(A^{(1)},A^{(2)})$ (in the root state), the variables $X_1$ and $X_2$ are free with respect to $\Psi$ \cite{accardi2007decompositions}.
\subsection{Conditional free product of graphs}

Let $(G_1,e_1)$ and $(G_2,e_2)$ be two rooted graphs.
Let $(H_1,f_1)$ and $(H_2,f_2)$ be two additional rooted graphs. We now come to a definition of graph product original to this work, to the best of our knowledge. We define \emph{free product of} $H_1$ \emph{and} $H_2$ \emph{conditionally to} $G_1$ \emph{and} $G_2$ by 
\begin{equation}
\label{eqn:conditionproductsubordinationbranches}
H_1\prescript{}{G_1}{*}^{}_{G_2}H_2 := (H_1\vdash  \Gamma_2)\star (H_2 \vdash  \Gamma_1),
\end{equation}
where  $\Gamma_1$ and $\Gamma_2$ are the subordination branches associated to $G_1$ and $G_2$, as defined in the previous section.

To build the conditional product we first take the star product of the two graphs $H_1$ and $H_2$. In a second time, to every vertex of $H_1 \subset H_1 \star H_2$ but the root, we attach the subordination branch $\Gamma_{2}$ which 'starts' with the graph $G_2$. We repeat this process with $H_2$ and the subordination branch $\Gamma_1$.

\begin{figure}[!ht]
	\centering
	\includesvg[width=0.85\textwidth]{bunbdeux}
	\caption{\label{fig:conditionallyfreeproductgraph}The graphs $H_1\prescript{}{G_1}{*}^{}_{G_2}H_2$ and $G_1*G_2$ together with their branches $\Gamma_1$ and $\Gamma_2$. }

\end{figure}
Denoting by $\mu_1$, $\mu_2$, $\nu_1$ and $\nu_2$ the spectral distributions in the root state of $G_1$, $G_2$, $H_1$ and $H_2$ respectively, the spectral distribution in the root state of $H_1\prescript{}{G_1}{*}^{}_{G_2}H_2$ is given by
$\nu_1\prescript{}{\mu_1}{\boxplus}^{}_{\mu_2}\nu_2$, the additive free convolution of $\nu_1$ and $\nu_2$ conditionally to $\mu_1$ and $\mu_2$.
Furthermore, the adjacency matrix $A$ of $H_1\prescript{}{G_1}{*}^{}_{G_2}H_2$ can be written as:
$$A=A^{(1)}+A^{(2)},$$
with $A^{(1)}$ and $A^{(2)}$ having respective distribution $\nu_1$ and $\nu_2$ in the root state. If we denote by $$\Varphi:\mathbb{C}\langle X_1,X_2\rangle\to \mathbb{C}$$ the joint distribution of $(B^{(1)},B^{(2)})$ in the root state, the variables $X_1$ and $X_2$ are conditionally free with respect to $(\Psi,\Varphi)$. The matrix $A^{(1)}$ is the adjacency matrix of the graph obtained by retaining from the graph $H_{1} \prescript{}{G_1}{*}^{}_{G_2} H_2$ all vertices and adjacent edges in copies of the two graphs $H_1$ and $G_1$. The matrix $A^{(2)}$ is obtained similarly, by retaining the edges corresponding to copies of the graphs $H_2$ and $G_2$.

In preparation for the forthcoming section, let us develop the argument that entails conditional freeness of $X_1$ and $X_2$ with respect to the pair of states $(\Psi,\Varphi)$. Let $\bm{p} = (p_1,\ldots,p_n)$ be an alternating sequence of polynomials in $\mathbb{C}\langle X_1,X_2\rangle$ and suppose that $p_n\in\mathbb{C}\langle X_1\rangle$ (so that $p_{n-1} \in \mathbb{C}\langle X_2\rangle$ and so on). We assume that $\Psi(p_i)=0,$ and want to show
$$
	\Varphi(p_1\cdots p_n) = \langle e, p_1(B^{(x)})\cdots p_n(B^{(1)})e \rangle = \langle e, p_1(B^{(x)})e\rangle\cdots\langle e, p_n(B^{(1)})e \rangle.
$$
We may write $\Varphi(p_1\cdots p_n)$ as a weighted sum over paths drawn in the graph $H_1\prescript{}{G_1}{*}^{}_{G_2}H_2$
$$
	\Varphi(p_1(B^{(x)})\cdots p_n(B^{(1)}))=\sum_{\substack{\gamma:e\to e}} a_{\gamma}(\bm{p},B^{(1)},B^{(2)})
$$

When the path $\gamma$, starting at the root $e$ hits a vertex of $H^{1}$ that is not the root, it continues exploring vertices of the subordination branch $B^{(1)}$ attached to $v$. Besides, in order for $\gamma$ to loop at $e$, while passing through $v\in H_1\backslash\{e\}$, it must pass a second time through $v$. By denoting $A_{G_1}$ the adjacency matrix of the subordination branch, and writing it as
$$
	A_{G_1} = A^{(1)}_{G_1} + A^{(2)}_{G_1}
$$
where as usual, we retain in each $A^{(i)}_{G_1}$, $i\in \{1,2\}$ only edges in copies of $G_i$. Finally, since a path that joins a vertex of $H_1$ must go through a copy of $G_2$, one has the factorization:
\begin{align*}
	 & \langle e, p_1(B^{(x)})\cdots p_n(B^{(1)})e \rangle = \sum_{\substack{\gamma \subset \mathcal{B}_1(v) \\ \gamma: v \to v}} a_{\gamma}({( p_2,\ldots,p_n)}, A^{(1)}_{G_1}, A_{G_1}^{(2)})\sum_{\substack{\gamma_2 \subset H_1 \\ \gamma_2 : v \to e}}a_{\gamma_2}(p_1,A_{H_1},A_{H_2}) \\
	 & \hspace{5cm}+ \langle e, p_1(B^{(x)})\cdots p_{n-1}(B^{(2)})e\rangle \langle e,p_n(B^{(1)} e\rangle   \\
	 & = 0 + \langle e, p_1(B^{(x)})\cdots p_{n-1}(B^{(2)})e\rangle \langle e,p_n(B^{(1)} e\rangle.
\end{align*}
The polynomials $p_i \in \mathbb{C}\langle X_1,X_2\rangle$ are alternating and satisfy $\Psi(p_i)=0$. The last equality follows from freeness. We conclude by induction.
\begin{remark} Let us comment on the connection of the following relation between the $F$ transforms, the subordination series, of elements, conditionally free elements $a$ and $b$ and their sums 
\begin{equation}
\label{rk:relation}
(W_a + W_b  - z)  + (F_{a+b}^{\Varphi}(z) - z) = (F_a^{\Varphi}(W_a(z))+F_b^{\Varphi}(W_b(z)) - z),
\end{equation}
in terms of our graph product. If $a$ and $b$ are distributed as the pairs of graphs $(G_1,H_1)$ and $(G_2,H_2)$ in their root states, the subordination series $W_a$, $W_b$ are the $F$ transforms of the subordination branches $\Gamma_1$ and $\Gamma_2$, while $F_a^{\Varphi}(W_a(z))$ (resp. $F_b^{\Varphi}(W_b(z))$) is the $F$ transform of the monotone product $H_1 \vartriangleleft G_1$ (resp. $F_b^{\Varphi}(W_b(z))$). The right-hand side of \eqref{rk:relation} is equal to the $F$ transform of the boolean product 
$$
K:=(H_1 \vartriangleright \Gamma_1) \star (H_2 \vartriangleright \Gamma_2) 
$$
and the relation \eqref{rk:relation} is equivalent to the decomposition of the graph $K$ as a boolean product between the conditional free product of $(G_1,H_1)$ and $(G_2),H_2)$, and the free product of the graph $G$'s. This can be understood as follows: comparing the relation defining $K$ with \eqref{eqn:conditionproductsubordinationbranches}, the effect of taking the monotone convolution, instead of the subordination product, is to "graft" an additional copy of the subordination branches to the roots. In computing the boolean product defining $K$, these subordination branches attached to the roots are grafted by their roots, resulting in the free product of the graph $G$'s being attached to the root of the conditional free product of $(G_1,H_1)$ and $(G_2,H_2)$.
\end{remark}
\subsection{Free, star and comb products of graphs from the conditional free product of graphs}
We know that free, Boolean and monotone independences are particular cases of conditional freeness. From the graph product side, it is also possible to recover the free product $*$, the star product $\star$, and the comb product $\rhd$ from the graph product $H_1\prescript{}{G_1}{*}^{}_{G_2}H_2$.

First from
$H_1\prescript{}{G_1}{*}^{}_{G_2}H_2 \simeq (H_1\vdash \Gamma_2)\star (H_2 \vdash \Gamma_1),$
if $G_1=H_1$ and $H_2=G_2$, then  $(H_1\vdash \Gamma_2)$ and $(H_1\vdash \Gamma_2)$, are just the branches of the free product, i.e. $(H_1\vdash \Gamma_2)=\Gamma_1$ and $(H_1\vdash \Gamma_2)=\Gamma_1$, and hence
$$H_1\prescript{}{H_1}{*}^{}_{H_2}H_2\simeq H_1* H_2.$$


For the star and comb product, we note that there is an identity with respect to both of this products, namely $G_\emptyset:=(\{e_1,\emptyset\},e_1)$, i.e. , the rooted graph with one vertex $e_1$, no edges and distinguished vertex  $e_1$, satisfying the rules
$$(G_1,e_1)\star G_\emptyset\simeq G_\emptyset \star (G_1,e_1) \simeq (G_1,e_1),$$
$$(G_1,e_1)*G_\emptyset \simeq G_\emptyset * (G_1,e_1) \simeq (G_1,e_1).$$

We still have
$$G_1*G_2\simeq \Gamma_1\star  \Gamma_2$$
if we consider that $(G_1,e_1)$ is the branch of $(G_1,e_1)\star G_\emptyset\simeq G_\emptyset \star (G_1,e_1)$ subordinate to $(G_1,e_1)$ and that $G_\emptyset$ is the branch of $(G_1,e_1)\star G_\emptyset\simeq G_\emptyset \star (G_1,e_1)$ subordinate to $G_\emptyset$. We also define
$$(G_1,e_1)\vdash G_\emptyset\simeq (G_1,e_1)$$
and
$$G_\emptyset  \vdash (G_1,e_1)\simeq G_\emptyset.$$
With these conventions, it is clear that
$$ H_1\prescript{}{G_\emptyset}{*}^{}_{G_\emptyset}H_2\simeq H_1\star H_2\ \ \text{and}\ \ H_1\prescript{}{G_\emptyset}{*}^{}_{H_2}H_2\simeq H_1\rhd H_2.$$

\subsection{Cyclic conditional freeness and graph products}
We assume here that $H_1$ and $H_2$ are finite.
Let us denote by $W$ the finite set of vertices of $H_1\prescript{}{G_1}{*}^{}_{G_2}H_2$ which belongs to $H_1$ or $H_2$ (the cardinality of $W$ is the sum of the number of vertices of  $H_1$ and $H_2$ minus one, because of the identification of the roots). Recall the definition of $A^{(1)}$ and $A^{(2)}$ of the last section. We define the linear functional $$\OmegA:\mathbb{C}\langle X_1,X_2\rangle\to \mathbb{C}$$ by the sum of vector states given by the vertices in $W$, i.e.
$$\OmegA(P)=\sum_{v\in W}\Big\langle P(A^{(1)},A^{(2)}) \delta(v),\delta(v) \Big\rangle.$$
In addition, for each $i\in \{1,2\}$, we denote by $\OmegA_{H_i}\colon \mathbb{C}\langle X_i \rangle \to \mathbb{C}$ the distribution of $H_i$ with respect to the non-normalized trace:
$$
	\OmegA_{H_i}(P) =  {\rm Tr}(P(A_{H_i})), \quad P\in\mathbb{C}\langle X_i\rangle.
$$

\begin{theorem} We let $\mathcal{I}$ be the two sided ideal of $\mathbb{C}\langle X_1,X_2\rangle$ generated by $X_1,X_2$ of polynomials with non-constant coefficient.
	For each integer $i \in \{1,2\}$ and for any polynomial $P(X_i) \in \mathcal{I}$ with non-constant coefficient, $\OmegA_{H_i}(P)=\OmegA(P)$.
	Moreover, $\OmegA$ is tracial and the variables $X_1$ and $X_2$ are cyclic conditionally free with respect to $(\Psi,\Varphi,\OmegA)$.
\end{theorem}
\begin{proof}Let us denote by $V_1$ the set of vertices of $H_1$ in $H_1\prescript{}{G_1}{*}^{}_{G_2}H_2$ and by $e$ the root of $H_1\prescript{}{G_1}{*}^{}_{G_2}H_2$. The operator $B^{(1)}$ leaves invariant the subspace $\ell^2(V_1)$ and vanishes on $\ell^2(W\setminus V_2)$  so $$\OmegA(P(X_1))=\sum_{v\in W}\Big\langle P(B^{(1)}) \delta(v),\delta(v) \Big\rangle=\sum_{v\in V_1}\Big\langle P(B^{(1)}) \delta(v),\delta(v) \Big\rangle=Tr_{\ell^2(V_1)}\Big(P(B^{(1)}|_{\ell^2(V_1)})\Big)$$
	as announced. Similarly, the distribution of $X_2$ with respect to $\OmegA$ is the spectral distribution of $H_2$ with respect to the non-normalized trace. As a consequence, $\OmegA$ is tracial on the algebra generated by $X_1$ (respectively on the algebra generated by $X_2$).

	Now, we claim that, for any alternating tuple $\bm{p}= (p_1,\ldots,p_n)$ of polynomials with $n\geq 2$, satisfying $\Psi(p_1)=\cdots = \Psi(p_n)=0$, we have
	$$
		\OmegA(p_n \cdots p_1)  = \Varphi(p_n) \cdots \Varphi(p_1)$$if $\bm{p}$ is cyclically alternating and, if not,
	$$\OmegA(p_n \cdots p_1)  = \Varphi(p_1p_n) \Varphi(p_{n-1})\cdots \Varphi(p_2).$$

	This claim implies the traciality of $\OmegA$ and the cyclic conditional-freeness.

	In order to prove the claim, we denote by $P_W$ the orthogonal projection to $\ell^2(W)$, by $P_{V_1}$ the orthogonal projection to $\ell^2(V_1)$ and by $P_v$ the orthogonal projection to any vertex $\ell^2(\{v\})$. Let us assume that $p_1\in \mathbb{C}\langle X_1\rangle$ (the case $p_1\in \mathbb{C}\langle X_2\rangle$ is treated similarly). We set
	$$B_1:=p_1(B^{(1)}),\ B_2:=p_2(B^{(2)}),\ B_3:=p_3(B^{(1)}),\ B_4:=p_4(B^{(2)}),\, \ldots $$
	Because of the centering of $p_1,\ldots,p_n$ with respect to $\Psi$, it is impossible to escape from $\ell^2(W)$ and come back to $\ell^2(W)$ by an alternating tuple of operators. Consequently, we have
	\begin{align*}
		\OmegA(p_n \cdots p_1) & =\sum_{v\in W}\Big\langle B_n\cdots B_2 B_1 \delta(v),\delta(v) \Big\rangle           \\
		                       & =\sum_{v\in W}\Big\langle B_nP_W \cdots P_WB_2 P_W B_1\delta(v),\delta(v) \Big\rangle \\
		                       & =\sum_{v\in W}Tr(B_nP_W \cdots P_WB_2 P_W B_1P_v)                                     \\
		                       & = {\rm Tr}(B_nP_W \cdots P_WB_2 P_W B_1P_W)                                           \\
	\end{align*}
	Because $B_1,\ldots,B_n$ are centered, $P_W B_1P_W=P_{V_1}B_1P_{V_1}$, $P_WB_2P_{W}=P_{V_2} B_2P_{V_2}$, $P_WB_3P_W=P_{V_1} B_3P_{V_1}$ and so on. As a consequence, and using $P_{V_1}P_{V_2}=P_e$, we get
	\begin{align*}
		\OmegA(p_n \cdots p_1)
		 & ={\rm Tr}(B_nP_W \cdots P_WB_2 P_W B_1P_W)                             \\
		 & ={\rm Tr}(P_{V_2}B_n P_e \cdots P_eB_2 P_e B_1P_{V_1})                 \\
		 & =\Varphi(p_{n-1})\cdots \Varphi(p_2) {\rm Tr}(P_{V_2}B_nP_eB_1P_{V_1}) \\
		 & =\Varphi(p_{n-1})\cdots \Varphi(p_2){\rm Tr}(P_eB_nP_eB_1)             \\
		 & =\Varphi(p_{n})\cdots \Varphi(p_1)
	\end{align*}
	if $\bm{p}$ is cyclically alternating, and
	\begin{align*}
		\OmegA(p_n \cdots p_1)
		 & ={\rm Tr}(B_nP_W \cdots P_WB_2 P_W B_1P_W)                             \\
		 & ={\rm Tr}(P_{V_1}B_n P_e \cdots P_eB_2 P_e B_1P_{V_1})                 \\
		 & =\Varphi(p_{n-1})\cdots \Varphi(p_2) {\rm Tr}(P_{V_1}B_nP_eB_1P_{V_1}) \\
		 & =\Varphi(p_{n-1})\cdots \Varphi(p_2){\rm Tr}(B_1P_{V_1}B_nP_e)         \\
		 & =\Varphi(p_{n-1})\cdots \Varphi(p_2)\Varphi(p_1p_n)
	\end{align*}
	if not.
\end{proof}
\begin{remark}
	Note that, from the definition of the conditional free product of graphs, $\OmegA(1)=|V(H_1)|+|V(H_2)|-1$, whereas $\OmegA_{H_i}(1)=|V(H_i)|$.
	In the following particular cases:
	$$ H_1\prescript{}{\emptyset}{*}^{}_{\emptyset}H_2\simeq H_1\star H_2\ \ \text{and}\ \ H_1\prescript{}{\emptyset}{*}^{}_{H_2}H_2\simeq H_1\rhd H_2,$$
	we recover cyclic Boolean  independence  and cyclic monotone independence between $A^{(1)}$ and $A^{(2)}$. In the case $$H_1\prescript{}{H_1}{*}^{}_{H_2}H_2\simeq H_1* H_2,$$
	we obtain cyclic freeness (and not infinitesimal freeness, since the root state is not necessarily tracial).
\end{remark}

\bibliographystyle{alpha}
\bibliography{main}
\end{document}